\documentclass[11pt]{amsart}

\usepackage[centertags]{amsmath}
\usepackage{amsfonts}
\usepackage{amssymb}
\usepackage{amsthm}
\usepackage[english]{babel}
\usepackage{enumerate}
\usepackage{mathrsfs}
\RequirePackage{srcltx}
\usepackage{color}

\numberwithin{equation}{section}

\newtheorem{thm}{Theorem}[section]
\newtheorem{cor}[thm]{Corollary}
\newtheorem{lem}[thm]{Lemma}
\newtheorem{prop}[thm]{Proposition}

\newtheorem{defn}[thm]{Definition}

\def\N{{{\Bbb N}}}

\def\L{\text{Lip}}
\def\R{\mathbb{R}}
\def\T{\mathbb{T}}
\def\Z{\mathbb{Z}}

\def\a{{\alpha}}
\def\){\right)}
\def\({\left(}
\def\t{{\theta }}
\def\b{{\beta}}
\def\d{{\delta}}
\newcommand{\dint}{\ensuremath{\mathrm d}}

\theoremstyle{remark}
\newtheorem{rem}[thm]{Remark}

\newcommand{\Lipx}[4]{\ensuremath {\mathrm{Lip}^{(#1,#2)}_{#3,#4}}}

\newcommand{\vertiii}[1]{{\left\vert\kern-0.25ex\left\vert\kern-0.25ex\left\vert #1
    \right\vert\kern-0.25ex\right\vert\kern-0.25ex\right\vert}}

\begin{document}

\title[Embeddings and characterizations of Lipschitz spaces]{Embeddings and characterizations of Lipschitz spaces
}
\author
{\'Oscar Dom\'inguez}

\address{O. Dom\'inguez, Departamento de An\'alisis Matem\'atico y Matem\'atica Aplicada, Facultad de Matem\'aticas, Universidad Complutense de Madrid\\
Plaza de Ciencias 3, 28040 Madrid, Spain.}
\email{oscar.dominguez@ucm.es}

\author{Dorothee D. Haroske}
\address{D.D. Haroske, Institute of Mathematics, Friedrich Schiller University Jena\\ Ernst-Abbe-Platz 2, 07743 Jena, Germany.}
\email{dorothee.haroske@uni-jena.de}

\author{Sergey Tikhonov}
 \address{S. Tikhonov, Centre de Recerca Matem\`{a}tica\\
Campus de Bellaterra, Edifici C
08193 Bellaterra (Barcelona), Spain;
ICREA, Pg. Llu\'{i}s Companys 23, 08010 Barcelona, Spain,
 and Universitat Aut\`{o}noma de Barcelona.}
\email{ stikhonov@crm.cat}

\thanks{The first author was partially supported by MTM 2017-84058-P.
The third author was partially supported by
 MTM 2017-87409-P,  2017 SGR 358, and
 the CERCA Programme of the Generalitat de Catalunya. Part of this work was done during the visit of the authors to the Isaac Newton Institute for Mathematical Sciences, Cambridge, EPSCR Grant no EP/K032208/1.}

\maketitle

\bigskip
\begin{abstract}
In this paper we give  a thorough study of Lipschitz  spaces.
We obtain  the following new results:
\begin{enumerate}
	\item Sharp  Jawerth-Franke-type embeddings between the Besov and Lipschitz spaces extending the classical results for Besov and Sobolev spaces;
	\item Sharp embeddings between  Lipschitz spaces with different parameters extending the Br\'{e}zis-Wainger result;
	\item Characterizations for Lipschitz norms via Fourier transforms and  wavelets;
	\item Sharp embeddings from Lipschitz spaces into Lebesgue/Lorentz--Zygmund spaces.
\end{enumerate}
 \vskip 1.6cm
\end{abstract}

\tableofcontents

\section{Introduction}
The classical H\"{o}lder condition
\begin{equation}\label{Intro1}
\|f(\cdot+h)-f\|_{L_p(\R^d)}\le C|h|^\alpha
\end{equation}
is known to be extremely useful in a variety of questions in analysis and the PDE's.
Embedding properties of the corresponding function space for a non-limiting parameter $0<\alpha<1$ -- the Lipschitz/Besov space --
are  well known and can be found in many textbooks (see, e.g., \cite{Triebel78, Triebel83, Triebel01}).
In the limiting case $\alpha = 1$ the  Lipschitz space has been much less investigated (see \cite{DeVoreLorentz}). Moreover, there is a need to study not only the classical Lipschitz spaces but also the logarithmic Lipschitz
spaces, which are obtained by introducing an additional logarithmic majorant in \eqref{Intro1}. The latter are motivated in part by the celebrated Br\'{e}zis-Wainger inequality \cite{BrezisWainger}, which claims that every function $f$ in the Sobolev space $H^{1+d/p}_p(\R^d), 1 < p < \infty$, is ``almost" Lipschitz-continuous in the sense that
\begin{equation*}
	|f(x+h) - f(x)| \leq C |h| |\log |h||^{1-1/p} \|f\|_{H^{1+d/p}_p(\R^d)}
\end{equation*}
for all $x \in \R^d$ and $0 < |h| < 1/2$. Taking into account that Sobolev spaces are particular cases of the Lipschitz spaces, namely $H^\alpha_p(\R^d) = \L^{(\alpha,0)}_{p,\infty}(\R^d), \, \alpha > 0,$ the previous inequality can be interpreted in terms of the embeddings
\begin{equation}\label{IntroBW}
	 \L^{(1+d/p,0)}_{p,\infty}(\R^d) \hookrightarrow \text{Lip}^{(1,-1+1/p)}_{\infty, \infty}(\R^d), \quad 1 < p < \infty,
\end{equation}
(the precise definitions of function spaces 
  will be given in Section 2). This embedding has been extensively studied in the framework of Triebel-Lizorkin and Besov spaces \cite{EdmundsHaroske, EdmundsHaroske00}, paving the way for the nowadays well-established theory of envelopes of function spaces (see \cite{Triebel01, Haroske} and the references given there).

Besides their intrinsic interest, embedding theorems for function spaces of logarithmic Lipschitz type are proving useful in the theory of differential equations. For instance, in virtue of the well-known Osgood condition, the embedding \eqref{IntroBW}, as well as its extensions to vector fields \cite{Zuazua}, is the key result to establish well-posedness of solutions of initial value problems in critical cases. On the other hand, as can be seen in \cite[Section 2.11]{BahouriCheminDanchin}, \cite[Proposition 3.5]{Colombinia} \cite[Lemma 3.13]{Colombini}, \cite[Proposition 3.3]{ColombiniLerner} (see also \cite{MisiolekYoneda}), the embeddings
\begin{equation}\label{IntroBCD}
	B^1_{\infty, \infty}(\R^d) \hookrightarrow \L^{(1,-1)}_{\infty,\infty}(\R^d) \hookrightarrow B^{1,-1}_{\infty, \infty}(\R^d)
\end{equation}
and
\begin{equation}\label{IntroBCD*}
	B^{1,-1}_{\infty, \infty}(\R^d) \hookrightarrow \L^{(1,-2)}_{\infty,\infty}(\R^d)
\end{equation}
play an important role in the analysis of the incompressible Euler equations and hyperbolic operators. In particular, the left-hand side embedding in \eqref{IntroBCD} sharpens the widely used criterion in the PDE's that $f \in \L^{(1,-1)}_{\infty,\infty}(\R^d)$ whenever $\nabla f$ is a function of bounded mean oscillation (see \cite[Theorem 3]{Cianchi}).

Let us also mention some important problems in PDE's which are formulated in terms of the integral log-Lipschitz spaces $\L^{(\alpha,-b)}_{p,q}$ (cf. \cite{ColombiniGiorgi, ColombiniLerner, FanelliZuazua, OlsonRobinson}). In this direction, Fanelli and Zuazua \cite{FanelliZuazua} studied the control problem in wave equations under lower smoothness assumptions. More precisely, they show the fulfillment of observability estimates in the classical sense for coefficients in the Zygmund class $B^{1}_{1,\infty}$ (in their notation $\mathcal{Z}_1$) and with a finite loss of derivatives for coefficients in $\L^{(1,-1)}_{1,\infty}$ and $ B^{1, -1}_{1,\infty}$ (in their notation, the spaces $\mathcal{L L}_1$ and $\mathcal{L Z}_1$, respectively). This extends the well-known observability result for $\text{BV}$ coefficients \cite{FernandezZuazua}. The precise interrelations between these function spaces will be given later.

Among many other applications of embeddings between Lipschitz spaces 
 we highlight  the Ulyanov-type inequalities for moduli of smoothness (see \cite{DominguezTikhonov19}).
This topic is a cornerstone in the study of various problems in approximation theory \cite{DitzianTikhonov, KolomoitsevTikhonov} and functional analysis \cite{KolyadaNafsa}. 

Our main goal in this paper is to present  the detailed study of embedding properties of the Lipschitz
spaces to fill in the gap in the existing literature. Our results
can be naturally divided into four parts.

In the first part we study embeddings between Lipschitz and Besov spaces. In particular, in Section 4 we extend \eqref{IntroBCD} and \eqref{IntroBCD*} to the full range of parameters. Specifically, if $\alpha > 0, 1 < p < \infty, 0 < q \leq \infty$, and $b > 1/q$, then
\begin{equation}\label{Intro2}
	B^{\alpha, -b+1/\min\{2,p,q\}}_{p,q}(\R^d) \hookrightarrow \text{Lip}^{(\alpha,-b)}_{p,q}(\R^d) \hookrightarrow B^{\alpha, -b+1/\max\{2,p,q\}}_{p,q}(\R^d)
\end{equation}
and
\begin{equation}\label{Intro3}
	B^{\alpha, -b+1/q}_{p,\min\{2, p, q\}}(\R^d) \hookrightarrow \text{Lip}^{(\alpha,-b)}_{p,q}(\R^d) \hookrightarrow B^{\alpha, -b+1/q}_{p,\max\{2,p,q\}}(\R^d).
\end{equation}
It is worthwhile to mention  that \eqref{Intro2} and \eqref{Intro3} extend the classical Besov--Sobolev embedding (see, e.g., \cite[Section 2.3.2]{Triebel83})
\begin{equation}\label{Intro4}
	B^\alpha_{p, \min\{2,p\}}(\R^d) \hookrightarrow H^\alpha_p(\R^d) 
   \hookrightarrow B^\alpha_{p, \max\{2,p\}}(\R^d).
\end{equation}
We also investigate the limiting cases $p=1, \infty$. In particular, we obtain the following embeddings
\begin{equation*}
	\text{BV}(\R^d) \hookrightarrow B^{1}_{1,\infty}(\R^d) \hookrightarrow \L^{(1,-1)}_{1,\infty}(\R^d) \hookrightarrow B^{1, -1}_{1,\infty}(\R^d).
\end{equation*}
These embeddings provide the connections between observability estimates for wave equations and smoothness conditions of their coefficients in the Fanelli and Zuazua theorem discussed above.

Moreover, we derive new embeddings of Jawerth-Franke type for Lipschitz spaces, namely,
if $\alpha > 0, 1 \leq p_0 < p < p_1 \leq \infty, 0 < q\leq \infty,$ and $b > 1/q$, we have
\begin{equation}\label{Intro6}
			B^{\alpha + d(1/p_0 - 1/p), -b + 1/\min\{p,q\}}_{p_0,q}(\R^d) \hookrightarrow \text{Lip}^{(\alpha,-b)}_{p,q}(\R^d) \hookrightarrow B^{\alpha + d(1/p_1 -1/p), -b + 1/\max\{p,q\}}_{p_1,q}(\R^d)
\end{equation}
		and
\begin{equation}\label{Intro7}
			B^{\alpha + d(1/p_0 - 1/p), -b + 1/q}_{p_0,\min\{p,q\}}(\R^d) \hookrightarrow \text{Lip}^{(\alpha,-b)}_{p,q}(\R^d) \hookrightarrow B^{\alpha + d(1/p_1 -1/p), -b + 1/q}_{p_1,\max\{p,q\}}(\R^d).
\end{equation}
These extend the long-established  Jawerth-Franke result for Sobolev spaces \cite{Franke, Jawerth} (cf. also \cite{Marschall} and \cite{Vybiral}) given by
\begin{equation}\label{Intro8}
		B^{\alpha + d(1/p_0 -1/p)}_{p_0,p}(\R^d) \hookrightarrow H^{\alpha}_p(\R^d) 
  \hookrightarrow B^{\alpha + d(1/p_1 - 1/p)}_{p_1,p}(\R^d).
\end{equation}

Before proceeding further, some comments are in order. The embeddings \eqref{Intro2}, \eqref{Intro3}, \eqref{Intro6} and \eqref{Intro7} reveal the importance of the parameters $b$ and $q$ in $\L^{(\alpha,-b)}_{p,q}(\R^d)$.
In particular, it is interesting to compare the
embeddings 
\begin{equation}\label{vspomsl}
\begin{split}
	H^\alpha_p(\R^d)=\text{Lip}^{(\alpha,0)}_{p,\infty}(\R^d)   &\hookrightarrow B^\alpha_{p, \max\{2,p\}}(\R^d),
\\
\text{Lip}^{(\alpha,-b)}_{p,\infty}(\R^d) &\hookrightarrow B^{\alpha, -b}_{p,\infty}(\R^d), \quad b > 0,
\end{split}
\end{equation}
cf. \eqref{Intro4} and
\eqref{Intro3} correspondingly.
Note that 
 both embeddings in (\ref{vspomsl}) are sharp, which shows substantial differences between embedding theorems for the spaces $\L^{(\alpha,0)}_{p,\infty}(\R^d)= H^\alpha_p(\R^d)$ and those for $\L^{(\alpha,-b)}_{p,q}(\R^d), \, b > 1/q$.
The same comment also applies to \eqref{Intro7} and \eqref{Intro8}.

Our second aim is to fully characterize the embeddings between various Lipschitz spaces, i.e.,
$$\text{Lip}^{(\alpha_0,-b_0)}_{p_0,q_0}(\R^d) \hookrightarrow \text{Lip}^{(\alpha_1,-b_1)}_{p_1,q_1}(\R^d).
$$
This is addressed in Section 5.
Among various other results,
we obtain the Sobolev type embedding with the parameters
 $1 < p_0 < p_1 < \infty$ such that $\alpha_0 -d/p_0 = \alpha_1 -d/p_1$, and $0 < q_0=q_1 \leq \infty$, $b_0=b_1 > 1/q$. More interestingly, we also deal with the extreme cases $p_0=1$ and/or $p_1=\infty$. In particular, setting $p_1=\infty$, we extend the Br\'ezis-Wainger embedding \eqref{IntroBW}
showing that
	$$
		\text{Lip}^{(\alpha + d/p,-b)}_{p,q}(\R^d) \hookrightarrow \text{Lip}^{(\alpha, -b-1 +1/p)}_{\infty, q}(\R^d), \quad 1 < p < \infty,
	$$
and
$$		\text{Lip}^{(k + d,-b)}_{1,q}(\R^d) \hookrightarrow \text{Lip}^{(k, -b)}_{\infty, q}(\R^d), \qquad k \in \N.
$$

Next, our third goal is to give a characterization of the Lipschitz spaces via Fourier transforms (Section 6) and wavelets (Section 7).
We show that the truncated square function
 \begin{equation*}
 	 \Big(\sum_{j=0}^k 2^{j \beta 2} |(\varphi_j \widehat{f})^\vee (x)|^2 \Big)^{1/2}, \quad x \in \R^d, \quad k \geq 0,
 \end{equation*}
can be used to develop a unified treatment of function spaces which includes Besov, Sobolev, Lebesgue and Lipschitz spaces. To be more precise, we prove that the family of norms given by
	\begin{equation*}
		\left(\sum_{k=0}^\infty 2^{k(\alpha-\beta) q} (1 + k)^{- b q} \Big\| \Big(\sum_{j=0}^k 2^{j \beta 2} |(\varphi_j \widehat{f})^\vee (\cdot)|^2 \Big)^{1/2} \Big\|_{L_p(\R^d)}^q \right)^{1/q}, \quad \beta \geq \alpha,
	\end{equation*}
 allows us to characterize Besov spaces ($\beta> \alpha$), Lipschitz spaces ($\beta = \alpha$), Lebesgue spaces ($\beta = \alpha = 0, q= \infty$ and $b=0$), and Sobolev spaces ($\beta = \alpha, q= \infty$ and $b=0$). Furthermore, a similar result holds when the Fourier decompositions of $f$ are replaced by wavelet decompositions.

Later, in Section 8 we give the characterizations of Lipschitz spaces for two important classes of functions: trigonometric series with lacunary Fourier coefficients and functions with monotone-type Fourier transform. These characterizations provide us simple criteria to determine whether a given function belongs to a Lipschitz space. Consequently, we are able to obtain sharpness statements for all the mentioned above embeddings in the full range of parameters. These sharpness results will be given in Section 9.

Finally, our last aim in this paper is to derive new sharp embeddings
of the following type:
$\text{Lip}^{(\alpha,-b)}_{p,q}(\R^d) \hookrightarrow X$, where $X$ is
a  Lebesgue space, or more generally, a Lorentz--Zygmund space.
See Section 10.
As an example of an embedding of this type we mention the remarkable Cheeger's estimate and concentration inequalities (cf. \cite{aida}, \cite{Ledoux} and \cite{aida1}).

Concluding the introduction, 
 we briefly present one of the novel ideas
  to investigate 
   the structure of the Lipschitz spaces. Using  the fact that Lipschitz spaces can be characterized as limiting interpolation spaces between classical Lebesgue and Sobolev spaces, i.e.,
\begin{equation}\label{odin}
	\L^{(\alpha,-b)}_{p,q}(\R^d) = (L_p(\R^d), H^\alpha_p(\R^d))_{(1,-b),q},
\end{equation}
we  employ the machinery of limiting interpolation with $\theta = 1$  (see \cite{CobosFernandezCabreraKuhnUllrich, EvansOpic, EvansOpicPick, GogatishviliOpicTrebels}) to transfer properties from $H^\alpha_p(\R^d)$ to $\L^{(\alpha,-b)}_{p,q}(\R^d)$.

It is worth mentioning  that a similar approach with $\theta =0$ was  applied in the study of Besov spaces with smoothness close to zero (cf. \cite{CobosDominguezTriebel, DominguezTikhonov}): \begin{equation}\label{dva}\mathbb{B}^{0,b}_{p,q}(\R^d)= (L_p(\R^d), H^\alpha_p(\R^d))_{(0,b),q},\end{equation}
where $\mathbb{B}^{0,b}_{p,q}(\R^d)$ is the Besov space with the logarithmic smoothness $b$.
However, there is a fundamental difference in applications of (\ref{odin}) and (\ref{dva})
 due to the fact that Lipschitz and Besov spaces have a completely different structure.
  For instance, a crucial step in the proofs of the wavelet/Fourier-analytical descriptions of the spaces $\mathbb{B}^{0,b}_{p,q}(\R^d)$ (cf. \cite[Theorems 4.3 and 5.5]{CobosDominguezTriebel}) is the fact that these spaces can be characterized via approximation processes. 
  Such characterization  fails to be true for Lipschitz spaces and this forces us to apply  more sophisticated limiting interpolation arguments.
  
  We collect the necessary background in interpolation theory in Section 2 and give some auxiliary interpolation results in Section 3. Some of them are interesting in themselves. For instance, we show that the scale of Lipschitz spaces with fixed integrability is stable under interpolation (see Lemma \ref{LemmaInterpolationLipschitz}). Furthermore, our interpolation-based method also works with Lipschitz spaces formed by periodic functions.

\section{Preliminaries}

\subsection{Function spaces}

Let $\mathcal{S}(\R^d)$ and $\mathcal{S}'(\R^d)$ be the Schwartz space of all complex-valued rapidly decreasing  infinitely differentiable functions on $\R^d$, and the space of tempered distributions on $\R^d$, respectively. The Fourier transform of $f \in \mathcal{S}(\R^d)$ is defined by
\begin{equation*}
	\widehat{f}(\xi) = (2 \pi)^{-d/2} \int_{\R^d} f(x) e^{-i x \cdot \xi} \, \dint x, \quad \xi \in \R^d,
\end{equation*}
where $x \cdot \xi = \sum_{j=1}^d x_j \xi_j$, and the inverse Fourier transform is given by
\begin{equation*}
	f^\vee(\xi) = (2 \pi)^{-d/2} \int_{\R^d} f(x) e^{i x \cdot \xi} \, \dint x.
\end{equation*}
These operators are extended to $\mathcal{S}'(\R^d)$ in the usual way.

Let $-\infty < \alpha < \infty$ and $1 < p < \infty$. The \emph{(fractional) Sobolev space} (also called, \emph{Bessel potential space}) $H^\alpha_p(\R^d)$ is the set of all $f \in L_p(\R^d)$ such that
\begin{equation*}
	\|f\|_{H^\alpha_p(\R^d)} = \|((1 + |x|^2)^{\alpha/2} \widehat{f})^\vee\|_{L_p(\R^d)} < \infty.
\end{equation*}
Note that $H^0_p(\R^d) = L_p(\R^d)$ and $H^k_p(\R^d) = W^k_p(\R^d), \, k \in \N$, the classical Sobolev space.

Let $1 \leq p \leq \infty, 0 < q \leq \infty$ and $-\infty < s, b < \infty$. The \emph{Besov space} $B^{s,b}_{p,q}(\R^d)$ is formed by all $f \in \mathcal{S}'(\R^d)$ having a finite quasi-norm
\begin{equation}\label{DefBesov}
	\|f\|_{B^{s,b}_{p,q}(\R^d)} = \left(\sum_{j=0}^\infty (2^{j s} (1 + j)^b \|(\varphi_j \widehat{f})^\vee\|_{L_p(\R^d)})^q  \right)^{1/q}.
\end{equation}
Here, $(\varphi_j)_{j=0}^\infty$ is the usual smooth dyadic resolution of unity in $\R^d$. For further details on these spaces, we refer the reader to \cite{EdmundsTriebel98, EdmundsTriebel99, FarkasLeopold, KalyabinLizorkin, Leopold, Moura}. Note that if $b=0$ in $B^{s,b}_{p,q}(\R^d)$ then we get the classical Besov spaces $B^s_{p,q}(\R^d)$. We also mention that the quasi-norm \eqref{DefBesov} can be used to introduce $B^{s,b}_{p,q}(\R^d)$ for $0 < p \leq \infty$.

For $\alpha > 0$ and $h \in \R^d$, we let
\begin{equation*}
	\Delta^\alpha_h f (x) = \sum_{j=0}^\infty (-1)^{j} \binom{\alpha}{j} f(x + (\alpha- j) h), \quad x \in \R^d,
\end{equation*}
where $\binom{\alpha}{j} = \frac{\alpha (\alpha-1) \ldots (\alpha-j+1)}{j!}, \quad \binom{\alpha}{0} =1$. We denote by $\omega_\alpha(f,t)_p$ the \emph{modulus of smoothness of fractional order $\alpha$ of $f \in L_p(\R^d)$} which is given by
  \begin{equation*}
	\omega_\alpha(f,t)_{p} = \sup_{|h| \leq t} \|\Delta^\alpha_h f\|_{L_p(\R^d)}, \quad t > 0.
\end{equation*}
Clearly, if $\alpha \in \N$ then we recover the classical modulus of smoothness.

It is well known that Besov spaces with positive smoothness can be characterized through moduli of smoothness. See, e.g., \cite[Section 2.5.12]{Triebel83} and \cite[Theorem 2.5]{HaroskeMoura}. Assume $0 < s < \alpha$. Then,
\begin{equation}\label{CharBesovModuli}
	\|f\|_{B^{s,b}_{p,q}(\R^d)} \asymp \|f\|_{L_p(\R^d)} + \Big(\int_0^1 (t^{-s} (1-\log t)^b \omega_\alpha(f,t)_p)^q \frac{\dint t}{t} \Big)^{1/q}
\end{equation}
(if $q=\infty$, then the integral should be replaced by the appropriate supremum).

Let $\alpha > 0, 1 \leq p \leq \infty, 0 < q \leq \infty,$ and $-\infty < b < \infty$. The \emph{logarithmic Lipschitz space} $\L^{(\alpha,-b)}_{p,q}(\R^d)$ is the collection of all $f \in L_p(\R^d)$ such that
\begin{equation}\label{DefLip}
	\|f\|_{\L^{(\alpha,-b)}_{p,q}(\R^d)} = \|f\|_{L_p(\R^d)} +  \left(\int_0^{1} (t^{-\alpha} (1 - \log t)^{-b} \omega_\alpha(f,t)_p)^q \frac{\dint t}{t} \right)^{1/q} < \infty
\end{equation}
(with the usual modification if $q=\infty$). We shall assume that $b > 1/q$ ($b \geq 0$ if $q=\infty$). Otherwise, $\L^{(\alpha,-b)}_{p,q}(\R^d) = \{0\}$. Setting $\alpha = 1$ we obtain the spaces $\L^{(1,-b)}_{p,q}(\R^d)$ studied in detail in \cite{Haroske}.
See also the paper \cite{Janson} by Janson for a thorough study of
(generalized) Lipschitz spaces. The reader is advised to take care that the spaces $\L^{(\alpha,-b)}_{\infty,q}(\R^d), \, 0 < \alpha < 1,$ endowed with \eqref{DefLip} differ from those considered by Haroske \cite{Haroske}. To be more precise, in our notation the Lipschitz-type spaces introduced in \cite[Definition 2.26]{Haroske} coincide with $B^{\alpha,-b}_{\infty,q}(\R^d)$ (see \eqref{CharBesovModuli}).

If $p=q=\infty, \alpha =1$ and $b=0$ in $\L^{(\alpha,-b)}_{p,q}(\R^d)$ then we obtain the classical Lipschitz spaces $\L(\R^d)$ formed by all $f \in L_\infty(\R^d)$ such that
\begin{equation}\label{DefClassicLip}
	\|f\|_{\L(\R^d)} = \|f\|_{L_\infty(\R^d)} + \sup_{0 < |x-y| < 1} \frac{|f(x) -f(y)|}{|x-y|}.
\end{equation}
More generally, $\L^{(k,0)}_{\infty,\infty}(\R^d) =\L^k(\R^d)$ where $\L^1(\R^d) = \L(\R^d)$ and
\begin{equation*}
		\|f\|_{\L^{k}(\R^d)} \asymp \sum_{|\beta| \leq k-1} \|D^\beta f\|_{\L(\R^d)}, \quad k \geq 2.
	\end{equation*}
See \cite[Lemma 14.5]{DominguezTikhonov}.

If $p=1, q=\infty, \alpha=1$ and $b=0$ then we recover $\text{BV}(\R^d)$, the space of integrable functions of bounded variation. More generally, we can deal with spaces of primitives of functions of bounded variation. That is, setting $\alpha=k \in \N$ we get $\text{BV}^{k-1}(\R^d)$ where
\begin{equation}\label{BV=Lip}
\text{BV}^{0}(\R^d) = \text{BV}(\R^d)=\L^{(1,0)}_{1, \infty}(\R^d),
\end{equation}
and
	\begin{equation*}
		\|f\|_{\text{BV}^{k}(\R^d)} = \sum_{|\beta| \leq k} \|D^\beta f\|_{\text{BV}(\R^d)}, \quad k \geq 1.
	\end{equation*}
	See \cite[Lemma 14.5]{DominguezTikhonov}.

Moreover, the space $H^\alpha_p(\R^d)$ is a special case of Lipschitz spaces, since
\begin{equation}\label{LipSob}
	\L^{(\alpha,0)}_{p, \infty}(\R^d) = H^\alpha_p(\R^d), \quad \alpha > 0, \quad 1 < p < \infty;
\end{equation}
see, e.g.,  \cite[Corollary 10]{Wilmes}.

Analogously, one can introduce the periodic counterparts $H^\alpha_p(\T^d), B^{s,b}_{p,q}(\T^d),$ and $ \L^{(\alpha,-b)}_{p,q}(\T^d)$.

Throughout the paper, we use the notation
$\, F \lesssim G$
with $F,G\ge 0$ for the
estimate
$\, F \le C\, G,$ where $\, C$ is a positive constant independent of
the essential variables in $\, F$ and $\, G$.  If $\, F \lesssim G \lesssim F$, we write $\, F \asymp G$ and say that $\, F$
is equivalent to $\, G$.

\subsection{Limiting interpolation}

Let $(A_0, A_1)$ be an ordered couple of quasi-Banach spaces, that is, $A_1 \hookrightarrow A_0$. The \emph{Peetre $K$-functional} is defined by
\begin{equation}\label{KFunct}
	K(t,f) = K(t, f; A_0, A_1) = \inf_{f_1 \in A_1} (\|f - f_1\|_{A_0} + t \|f_1\|_{A_1}), \quad t > 0, \quad f \in A_0.
\end{equation}

Let $0 <\theta < 1, -\infty < b < \infty$, and $0 < q \leq \infty$. The \emph{logarithmic interpolation space} $(A_0,A_1)_{\theta,q;b}$ is the set formed by all $f \in A_0$ such that
\begin{equation}\label{DefInter}
	\|f\|_{(A_0,A_1)_{\theta,q;b}} = \left(\int_0^\infty (t^{-\theta} (1 + |\log t|)^b K(t,f))^q \frac{\dint t}{t} \right)^{1/q} < \infty
\end{equation}
(appropriately modified if $q=\infty$). For further details and properties, we refer to \cite{Gustavsson}. In particular, if $b=0$ in $(A_0,A_1)_{\theta,q;b}$ then we obtain the classical real interpolation space $(A_0, A_1)_{\theta,q}$; see \cite{BennettSharpley, BerghLofstrom, Triebel78}.

The interpolation properties of Besov spaces of generalized smoothness were investigated by Cobos and Fern\'andez \cite{CobosFernandez}. For later use, we record some interpolation formulas obtained in \cite[Theorem 5.3]{CobosFernandez}. Let $1 \leq p \leq \infty, 0 < q_0, q_1, q \leq \infty, -\infty < s_0 \neq s_1 < \infty, - \infty < b_0, b_1, \alpha < \infty$ and $0 < \theta < 1$. Put $s = (1 - \theta) s_0 + \theta s_1$, and $b = (1 - \theta) b_0 + \theta b_1$. Then, we have
\begin{equation}\label{InterBes}
(B^{s_0, b_0}_{p,q_0}(\mathbb{R}^d), B^{s_1, b_1}_{p,q_1}(\mathbb{R}^d))_{\theta,q;\alpha} = B^{s,b + \alpha}_{p,q}(\mathbb{R}^d)
\end{equation}
and
\begin{equation}\label{InterBesSobLp}
	(H^{s_0}_p(\R^d), H^{s_1}_p(\R^d))_{\theta, q; \alpha} = B^{s, \alpha}_{p,q}(\R^d).
\end{equation}
In particular, if $b_0=b_1=\alpha = 0$ in (\ref{InterBes}) then
\begin{equation}\label{InterBesClass}
(B^{s_0}_{p,q_0}(\mathbb{R}^d), B^{s_1}_{p,q_1}(\mathbb{R}^d))_{\theta,q} = B^{s}_{p,q}(\mathbb{R}^d);
\end{equation}
see also \cite[Section 2.4.1, pages 181--184]{Triebel78} and \cite[Section 5, Theorem 4.17, page 343]{BennettSharpley}. On the other hand, setting $s_0 =0$ and $s_1 > 0$ in (\ref{InterBesSobLp}),
\begin{equation}\label{InterBesSobLp2}
(L_p(\R^d), H^{s_1}_p(\R^d))_{\theta, q; \alpha} = B^{\theta s_1, \alpha}_{p,q}(\R^d).
\end{equation}
For the extensions of (\ref{InterBesSobLp}) and (\ref{InterBesSobLp2}) to Triebel-Lizorkin spaces of generalized smoothness, we refer the reader to \cite{CobosFernandez}.

The corresponding formulas for periodic Besov spaces also hold true.

Since $A_1 \hookrightarrow A_0$, it is not hard to check that $K(t, f) \asymp \|f\|_{A_0}$ for $t > 1$. Consequently, we have
\begin{equation*}
	\|f\|_{(A_0,A_1)_{\theta,q;b}} \asymp \left(\int_0^1 (t^{-\theta} (1 + |\log t|)^b K(t,f))^q \frac{\dint t}{t} \right)^{1/q}.
\end{equation*}
This fact together with the finer tuning given by logarithmic weights allows us to introduce \emph{limiting interpolation spaces} with $\theta=1$. Namely, the space $(A_0, A_1)_{(1,b),q}$ is the collection of all $f \in A_0$ for which
\begin{equation}\label{Klimitspace}
	\|f\|_{(A_0,A_1)_{(1,b),q}} = \left(\int_0^1 (t^{-1} (1 + |\log t|)^b K(t,f))^q \frac{\dint t}{t} \right)^{1/q} < \infty.
\end{equation}
See \cite{EvansOpic}, \cite{EvansOpicPick}, \cite{GogatishviliOpicTrebels} and \cite{CobosFernandezCabreraKuhnUllrich}. Note that this space becomes trivial if $b \geq -1/q$ ($b > 0$ if $q=\infty$). Then, we shall assume that $b < -1/q$ ($b \leq 0$ if $q=\infty$).

Lipschitz spaces $\L^{(\alpha,-b)}_{p,q}(\R^d), 1 < p < \infty$, can be characterized as limiting interpolation spaces between $L_p(\R^d)$ and $H^\alpha_p(\R^d)$. Indeed, if $1 < p < \infty$ then
\begin{equation}\label{LipLimInter.}
	K(t^\alpha, f; L_p(\R^d), H^\alpha_p(\R^d)) \asymp t^\alpha \|f\|_{L_p(\R^d)} + \omega_\alpha(f,t)_p, \quad 0 < t <1
\end{equation}
(see \cite[(4.2)]{Wilmes}), which yields that
\begin{equation}\label{LipLimInter}
	(L_p(\R^d), H^\alpha_p(\R^d))_{(1,-b),q} = \L^{(\alpha,-b)}_{p,q}(\R^d).
\end{equation}
The corresponding formula for periodic spaces also holds true (see \cite[(21)]{Wilmes2}). Here it is important to mention that \eqref{LipLimInter.} and \eqref{LipLimInter} can be extended to cover the extreme cases $p=1, \infty$. This can be done with the help of the Sobolev-type spaces $\mathcal{H}^\alpha_p(\R^d), \alpha >0, 1 \leq p \leq \infty,$ defined by
\begin{equation*}
	\mathcal{H}^\alpha_p(\R^d) = \{f \in L_p(\R^d) : \|f\|_{\mathcal{H}^\alpha_p(\R^d) } = \|f\|_{L_p(\R^d)} +\sup_{\zeta \in \R^d, |\zeta| = 1} \|D^\alpha_\zeta f\|_{L_p(\R^d)} < \infty \}
\end{equation*}
were, $D^\alpha_\zeta f (x) = ((i \xi \cdot \zeta)^\alpha \widehat{f}(\xi))^\vee(x)$. Note that $\mathcal{H}^\alpha_p(\R^d) = H^\alpha_p(\R^d)$ if $1 < p < \infty$. It was shown in \cite[Property 13]{KolomoitsevTikhonov1} that
\begin{equation*}
	K(t^\alpha, f; L_p(\R^d), \mathcal{H}^\alpha_p(\R^d)) \asymp t^\alpha \|f\|_{L_p(\R^d)} + \omega_\alpha(f,t)_p,  \quad 1 \leq p \leq \infty,
\end{equation*}
and thus
\begin{equation}\label{LipLimInterKT}
	(L_p(\R^d), \mathcal{H}^\alpha_p(\R^d))_{(1,-b),q} = \L^{(\alpha,-b)}_{p,q}(\R^d), \quad 1 \leq p \leq \infty.
\end{equation}
The periodic counterpart of the latter formula also holds true.

Furthermore, it is well known that the corresponding results for the classical Sobolev spaces $W^k_p(\R^d)$ also hold true if $p=1, \infty$ (see \cite[Chapter 5, Theorem 4.12]{BennettSharpley}). More precisely, if $k \in \N$ and $1 \leq p \leq \infty$ then
\begin{equation*}
	K(t^k, f; L_p(\R^d), W^k_p(\R^d)) \asymp t^k \|f\|_{L_p(\R^d)} + \omega_k(f,t)_p, \quad 0 < t < 1,
\end{equation*}
and so
\begin{equation}\label{LipLimInter*}
	(L_p(\R^d), W^k_p(\R^d))_{(1,-b),q} = \L^{(k,-b)}_{p,q}(\R^d).
\end{equation}
The latter is the limiting version of the well-known interpolation formula for Besov spaces
\begin{equation}\label{BesovInter9}
	(L_p(\R^d), W^k_p(\R^d))_{\theta,q;-b} = B^{\theta k,-b}_{p,q}(\R^d).
\end{equation}

The following reiteration formulas for classical and limiting interpolation methods will be useful later. Let $0 <\theta < 1, 0 < p, q, q_0, q_1 \leq \infty, b < -1/q,$ and $b_0 + 1/q_0 < b_1 + 1/q_1 < 0$. Then,
	\begin{equation}\label{LemmaReiteration}
			(A_0, A_1)_{\theta, q; b + 1/\min\{p,q\}} \hookrightarrow (A_0, (A_0, A_1)_{\theta,p})_{(1,b),q} \hookrightarrow (A_0, A_1)_{\theta, q; b + 1/\max\{p,q\}};
		\end{equation}
		\begin{equation}\label{LemmaReiteration2}
			(A_0, (A_0, A_1)_{(1,b),q})_{\theta, p} = (A_0, A_1)_{\theta, p; \theta (b + 1/q)};
		\end{equation}
		\begin{equation}\label{LemmaReiteration3}
			((A_0, A_1)_{\theta,p}, A_1)_{(1,b),q} = (A_0, A_1)_{(1,b),q};
		\end{equation}
		\begin{equation}\label{LemmaReiteration4}
			((A_0, A_1)_{(1, b_0), q_0}, (A_0, A_1)_{(1,b_1), q_1})_{\theta, p} = (A_0, A_1)_{(1, (1-\theta)(b_0 + 1/q_0) + \theta (b_1 + 1/q_1) -1/p), p}.
		\end{equation}
		See \cite[Theorem 4.7*$^+$]{EvansOpic}, \cite[Theorems 7.1(v), 7.4* and Corollary 7.11]{EvansOpicPick} and \cite[Lemma 2.5(a)]{CobosDominguez}.

\section{Interpolation lemmas}

In this section we show some interpolation lemmas that will be useful in forthcoming considerations. Our first result concerns the limiting interpolation of vector-valued sequence spaces $\ell_p(A_j)$. Before we state it, let us recall some notation. Let $\N_0 = \N \cup \{0\}$ and $0 < p \leq \infty$. If $(\lambda_j)_{j \in \N_0}$ is a sequence of positive numbers and $(A_j)_{j \in \N_0}$ is a sequence of quasi-Banach spaces, by $\ell_p(\lambda_j A_j)$  we mean the space formed by all vector-valued sequences $a = (a_j)_{j \in \mathbb{N}_0}$ with $a_j \in A_j$ endowed with the quasi-norm
\begin{equation*}
	\|a\|_{\ell_p(\lambda_j A_j)} = \left(\sum_{j=0}^\infty \lambda_j^p \|a_j\|_{A_j}^p \right)^{1/p}
\end{equation*}
(where the sum should be replaced by the $\sup$ if $p=\infty$).

Working with vector-valued sequence spaces, it is convenient to introduce the so called $K_p$-functional (cf. \cite[page 75]{BerghLofstrom}). For $0 < p \leq \infty$, the $K_p$-functional is defined by
\begin{equation*}
	K_p(t,f) = K_p(t, f; A_0, A_1) = \inf_{f_1 \in A_1} (\|f - f_1\|^p_{A_0} + t^p \|f_1\|^p_{A_1})^{1/p}.
\end{equation*}
Note that if $p=1$ then we obtain the usual $K$-functional (\ref{KFunct}). In general, we have
\begin{equation}\label{KpFunct}
	K_p(t,f) \asymp K(t, f)
\end{equation}
where the equivalence constants only depend on $p$.

\begin{lem}\label{LemLimIntVectSeq}
	Let $0 < p, q \leq \infty, b < -1/q \, (b \leq 0 \text{ if } q = \infty)$ and let $(A_j)_{j \in \N_0}, (B_j)_{j \in \N_0}$ be sequences of quasi-Banach spaces such that $B_j \hookrightarrow A_j$ uniformly with respect to $j$. Then,
	\begin{equation}\label{LemLimIntVectSeq1}
		\ell_{\min\{p,q\}}((A_j, B_j)_{(1,b),q}) \hookrightarrow (\ell_p(A_j), \ell_p(B_j))_{(1,b),q} \hookrightarrow \ell_{\max\{p,q\}}((A_j, B_j)_{(1,b),q}).
	\end{equation}
\end{lem}
\begin{rem}
	The assumption $\sup_{j \in \N_0} \sup_{\|a\|_{B_j} \leq 1} \|a\|_{A_j} < \infty$ ensures that $\ell_p(B_j) \hookrightarrow \ell_p(A_j)$.
\end{rem}

\begin{proof}[Proof of Lemma \ref{LemLimIntVectSeq}]
  Let $a = (a_j)_{j \in \N_0}$. Assume, for convenience, $p<\infty$, the case $p=\infty$ can be done similarly. Elementary computations lead to
\begin{equation*}
	K_p(t, a; \ell_p(A_j), \ell_p(B_j)) = \left(\sum_{j=0}^\infty K_p(t, a_j; A_j, B_j)^p \right)^{1/p}.
\end{equation*}
Then, by (\ref{Klimitspace}) and (\ref{KpFunct}),
\begin{equation*}
	\|a\|_{(\ell_p(A_j), \ell_p(B_j))_{(1,b),q} } \asymp \left(\int_0^1 \left(\sum_{j=0}^\infty (t^{-1} (1 - \log t)^{b }K_p(t, a_j; A_j, B_j))^p \right)^{q/p} \frac{\dint t}{t} \right)^{1/q}.
\end{equation*}

Suppose now that $q \geq p$. Then, applying Minkowski's inequality, we have
\begin{align*}
	\|a\|_{(\ell_p(A_j), \ell_p(B_j))_{(1,b),q} } & \lesssim \left(\sum_{j=0}^\infty \left(\int_0^1 t^{-q} (1 - \log t)^{b q} K_p(t, a_j ; A_j, B_j)^q \frac{\dint t}{t} \right)^{p/q} \right)^{1/p} \\
	& \asymp \left(\sum_{j=0}^\infty \|a_j\|_{(A_j, B_j)_{(1,b),q}}^p \right)^{1/p} = \|a\|_{\ell_p((A_j, B_j)_{(1,b),q})}.
\end{align*}

On the other hand, if $q < p$ then
\begin{align*}
	\|a\|_{(\ell_p(A_j), \ell_p(B_j))_{(1,b),q} } & \lesssim \left(\int_0^1 \sum_{j=0}^\infty (t^{-1} (1 - \log t)^{b }K_p(t, a_j; A_j, B_j))^q \frac{\dint t}{t} \right)^{1/q} \\
	& \asymp \left(\sum_{j=0}^\infty \|a_j\|_{(A_j, B_j)_{(1,b),q}}^q \right)^{1/q} = \|a\|_{\ell_q((A_j, B_j)_{(1,b),q})}.
\end{align*}
This finishes the proof of the left-hand side embedding in (\ref{LemLimIntVectSeq1}). The proof of the right-hand side embedding is similar. Further details are left to the reader.

\end{proof}

As usual, given $\lambda > 0$, we denote by $\lambda A$ the space $A$ equipped with the quasi-norm
\begin{equation*}
	\|a\|_{\lambda A} = \lambda \|a\|_A, \quad a \in A.
\end{equation*}

\begin{lem}\label{LemLimIntVectSeq2}
	Let $\lambda > 0$ and let $(A_0,A_1)$ be a couple of quasi-Banach spaces with $A_1 \hookrightarrow A_0$. If $0 < q \leq \infty$ and $b < -1/q \, (b \leq 0 \text{ if } q=\infty)$ then
	\begin{equation}\label{LemLimIntVectSeq3}
		(\lambda A_0, \lambda A_1)_{(1,b), q} = \lambda (A_0, A_1)_{(1,-b),q}
	\end{equation}
	with equivalence constants which are independent of $\lambda$.
\end{lem}
\begin{proof}
	Clearly, $K(t, a ; \lambda A_0, \lambda A_1) = \lambda K(t, a; A_0, A_1)$. Then, (\ref{LemLimIntVectSeq3}) follows.
\end{proof}

\begin{lem}\label{LemLimIntVectSeq4}
	Let $\lambda \geq 1$ and let $A$ be a quasi-Banach space. If $0 < q \leq \infty$ and $b < -1/q \, (b \leq 0 \text{ if } q=\infty)$ then
	\begin{equation}\label{LemLimIntVectSeq5}
		(A, \lambda A)_{(1,b),q} = \lambda (1 + \log \lambda)^{b + 1/q} A
	\end{equation}
	with equivalence constants which are independent of $\lambda$.
\end{lem}
\begin{proof}
	Since $K(t, a; A, \lambda A) = \min\{1, t \lambda \} \|a\|_A$, we obtain
	\begin{equation*}
		\|a\|_{(A, \lambda A)_{(1,b),q}} = \left( \int_0^1 (t^{-1} (1 - \log t)^b \min\{1, t \lambda\})^q \frac{\dint t}{t}\right)^{1/q} \|a\|_A.
	\end{equation*}
	Assume $q < \infty$. We have
	\begin{align*}
		 \left( \int_0^1 (t^{-1} (1 - \log t)^b \min\{1, t \lambda\})^q \frac{\dint t}{t}\right)^{1/q} & \\
		 &\hspace{-5cm}\asymp \left(\int_0^{1/\lambda} (1 - \log t)^{b q} \frac{\dint t}{t} \right)^{1/q} \lambda
		  + \left(\int_{1/\lambda}^\infty t^{-q} (1 - \log t)^{b q} \frac{\dint t}{t} \right)^{1/q} \\
		  & \hspace{-5cm} \asymp \lambda (1 + \log \lambda)^{b + 1/q}
	\end{align*}
	because $b + 1/q < 0$. This gives the desired equivalence (\ref{LemLimIntVectSeq4}).
	
	The case $q=\infty$ is easier and we omit further details.
\end{proof}

Before going further, we recall the definitions of some vector-valued spaces that we shall use in the sequel.

Let $A$ be a Banach space. Let $-\infty < \alpha < \infty$ and $0 < r \leq \infty$. We denote by $\ell^\alpha_r(\N_0; A) = \ell^\alpha_r(A)$ the space formed by all sequences $x = (x_j)_{j \in \N_0}$ with $x_j \in A$ such that
\begin{equation}\label{DefVectorValuedSeqSpaces}
	\|x\|_{\ell^\alpha_r(A)} = \left(\sum_{j=0}^\infty 2^{j \alpha r} \|x_j\|_A^r \right)^{1/r} < \infty
\end{equation}
(with the usual modification if $r=\infty$).

If $1 \leq p \leq \infty$ then $L_p(\R^d;A)$ is the usual vector-valued Lebesgue space in the sense of the Bochner integral, that is, the set formed by all strongly measurable functions $f : \R^d \to A$ for which
\begin{equation*}
	\|f\|_{L_p(\R^d;A)} = \left(\int_{\R^d}  \|f(x)\|_A^p \, \dint x \right)^{1/p} < \infty.
\end{equation*}

Our next result provides a characterization of the limiting interpolation space with $\theta =1$ with respect  to the couple formed by the vector-valued Lebesgue spaces $(L_p(\R^d; \ell_r(A)), L_p(\R^d; \ell_r^\alpha(A)))$.

\begin{lem}\label{LemLimintingVectorLp}
	Let $\alpha > 0, 1 \leq p, r \leq \infty, 0 < q \leq \infty,$ and $b > 1/q \, (b \geq 0 \text{ if } q= \infty)$. Let $A$ be a Banach space. Then,
	\begin{equation}\label{LemLimintingVectorLp*}
		\|(f_j)\|_{(L_p(\R^d; \ell_r(A)), L_p(\R^d; \ell_r^\alpha(A)))_{(1,-b),q}} \asymp  \left(\sum_{k=0}^\infty (1 + k)^{- b q} \Big\| \Big(\sum_{j=0}^k 2^{j \alpha r} \|f_j (\cdot)\|_A^r \Big)^{1/r} \Big\|_{L_p(\R^d)}^q \right)^{1/q}.
	\end{equation}
	The corresponding result for periodic functions also holds true.
\end{lem}
\begin{proof}
	We start by computing the $K$-functional for the couple $(L_p(\R^d; \ell_r(A)), L_p(\R^d; \ell_r^\alpha(A)))$. Let $k \in \N_0$. We claim that
	\begin{align}
		K (2^{-k \alpha}, (f_j) ; L_p(\R^d; \ell_r(A)), L_p(\R^d; \ell_r^\alpha(A))) \nonumber\\
		& \hspace{-5cm}\asymp  \Big\| \Big(\sum_{j=k}^\infty  \|f_j (\cdot)\|_A^r \Big)^{1/r} \Big\|_{L_p(\R^d)} +  2^{-k \alpha} \, \Big\| \Big(\sum_{j=0}^k 2^{j \alpha r} \|f_j (\cdot)\|_A^r \Big)^{1/r} \Big\|_{L_p(\R^d)}. \label{LemLimintingVectorLp1}
	\end{align}
	Indeed, straightforward computations show that
	\begin{equation*}
		K(t, f ; L_p(\R^d; A_0), L_p(\R^d; A_1)) \asymp \left(\int_{\R^d} K(t, f(x) ; A_0, A_1)^p \, \dint x \right)^{1/p}
	\end{equation*}
	for any Banach couple $(A_0, A_1)$ and
	\begin{equation*}
		K(t, (x_j) ; \ell_r(A), \ell_r^\alpha(A)) \asymp \left(\sum_{j=0}^\infty (\min\{1, t 2^{j \alpha}\} \|x_j\|_A)^r \right)^{1/r},
	\end{equation*}
which implies \eqref{LemLimintingVectorLp1}.
	It follows from (\ref{LemLimintingVectorLp1}) that
	\begin{align*}
		\|(f_j)\|_{(L_p(\R^d; \ell_r(A)), L_p(\R^d; \ell_r^\alpha(A)))_{(1,-b),q}}\\
		& \hspace{-4cm} \asymp \left(\sum_{k=0}^\infty 2^{k \alpha q} (1 + k)^{-b q} K(2^{-k \alpha}, (f_j); L_p(\R^d; \ell_r(A)), L_p(\R^d; \ell_r^\alpha(A)))^q \right)^{1/q} \\
		& \hspace{-4cm} \asymp \left(\sum_{k=0}^\infty 2^{k \alpha q} (1 + k)^{-b q} \Big\| \Big(\sum_{j=k}^\infty  \|f_j (\cdot)\|_A^r \Big)^{1/r} \Big\|_{L_p(\R^d)}^q \right)^{1/q} \\
		& \hspace{-3cm} + \left(\sum_{k=0}^\infty   (1 + k)^{-b q}  \, \Big\| \Big(\sum_{j=0}^k 2^{j \alpha r} \|f_j (\cdot)\|_A^r \Big)^{1/r} \Big\|_{L_p(\R^d)}^q  \right)^{1/q}\\
		& \hspace{-4cm} = I + II.
	\end{align*}
	Hence, to get \eqref{LemLimintingVectorLp*} it will be enough to show that $I \lesssim II$. Let us distinguish two possible cases.
	
	\textsc{Case 1:} Assume $p \geq r$. For $k \in \N_0$, we apply Minkowski's inequality to obtain
	\begin{equation*}
		 \Big\| \Big(\sum_{j=k}^\infty  \|f_j (\cdot)\|_A^r \Big)^{1/r} \Big\|_{L_p(\R^d)} \leq \left(\sum_{j=k}^\infty \|f_j\|_{L_p(\R^d;A)}^r \right)^{1/r}
	\end{equation*}
	and so,
	\begin{equation}\label{LemLimintingVectorLp2}
		I \leq \left(\sum_{k=0}^\infty 2^{k \alpha q} (1 + k)^{-b q} \left(\sum_{j=k}^\infty \|f_j\|_{L_p(\R^d;A)}^r \right)^{q/r} \right)^{1/q}.
	\end{equation}
	If $q \geq r$ then we apply Hardy's inequality to get
	\begin{equation*}
		I \lesssim \left(\sum_{k=0}^\infty 2^{k \alpha q} (1 + k)^{-b q} \|f_k\|_{L_p(\R^d;A)}^q \right)^{1/q} \leq II.
	\end{equation*}
	
	Assume now that $q < r$. Then, by \eqref{LemLimintingVectorLp2} and changing the order of summation, we have
	\begin{align*}
		I &\leq \left(\sum_{k=0}^\infty 2^{k \alpha q} (1 + k)^{-b q} \sum_{j=k}^\infty \|f_j\|_{L_p(\R^d;A)}^q \right)^{1/q} \\
		& \asymp \left(\sum_{j=0}^\infty 2^{j \alpha q} (1 + j)^{-b q} \|f_j\|_{L_p(\R^d;A)}^q \right)^{1/q} \leq II.
	\end{align*}
	
	\textsc{Case 2:} Assume $p < r$. Then,
	\begin{equation*}
		I \leq \left(\sum_{k=0}^\infty 2^{k \alpha q} (1 + k)^{-b q} \left(\sum_{j=k}^\infty \|f_j\|_{L_p(\R^d;A)}^p \right)^{q/p} \right)^{1/q}
	\end{equation*}
	Let $q \geq p$. Invoking Hardy's inequality,
	\begin{equation*}
	I \lesssim \left(\sum_{k=0}^\infty 2^{k \alpha q} (1 + k)^{-b q} \|f_k\|_{L_p(\R^d;A)}^q \right)^{1/q} \leq II.
	\end{equation*}
	On the other hand, if $q < p$ then
		\begin{equation*}
		I \leq \left(\sum_{k=0}^\infty 2^{k \alpha q} (1 + k)^{-b q} \sum_{j=k}^\infty \|f_j\|_{L_p(\R^d;A)}^q \right)^{1/q} \lesssim II.
	\end{equation*}
	
\end{proof}

The following result, which is interesting by its own sake, shows that the scale of Lipschitz spaces is closed under real interpolation. Namely, we obtain the following

\begin{lem}\label{LemmaInterpolationLipschitz}
	Let $\alpha > 0, 1 \leq p \leq \infty, 0 < q, q_0, q_1 \leq \infty, 0 < \theta < 1$, and $b_0 -1/q_0 > b_1 - 1/q_1 > 0$. Then,
	\begin{equation*}
		(\emph{Lip}^{(\alpha,-b_0)}_{p,q_0}(\R^d), \emph{Lip}^{(\alpha,-b_1)}_{p,q_1}(\R^d))_{\theta,q} = \emph{Lip}^{(\alpha, -b)}_{p,q}(\R^d)
	\end{equation*}
	where $b -1/q= (1-\theta)(b_0 -1/q_0) + \theta (b_1 -1/q_1)$.
	
	The corresponding result for periodic spaces also holds true.
\end{lem}

\begin{proof}
	By \eqref{LipLimInterKT},
	\begin{equation*}
		\L^{(\alpha,-b_i)}_{p,q_i}(\R^d) = (L_p(\R^d), \mathcal{H}^\alpha_p(\R^d))_{(1,-b_i),q_i}, \quad i =0,1.
	\end{equation*}
	Applying now \eqref{LemmaReiteration4}, we have
	\begin{align*}
		(\L^{(\alpha,-b_0)}_{p,q_0}(\R^d), \L^{(\alpha,-b_1)}_{p,q_1}(\R^d))_{\theta,q} &  \\
		& \hspace{-4cm}= ((L_p(\R^d), \mathcal{H}^\alpha_p(\R^d))_{(1,-b_0),q_0}, (L_p(\R^d), \mathcal{H}^\alpha_p(\R^d))_{(1,-b_1),q_1})_{\theta,q} \\
		& \hspace{-4cm} = (L_p(\R^d), \mathcal{H}^\alpha_p(\R^d))_{(1, -((1-\theta)(b_0 - 1/q_0) + \theta (b_1 - 1/q_1) + 1/q)), q} \\
		& \hspace{-4cm} = \L^{(\alpha, -b)}_{p,q}(\R^d),
	\end{align*}
	where we have used \eqref{LipLimInterKT} in the last step.
	
\end{proof}

\section{Interrelations between Lipschitz and Besov spaces}

The goal of this section is to show that Lipschitz spaces are closely related to Besov spaces. Specifically, we are interested in the relationships between the spaces $\L^{(\alpha,-b)}_{p,q}(\R^d)$ and $B^{\alpha_0, -c}_{p_0,r}(\R^d)$ with constant differential dimension, that is, $\alpha-d/p = \alpha_0 - d/p_0$. Firstly, we shall study the case $p=p_0$ (or equivalently, $\alpha = \alpha_0$), i.e., embeddings between $\L^{(\alpha,-b)}_{p,q}(\R^d)$ and $B^{\alpha, -c}_{p,r}(\R^d)$ (see Theorem \ref{ThmBL}). Secondly, we shall deal with Jawerth-Franke embeddings, that is, embeddings between $\L^{(\alpha,-b)}_{p,q}(\R^d)$ and $B^{\alpha_0, -c}_{p_0,r}(\R^d)$ with $\alpha-d/p = \alpha_0 - d/p_0, p \neq p_0$ (see Theorem \ref{ThMFrankeLip}).

\subsection{Embeddings between Lipschitz and Besov spaces with fixed integrability}

It is well known that
\begin{equation}\label{ClassBH}
	B^\alpha_{p, \min\{2,p\}}(\R^d) \hookrightarrow H^\alpha_p(\R^d) \hookrightarrow B^\alpha_{p, \max\{2,p\}}(\R^d), \quad  - \infty < \alpha < \infty, \quad 1 < p < \infty;
\end{equation}
see, e.g., \cite[Section 2.3.2]{Triebel83}. Note that these embeddings can be rewritten in terms of Lipschitz  spaces as
\begin{equation}\label{ClassBH*}
	B^\alpha_{p, \min\{2,p\}}(\R^d) \hookrightarrow \L^{(\alpha,0)}_{p,\infty}(\R^d) \hookrightarrow B^\alpha_{p, \max\{2,p\}}(\R^d), \quad \alpha > 0
\end{equation}
(see (\ref{LipSob})).

Our next result establishes the counterparts of (\ref{ClassBH*}) for the Lipschitz spaces $\L^{(\alpha,-b)}_{p,q}(\R^d), \, b > 1/q$.

\begin{thm}\label{ThmBL}
Let $\alpha > 0, 1 < p < \infty, 0 < q \leq \infty$, and $b > 1/q$. Then, we have
\begin{equation}\label{ThmBL1}
	B^{\alpha, -b+1/\min\{2,p,q\}}_{p,q}(\R^d) \hookrightarrow \emph{Lip}^{(\alpha,-b)}_{p,q}(\R^d) \hookrightarrow B^{\alpha, -b+1/\max\{2,p,q\}}_{p,q}(\R^d)
\end{equation}
and
\begin{equation}\label{ThmBL2}
	B^{\alpha, -b+1/q}_{p,\min\{2, p, q\}}(\R^d) \hookrightarrow \emph{Lip}^{(\alpha,-b)}_{p,q}(\R^d) \hookrightarrow B^{\alpha, -b+1/q}_{p,\max\{2,p,q\}}(\R^d).
\end{equation}
In particular,
\begin{equation*}
	\emph{Lip}^{(\alpha,-b)}_{2,2}(\R^d)  = B^{\alpha, -b + 1/2}_{2,2}(\R^d), \quad b > 1/2.
\end{equation*}
The corresponding embeddings for periodic spaces also hold true.
\end{thm}

\begin{rem}\
	\begin{enumerate}[{\upshape(i)}]
	\item The left-hand side embeddings of (\ref{ThmBL1}) and (\ref{ThmBL2}) coincide if $q=\min\{2,p,q\}$. However, in general, these embeddings are independent of each other. For instance, assume that either $q > 2$ or $q > p$. Then, we have
	\begin{equation*}
	B^{\alpha, -b+1/\min\{2,p\}}_{p,q}(\R^d) \not \subseteq B^{\alpha, -b+1/q}_{p,\min\{2, p\}}(\R^d)
	\end{equation*}
	and
	\begin{equation*}
	B^{\alpha, -b+1/q}_{p,\min\{2, p\}}(\R^d) \not \subseteq B^{\alpha, -b+1/\min\{2,p\}}_{p,q}(\R^d);
	\end{equation*}
	see \cite[Theorem 1]{Leopold} and \cite[Proposition 6.1]{DominguezTikhonov}. The corresponding comment also applies to the right-hand side embeddings of (\ref{ThmBL1}) and (\ref{ThmBL2}).
\item	The special case $\alpha = 1$ in (\ref{ThmBL1}) was already shown in \cite[Theorem 5.2]{CobosDominguezJMAA}. On the other hand, \eqref{ThmBL2} with $\alpha=1$ improves the following chain of embeddings given in \cite[(7.59) and (7.61)]{Haroske}
\begin{equation*}
	B^{1,-b+1/q}_{p, \min\{1,q\}}(\R^d) \hookrightarrow \L^{(1,-b)}_{p,q}(\R^d) \hookrightarrow B^{1,-b + 1/q}_{p, \infty}(\R^d)
\end{equation*}
because
\begin{equation*}
B^{1,-b+1/q}_{p, \min\{1,q\}}(\R^d)  \hookrightarrow B^{1, -b+1/q}_{p,\min\{2, p, q\}}(\R^d), \quad B^{1,-b+1/q}_{p, \min\{1,q\}}(\R^d)  \neq B^{1, -b+1/q}_{p,\min\{2, p, q\}}(\R^d) \quad \text{if} \quad q > 1,
\end{equation*}
and
\begin{equation*}
	B^{1, -b+1/q}_{p,\max\{2,p,q\}}(\R^d) \hookrightarrow B^{1,-b + 1/q}_{p, \infty}(\R^d), \quad B^{1, -b+1/q}_{p,\max\{2,p,q\}}(\R^d) \neq B^{1,-b + 1/q}_{p, \infty}(\R^d) \quad \text{if} \quad q < \infty.
\end{equation*}
\item If $q=\infty$ in (\ref{ThmBL2}) then
\begin{equation}\label{ClassBH**}
	B^{\alpha, -b}_{p,\min\{2, p\}}(\R^d) \hookrightarrow \L^{(\alpha,-b)}_{p,\infty}(\R^d) \hookrightarrow B^{\alpha, -b}_{p,\infty}(\R^d), \quad b > 0.
\end{equation}
Comparing the right-hand side embeddings of (\ref{ClassBH*}) and (\ref{ClassBH**}), one can observe the significant role played by logarithmic smoothness in the fine indices of Besov spaces. Furthermore, we will see in Theorem \ref{SharpThmBl} below that $\L^{(\alpha,-b)}_{p,\infty}(\R^d) \hookrightarrow B^{\alpha, -b}_{p,\infty}(\R^d)$ is optimal, in the sense that one cannot replace the target space $B^{\alpha, -b}_{p,\infty}(\R^d)$ by $B^{\alpha, -b}_{p,r}(\R^d), \, r < \infty$.

\item
In the limiting cases  $p=1$ and $p=\infty$ analogues of (\ref{ThmBL1}) and (\ref{ThmBL2}) read as follows:
\begin{equation*}
	B^{\alpha, -b+1/\min\{1,q\}}_{p,q}(\R^d) \hookrightarrow \L^{(\alpha,-b)}_{p,q}(\R^d) \hookrightarrow B^{\alpha, -b}_{p,q}(\R^d)
\end{equation*}
and
\begin{equation*}
	B^{\alpha, -b+1/q}_{p,\min\{1, q\}}(\R^d) \hookrightarrow \L^{(\alpha,-b)}_{p,q}(\R^d) \hookrightarrow B^{\alpha, -b+1/q}_{p,\infty}(\R^d).
\end{equation*}
See \cite[Corollary 7.20]{Haroske}
for the case $\alpha = 1$.
	\end{enumerate}
\end{rem}

\begin{proof}[Proof of Theorem \ref{ThmBL}]
Applying the interpolation property of the limiting interpolation method with $\theta = 1$ (see  (\ref{Klimitspace})) to the embeddings (\ref{ClassBH}), we obtain
\begin{align*}
	(B^0_{p, \min\{2,p\}}(\R^d), B^\alpha_{p, \min\{2,p\}}(\R^d))_{(1,-b),q} &  \hookrightarrow (L_p(\R^d), H^\alpha_p(\R^d))_{(1,-b),q} \\
	&\hspace{-4cm} \hookrightarrow (B^0_{p, \max\{2,p\}}(\R^d), B^\alpha_{p, \max\{2,p\}}(\R^d))_{(1,-b),q}
\end{align*}
and then, by (\ref{LipLimInter}),
\begin{align}
	(B^0_{p, \min\{2,p\}}(\R^d), B^\alpha_{p, \min\{2,p\}}(\R^d))_{(1,-b),q} &  \hookrightarrow \L^{(\alpha,-b)}_{p,q}(\R^d) \nonumber\\
	&\hspace{-4cm} \hookrightarrow (B^0_{p, \max\{2,p\}}(\R^d), B^\alpha_{p, \max\{2,p\}}(\R^d))_{(1,-b),q}. \label{ThmBLProof1}
\end{align}

Next we show that
\begin{equation}\label{ThmBLProof2}
	B^{\alpha, -b+1/\min\{2,p,q\}}_{p,q}(\R^d) + B^{\alpha, -b+1/q}_{p,\min\{2, p, q\}}(\R^d)  \hookrightarrow (B^0_{p, \min\{2,p\}}(\R^d), B^\alpha_{p, \min\{2,p\}}(\R^d))_{(1,-b),q}.
\end{equation}
Let $\alpha_0$ be such that $\alpha_0 > \alpha$ and set $\theta = \alpha/\alpha_0$. By (\ref{InterBesClass}),
\begin{equation}\label{ThmBLProof3*}
 (B^0_{p, r}(\R^d), B^{\alpha_0}_{p, r}(\R^d))_{\theta, r} = B^\alpha_{p,r}(\R^d), \quad 0 < r \leq \infty.
 \end{equation}
Putting $r = \min\{2,p\}$, it follows from (\ref{LemmaReiteration}) and (\ref{InterBes}) that
\begin{align*}
	 (B^0_{p, \min\{2,p\}}(\R^d), B^\alpha_{p, \min\{2,p\}}(\R^d))_{(1,-b),q}\\
	 & \hspace{-4cm} \hookleftarrow (B^0_{p, \min\{2,p\}}(\R^d), B^{\alpha_0}_{p, \min\{2,p\}}(\R^d))_{\theta, q; -b + 1/\min\{2,p,q\}} \\
	 &\hspace{-4cm} = B^{\alpha, -b + 1/\min\{2,p,q\}}_{p, q}(\R^d).
\end{align*}
On the other hand, by (\ref{LemLimIntVectSeq1}) and (\ref{LemLimIntVectSeq5}),
\begin{align*}
	(\ell_{\min\{2,p\}}(L_p(\R^d)), \ell_{\min\{2,p\}}(2^{j \alpha} L_p(\R^d)))_{(1,-b),q} \\
	& \hspace{-6cm}\hookleftarrow \ell_{\min\{2,p,q\}}((L_p(\R^d), 2^{j \alpha} L_p(\R^d))_{(1,-b),q}) \\
	&\hspace{-6cm} = \ell_{\min\{2,p,q\}} (2^{j \alpha} (1 + j)^{-b+1/q} L_p(\R^d)).
\end{align*}
Now, the embedding
\begin{equation*}
	(B^0_{p, \min\{2,p\}}(\R^d), B^\alpha_{p, \min\{2,p\}}(\R^d))_{(1,-b),q}  \hookleftarrow  B^{\alpha, -b+1/q}_{p,\min\{2,p,q\}}(\R^d)
\end{equation*}
is an immediate consequence of the retract theorem for interpolation methods (see \cite[Sects. 1.2.4, 2.4.1]{Triebel78}) since $B^0_{p, \min\{2,p\}}(\R^d), B^\alpha_{p, \min\{2,p\}}(\R^d)$ and $B^{\alpha, -b+1/q}_{p,\min\{2,p,q\}}(\R^d) $ are retracts of the vector-valued sequence spaces $\ell_{\min\{2,p\}}(L_p(\R^d)), \ell_{\min\{2,p\}}(2^{j \alpha} L_p(\R^d))$ and $\ell_{\min\{2,p,q\}} (2^{j \alpha} (1 + j)^{-b+1/q} L_p(\R^d))$, respectively. This completes the proof of (\ref{ThmBLProof2}).

We claim that
\begin{equation}\label{ThmBLProof3}
 (B^0_{p, \max\{2,p\}}(\R^d), B^\alpha_{p, \max\{2,p\}}(\R^d))_{(1,-b),q} \hookrightarrow B^{\alpha, -b+1/\max\{2,p,q\}}_{p,q}(\R^d) \cap B^{\alpha, -b+1/q}_{p,\max\{2, p, q\}}(\R^d).
\end{equation}
Indeed, applying (\ref{ThmBLProof3*}) with $r=\max\{2,p\}$, (\ref{LemmaReiteration}) and (\ref{InterBes}), we get
\begin{align*}
	 (B^0_{p, \max\{2,p\}}(\R^d), B^\alpha_{p, \max\{2,p\}}(\R^d))_{(1,-b),q}\\
	 & \hspace{-4cm} \hookrightarrow (B^0_{p, \max\{2,p\}}(\R^d), B^{\alpha_0}_{p, \max\{2,p\}}(\R^d))_{\theta, q; -b + 1/\max\{2,p,q\}} \\
	 &\hspace{-4cm} = B^{\alpha, -b + 1/\max\{2,p,q\}}_{p, q}(\R^d).
\end{align*}
Furthermore, in light of (\ref{LemLimIntVectSeq1}) and (\ref{LemLimIntVectSeq5}),
\begin{align*}
	(\ell_{\max\{2,p\}}(L_p(\R^d)), \ell_{\max\{2,p\}}(2^{j \alpha} L_p(\R^d)))_{(1,-b),q} \\
	& \hspace{-6cm}\hookrightarrow \ell_{\max\{2,p,q\}}((L_p(\R^d), 2^{j \alpha} L_p(\R^d))_{(1,-b),q}) \\
	&\hspace{-6cm} = \ell_{\max\{2,p,q\}} (2^{j \alpha} (1 + j)^{-b+1/q} L_p(\R^d))
\end{align*}
and so, the retraction method for interpolation allows us to derive
\begin{equation*}
	 (B^0_{p, \max\{2,p\}}(\R^d), B^\alpha_{p, \max\{2,p\}}(\R^d))_{(1,-b),q} \hookrightarrow B^{\alpha,-b+1/q}_{p, \max\{2,p,q\}}(\R^d).
\end{equation*}

Combining (\ref{ThmBLProof1}), (\ref{ThmBLProof2}) and (\ref{ThmBLProof3}) we arrive at the desired embeddings (\ref{ThmBL1}) and (\ref{ThmBL2}).

\end{proof}

\subsection{Jawerth-Franke type embeddings  for Lipschitz spaces}

We start by recalling the classical embeddings of Jawerth-Franke between Besov spaces and Sobolev spaces. See \cite{Jawerth, Franke} (cf. also \cite{Marschall} and \cite{Vybiral}).

\begin{thm}[\bf{Embeddings of Jawerth-Franke  for Sobolev spaces}]\label{ThmFrankeJawerthRecall}
	Let $1 \leq p_0 < p < p_1 \leq \infty$ and $-\infty < \alpha < \infty$. Then
\begin{equation}\label{FrankeMarschall}
		B^{\alpha + d(1/p_0 -1/p)}_{p_0,p}(\R^d) \hookrightarrow H^{\alpha}_p(\R^d) \hookrightarrow B^{\alpha + d(1/p_1 - 1/p)}_{p_1,p}(\R^d).
	\end{equation}
	The periodic counterparts also hold true.
\end{thm}

In view of (\ref{LipSob}), the embeddings (\ref{FrankeMarschall}) can be rewritten as follows
\begin{equation}\label{BL1}
		B^{\alpha + d(1/p_0 -1/p)}_{p_0,p}(\R^d) \hookrightarrow \L^{(\alpha,0)}_{p,\infty}(\R^d) \hookrightarrow B^{\alpha + d(1/p_1 - 1/p)}_{p_1,p}(\R^d), \quad \alpha > 0.
	\end{equation}
	
	The goal of this section is to extend (\ref{BL1}) to the full scale of the logarithmic Lipschitz spaces $\L^{(\alpha,-b)}_{p,q}(\R^d)$. In the first attempt, one can combine \eqref{ThmBL1}, \eqref{ThmBL2} and \eqref{ClasSobEmbBesov} below to get
	\begin{equation*}
		B^{\alpha + d(1/p_0 - 1/p), -b + 1/\min\{2,p,q\}}_{p_0,q}(\R^d) \hookrightarrow \L^{(\alpha,-b)}_{p,q}(\R^d) \hookrightarrow B^{\alpha + d(1/p_1 -1/p), -b + 1/\max\{2,p,q\}}_{p_1,q}(\R^d)
	\end{equation*}
	and
		\begin{equation*}
			B^{\alpha + d(1/p_0 - 1/p), -b + 1/q}_{p_0,\min\{2,p,q\}}(\R^d) \hookrightarrow \L^{(\alpha,-b)}_{p,q}(\R^d) \hookrightarrow B^{\alpha + d(1/p_1 -1/p), -b + 1/q}_{p_1,\max\{2,p,q\}}(\R^d)
		\end{equation*}
		for $\alpha > 0, 1 \leq p_0 < p < p_1 \leq \infty, 0 < q \leq \infty$, and $b > 1/q$. However, our next result shows that these embeddings can be strengthened.

	\begin{thm}[\bf{Embeddings of Jawerth-Franke type for Lipschitz spaces}]\label{ThMFrankeLip}
		Let $\alpha > 0, 1 \leq p_0 < p < p_1 \leq \infty, 0 < q\leq \infty,$ and $b > 1/q$. Then, we have
		\begin{equation}\label{ThmFrankeLip1}
			B^{\alpha + d(1/p_0 - 1/p), -b + 1/\min\{p,q\}}_{p_0,q}(\R^d) \hookrightarrow \emph{Lip}^{(\alpha,-b)}_{p,q}(\R^d) \hookrightarrow B^{\alpha + d(1/p_1 -1/p), -b + 1/\max\{p,q\}}_{p_1,q}(\R^d)
		\end{equation}
		and
		\begin{equation}\label{ThmFrankeLip2}
			B^{\alpha + d(1/p_0 - 1/p), -b + 1/q}_{p_0,\min\{p,q\}}(\R^d) \hookrightarrow \emph{Lip}^{(\alpha,-b)}_{p,q}(\R^d) \hookrightarrow B^{\alpha + d(1/p_1 -1/p), -b + 1/q}_{p_1,\max\{p,q\}}(\R^d).
		\end{equation}
		The corresponding embeddings for periodic spaces also hold true.
	\end{thm}
	
	\begin{rem}
	Note that the left-hand side embeddings (respectively, right-hand side embeddings) of (\ref{ThmFrankeLip1}) and (\ref{ThmFrankeLip2}) coincide if $q \leq p$ (respectively, $q \geq p$). However, in general, these embeddings are not comparable. For instance, we give the left-hand side embeddings of (\ref{ThmFrankeLip1}) and (\ref{ThmFrankeLip2}) with $q > p$,
	\begin{equation*}
		B^{\alpha + d(1/p_0 - 1/p), -b + 1/p}_{p_0,q}(\R^d) + B^{\alpha + d(1/p_0 - 1/p), -b + 1/q}_{p_0,p}(\R^d)  \hookrightarrow \L^{(\alpha,-b)}_{p,q}(\R^d),
	\end{equation*}
	and we observe that
	\begin{equation*}
	 B^{\alpha + d(1/p_0 - 1/p), -b + 1/p}_{p_0,q}(\R^d) \not \subseteq B^{\alpha + d(1/p_0 - 1/p), -b + 1/q}_{p_0,p}(\R^d)
	 \end{equation*}
	 and
	 \begin{equation*}
	 B^{\alpha + d(1/p_0 - 1/p), -b + 1/q}_{p_0,p}(\R^d) \ \not \subseteq B^{\alpha + d(1/p_0 - 1/p), -b + 1/p}_{p_0,q}(\R^d),
	 \end{equation*}
	 see \cite[Theorem 1]{Leopold} and \cite[Remark 6.4]{DominguezTikhonov}. The same comment applies to the right-hand side embeddings of (\ref{ThmFrankeLip1}) and (\ref{ThmFrankeLip2}) with $q < p$. More precisely, we have
	 \begin{equation*}
	  \L^{(\alpha,-b)}_{p,q}(\R^d) \hookrightarrow B^{\alpha + d(1/p_1 -1/p), -b + 1/p}_{p_1,q}(\R^d) \cap  B^{\alpha + d(1/p_1 -1/p), -b + 1/q}_{p_1,p}(\R^d),
	 \end{equation*}
	 and
	 \begin{equation*}
	 	B^{\alpha + d(1/p_1 -1/p), -b + 1/p}_{p_1,q}(\R^d) \not \subseteq  B^{\alpha + d(1/p_1 -1/p), -b + 1/q}_{p_1,p}(\R^d),
	 \end{equation*}
	  \begin{equation*}
	 	B^{\alpha + d(1/p_1 -1/p), -b + 1/q}_{p_1,p}(\R^d) \not \subseteq  B^{\alpha + d(1/p_1 -1/p), -b + 1/p}_{p_1,q}(\R^d).
	 \end{equation*}
	\end{rem}
	
	\begin{proof}[Proof of Theorem \ref{ThMFrankeLip}]
	Since
	\begin{equation*}
		B^{d(1/p_0 - 1/p)}_{p_0,p}(\R^d) \hookrightarrow L_p(\R^d) \hookrightarrow B^{d(1/p_1-1/p)}_{p_1,p}(\R^d)
	\end{equation*}
	and
	\begin{equation*}
	B^{\alpha + d(1/p_0 - 1/p)}_{p_0,p}(\R^d) \hookrightarrow H^{\alpha}_p(\R^d) \hookrightarrow B^{\alpha + d(1/p_1-1/p)}_{p_1,p}(\R^d)
	\end{equation*}
	(see (\ref{FrankeMarschall})), we can apply the limiting interpolation method with $\theta = 1$ (see (\ref{Klimitspace})) to obtain
	\begin{align}
		(B^{d(1/p_0 - 1/p)}_{p_0,p}(\R^d), B^{\alpha + d(1/p_0 - 1/p)}_{p_0,p}(\R^d))_{(1,-b),q} & \hookrightarrow (L_p(\R^d) , H^{\alpha}_p(\R^d))_{(1,-b),q} \nonumber \\
		& \hspace{-4cm}\hookrightarrow (B^{d(1/p_1-1/p)}_{p_1,p}(\R^d), B^{\alpha + d(1/p_1-1/p)}_{p_1,p}(\R^d))_{(1,-b),q}. \label{ThmEmbLipschitz2}
	\end{align}
	
	According to (\ref{LipLimInter}), we have
	\begin{equation}\label{ThmEmbLipschitz3}
	 (L_p(\R^d) , H^{\alpha}_p(\R^d))_{(1,-b),q} = \text{Lip}^{(\alpha, -b)}_{p, q}(\R^d).
	 \end{equation}
	
	 Next, we show that
	 \begin{align}
	 	B^{\alpha +d(1/p_0- 1/p), -b + 1/\min\{p,q\}}_{p_0,q}(\R^d) + B^{\alpha + d(1/p_0 - 1/p), -b + 1/q}_{p_0,\min\{p,q\}}(\R^d) \nonumber \\
		& \hspace{-5cm} \hookrightarrow (B^{d(1/p_0 - 1/p)}_{p_0,p}(\R^d), B^{\alpha + d(1/p_0 - 1/p)}_{p_0,p}(\R^d))_{(1,-b),q}. \label{ThmEmbLipschitz4}
	 \end{align}
	 Let $\alpha_0 > \alpha + d(1/p_0 - 1/p)$ and $\theta_0 = \alpha/(\alpha_0 - d/p_0+ d/p)$. By (\ref{InterBesClass}),
	 \begin{equation*}
	 	B^{\alpha + d(1/p_0 - 1/p)}_{p_0,p}(\R^d) = (B^{d(1/p_0 - 1/p)}_{p_0,p}(\R^d), B^{\alpha_0}_{p_0,p}(\R^d))_{\theta_0,p}.
	 \end{equation*}	
	 Therefore, by the left-hand side embedding in (\ref{LemmaReiteration}) and (\ref{InterBes}), we get
	 \begin{align}
	 	(B^{d(1/p_0 - 1/p)}_{p_0,p}(\R^d), B^{\alpha + d(1/p_0 - 1/p)}_{p_0,p}(\R^d))_{(1,-b),q}& \nonumber\\
		 & \hspace{-4cm} = (B^{d(1/p_0-1/p)}_{p_0,p}(\R^d),  (B^{d(1/p_0 - 1/p)}_{p_0,p}(\R^d), B^{\alpha_0}_{p_0,p}(\R^d))_{\theta_0,p})_{(1,-b),q} \nonumber \\
		& \hspace{-4cm}\hookleftarrow (B^{d(1/p_0-1/p)}_{p_0,p}(\R^d), B^{\alpha_0}_{p_0,p}(\R^d))_{\theta_0, q; -b + 1/\min\{p, q\} } \nonumber \\
		& \hspace{-4cm} = B^{\alpha + d(1/p_0 - 1/p), -b + 1/\min\{p,q\}}_{p_0,q}(\R^d).  \label{ThmEmbLipschitz4new1}
	 \end{align}
	
	 We claim that
	 \begin{equation}\label{ThmEmbLipschitz4new2}
	 B^{\alpha + d(1/p_0 - 1/p), -b + 1/q}_{p_0,\min\{p,q\}}(\R^d) \hookrightarrow (B^{d(1/p_0 - 1/p)}_{p_0,p}(\R^d), B^{\alpha + d(1/p_0 - 1/p)}_{p_0,p}(\R^d))_{(1,-b),q}.
	 \end{equation}
	 To get this we will apply the retraction property of Besov spaces (see \cite[Sections 1.2.4 and 2.4.1]{Triebel78}). It follows from the left-hand side embedding in (\ref{LemLimIntVectSeq1}), (\ref{LemLimIntVectSeq3}) and (\ref{LemLimIntVectSeq5}) that
	 \begin{align*}
	 	(\ell_p(2^{j(d/p_0 -d/p)} L_{p_0}(\R^d)), \ell_p(2^{j(\alpha + d/p_0 -d/p)} L_{p_0}(\R^d)))_{(1,-b),q} & \\
		& \hspace{-7cm} \hookleftarrow \ell_{\min\{p,q\}}((2^{j(d/p_0 -d/p)} L_{p_0}(\R^d), 2^{j(\alpha + d/p_0 -d/p)} L_{p_0}(\R^d))_{(1,-b),q}) \\
		& \hspace{-7cm} = \ell_{\min\{p,q\}} (2^{j(d/p_0 -d/p)} (L_{p_0}(\R^d), 2^{j \alpha} L_{p_0}(\R^d))_{(1,-b),q}) \\
		& \hspace{-7cm} = \ell_{\min\{p,q\}} (2^{j(\alpha + d/p_0 -d/p)} (1 + j)^{-b + 1/q} L_{p_0}(\R^d)).
	 \end{align*}
	 Now the embedding (\ref{ThmEmbLipschitz4new2}) follows from the fact that 
 $B^{\alpha + d(1/p_0 - 1/p), -b + 1/q}_{p_0,\min\{p,q\}}(\R^d)$, $B^{d(1/p_0 - 1/p)}_{p_0,p}(\R^d)$, and $B^{\alpha + d(1/p_0 - 1/p)}_{p_0,p}(\R^d)$ are retracts of $ \ell_{\min\{p,q\}} (2^{j(\alpha + d/p_0 -d/p)} (1 + j)^{-b + 1/q} L_{p_0}(\R^d)), \ell_p(2^{j(d/p_0 -d/p)} L_{p_0}(\R^d))$, and $\ell_p(2^{j(\alpha + d/p_0 -d/p)} L_{p_0}(\R^d))$, respectively.
	
	 As a byproduct of (\ref{ThmEmbLipschitz4new1}) and (\ref{ThmEmbLipschitz4new2}), we arrive at (\ref{ThmEmbLipschitz4}).
	
	 Next we will prove that the following embedding holds
	 \begin{align}
	 (B^{d(1/p_1-1/p)}_{p_1,p}(\R^d), B^{\alpha + d(1/p_1-1/p)}_{p_1,p}(\R^d))_{(1,-b),q} \nonumber \\
	 & \hspace{-4cm} \hookrightarrow B^{\alpha + d(1/p_1- 1/p), -b + 1/\max\{p,q\}}_{p_1,q}(\R^d) \cap B^{\alpha + d(1/p_1 - 1/p), -b + 1/q}_{p_1,\max\{p,q\}}(\R^d). \label{ThmEmbLipschitz4new2.2}
	 \end{align}
	
	  Let $\alpha_1 > \alpha + d(1/p_1 - 1/p)$ and $\theta_1 = \alpha/(\alpha_1 - d/p_1+ d/p)$. According to (\ref{InterBesClass}),
	 \begin{equation*}
	 	B^{\alpha + d(1/p_1 - 1/p)}_{p_1,p}(\R^d) = (B^{d(1/p_1 - 1/p)}_{p_1,p}(\R^d), B^{\alpha_1}_{p_1,p}(\R^d))_{\theta_1,p}.
	 \end{equation*}	
	 Therefore, invoking the right-hand side embedding in (\ref{LemmaReiteration}) together with (\ref{InterBes}),
	 \begin{align*}
	 	(B^{d(1/p_1 - 1/p)}_{p_1,p}(\R^d), B^{\alpha + d(1/p_1 - 1/p)}_{p_1,p}(\R^d))_{(1,-b),q}&\\
		 & \hspace{-4cm} = (B^{d(1/p_1-1/p)}_{p_1,p}(\R^d),  (B^{d(1/p_1 - 1/p)}_{p_1,p}(\R^d), B^{\alpha_1}_{p_1,p}(\R^d))_{\theta_1,p})_{(1,-b),q}  \\
		& \hspace{-4cm}\hookrightarrow (B^{d(1/p_1-1/p)}_{p_1,p}(\R^d), B^{\alpha_1}_{p_1,p}(\R^d))_{\theta_1, q; -b + 1/\max\{p, q\} } \\
		& \hspace{-4cm} = B^{\alpha + d(1/p_1 - 1/p), -b + 1/\max\{p,q\}}_{p_1,q}(\R^d).
	 \end{align*}
	
	 On the other hand, using the right-hand side embedding in (\ref{LemLimIntVectSeq1}), (\ref{LemLimIntVectSeq3}) and (\ref{LemLimIntVectSeq5}),
	 \begin{align*}
	 	(\ell_p(2^{j(d/p_1 -d/p)} L_{p_1}(\R^d)), \ell_p(2^{j(\alpha + d/p_1 -d/p)} L_{p_1}(\R^d)))_{(1,-b),q} & \\
		& \hspace{-7cm} \hookrightarrow \ell_{\max\{p,q\}}((2^{j(d/p_1 -d/p)} L_{p_1}(\R^d), 2^{j(\alpha + d/p_1-d/p)} L_{p_1}(\R^d))_{(1,-b),q}) \\
		& \hspace{-7cm} = \ell_{\max\{p,q\}} (2^{j(d/p_1 -d/p)} (L_{p_1}(\R^d), 2^{j \alpha} L_{p_1}(\R^d))_{(1,-b),q}) \\
		& \hspace{-7cm} = \ell_{\max\{p,q\}} (2^{j(\alpha + d/p_1 -d/p)} (1 + j)^{-b + 1/q} L_{p_1}(\R^d)).
	 \end{align*}
	 Consequently, the retraction method allows us to get
	 \begin{equation*}
		(B^{d(1/p_1 - 1/p)}_{p_1,p}(\R^d), B^{\alpha + d(1/p_1 - 1/p)}_{p_1,p}(\R^d))_{(1,-b),q} \hookrightarrow B_{p_1, \max\{p,q\}}^{\alpha + d(1/p_1 -1/p), -b + 1/q}(\R^d).
	\end{equation*}
	This completes the proof of (\ref{ThmEmbLipschitz4new2.2}).

	 Combining (\ref{ThmEmbLipschitz2}), (\ref{ThmEmbLipschitz3}), (\ref{ThmEmbLipschitz4}) and (\ref{ThmEmbLipschitz4new2.2}), we derive (\ref{ThmFrankeLip1}) and (\ref{ThmFrankeLip2}).
	
	 The same methodology can be applied to deal with periodic function spaces. Further details are left to the reader.
	
	 \end{proof}

\section{Embeddings between Lipschitz spaces}

\subsection{Embeddings with constant integrability}

In this section we investigate embeddings between the Lipschitz spaces $\L^{(\alpha,-b)}_{p,q}(\R^d)$ with fixed integrability $p$. Namely, we obtain the following

\begin{thm}\label{ThmEmbLipIntegrability}
	Let $1 < p < \infty, \alpha_i > 0, 0 < q_i \leq \infty$, and $b_i > 1/q_i, i = 0,1$. Assume that one of the following conditions is satisfied
	\begin{enumerate}[\upshape(i)]
		\item\label{ThmEmbLipIntegrability1} $\alpha_0 > \alpha_1$,
		\item\label{ThmEmbLipIntegrability2} $\alpha_0 = \alpha_1, \quad b_1 -\frac{1}{q_1} > b_0 -\frac{1}{q_0}$,
		\item\label{ThmEmbLipIntegrability3} $\alpha_0 = \alpha_1, \quad b_1 -\frac{1}{q_1} = b_0 -\frac{1}{q_0}, \quad q_0 \leq q_1$.
	\end{enumerate}
	Then,
	\begin{equation}\label{ThmEmbLipIntegrability4}
		\emph{Lip}^{(\alpha_0,-b_0)}_{p,q_0}(\R^d) \hookrightarrow \emph{Lip}^{(\alpha_1,-b_1)}_{p,q_1}(\R^d).
	\end{equation}
	The corresponding embeddings for periodic spaces also hold true.
\end{thm}

\begin{rem}
	It will be shown in Theorem \ref{ThmEmbLipIntegrabilitySharp} below that the above conditions are indeed necessary to establish \eqref{ThmEmbLipIntegrability4}.
\end{rem}

\begin{proof}[Proof of Theorem \ref{ThmEmbLipIntegrability}]
	\eqref{ThmEmbLipIntegrability1}: We will make use of the following well-known embeddings for Besov spaces
	\begin{equation}\label{ProofThmEmbLipIntegrability1}
		B^{\alpha_0, \xi_0}_{p, q_0}(\R^d) \hookrightarrow B^{\alpha_1, \xi_1}_{p, q_1}(\R^d)
	\end{equation}
	for $-\infty <\alpha_1 < \alpha_0< \infty, -\infty < \xi_0, \xi_1 < \infty,$ and $0 < q_0, q_1 \leq \infty$. See \cite[Section 2.3.2, Proposition 2]{Triebel83}, \cite[Proposition 1.9(ii)]{Moura} and \cite[Proposition 5.3]{CaetanoHaroske}.
	
	In light of \eqref{ThmBL1} and \eqref{ProofThmEmbLipIntegrability1}, we have
	\begin{align*}
		\L^{(\alpha_0,-b_0)}_{p,q_0}(\R^d) & \hookrightarrow B^{\alpha_0,-b_0 + 1/\max\{2,p,q_0\}}_{p,q_0}(\R^d) \hookrightarrow B^{\alpha_1,-b_1 + 1/\min\{2,p,q_1\}}_{p,q_1}(\R^d) \\
		& \hookrightarrow \L^{(\alpha_1,-b_1)}_{p,q_1}(\R^d).
	\end{align*}
	
	\eqref{ThmEmbLipIntegrability2}, \eqref{ThmEmbLipIntegrability3}: Let $\alpha = \alpha_0 = \alpha_1$. Firstly, we assume $q_1 < q_0$ and $b_1 -\frac{1}{q_1} > b_0 -\frac{1}{q_0}$. Applying H\"older's inequality,
	\begin{align*}
		\Big(\int_0^1 (t^{-\alpha} (1 - \log t)^{-b_1} \omega_\alpha(f,t)_p)^{q_1} \frac{\dint t}{t} \Big)^{1/q_1} \\
		& \hspace{-6.5cm} \leq \Big(\int_0^1 (t^{-\alpha} (1 - \log t)^{-b_0} \omega_\alpha(f,t)_p)^{q_0} \frac{\dint t}{t} \Big)^{1/q_0}   \Big(\int_0^1 (1 - \log t)^{(b_0 -b_1)(\frac{1}{q_1}-\frac{1}{q_0})^{-1}}  \frac{\dint t}{t}\Big)^{\frac{1}{q_1}-\frac{1}{q_0}} \\
		& \hspace{-6.5cm} \lesssim  \Big(\int_0^1 (t^{-\alpha} (1 - \log t)^{-b_0} \omega_\alpha(f,t)_p)^{q_0} \frac{\dint t}{t} \Big)^{1/q_0}.
	\end{align*}
	
	Secondly, suppose that $q_0 \leq q_1$ and $b_1 -\frac{1}{q_1} \geq b_0 -\frac{1}{q_0}$. By monotonicity properties (noting that $\omega_\alpha(f,t)_p/t^\alpha$ is equivalent to a decreasing function, see, e.g., \cite{KolomoitsevTikhonov1}), we have
	\begin{align*}
		\Big(\int_0^1 (t^{-\alpha} (1 - \log t)^{-b_1} \omega_\alpha(f,t)_p)^{q_1} \frac{\dint t}{t} \Big)^{1/q_1} & \lesssim \Big(\sum_{n=0}^\infty (2^{2^n \alpha} 2^{n(-b_1 + 1/q_1)} \omega_\alpha(f,2^{-2^n})_p )^{q_1} \Big)^{1/q_1} \\
		& \hspace{-5cm}\leq \Big(\sum_{n=0}^\infty (2^{2^n \alpha} 2^{n(-b_0 + 1/q_0)} \omega_\alpha(f,2^{-2^n})_p )^{q_0} \Big)^{1/q_0} \\
		& \hspace{-5cm}\lesssim \Big(\int_0^1 (t^{-\alpha} (1 - \log t)^{-b_0} \omega_\alpha(f,t)_p)^{q_0} \frac{\dint t}{t} \Big)^{1/q_0}.
	\end{align*}
\end{proof}

\subsection{Embeddings with constant differential dimension}
Let $1 < p_0 < p_1 < \infty, 0 < q \leq \infty$ and $-\infty < \alpha_1 < \alpha_0 < \infty$. Further, we assume that
\begin{equation}\label{AssumptionClasSobEmbBesov}
	\alpha_0 -\frac{d}{p_0} = \alpha_1 - \frac{d}{p_1}.
\end{equation}
Then, the classical Sobolev embeddings assert that
\begin{equation}\label{ClasSobEmbBesov}
	B^{\alpha_0,b}_{p_0,q}(\R^d) \hookrightarrow B^{\alpha_1,b}_{p_1,q}(\R^d), \quad -\infty < b < \infty,
\end{equation}
and
\begin{equation}\label{ClasSobEmb}
	H^{\alpha_0}_{p_0}(\R^d) \hookrightarrow H^{\alpha_1}_{p_1}(\R^d).
\end{equation}
We remark that \eqref{ClasSobEmbBesov} also holds true for $1 \leq p_0 < p_1 \leq \infty$ satisfying \eqref{AssumptionClasSobEmbBesov}, and \eqref{ClasSobEmb} can be extended to the broader scale of Triebel-Lizorkin spaces. Furthermore, both embeddings can be given in the more general setting of quasi-Banach spaces. For further details, we refer to \cite[Section 2.7.1]{Triebel83}, \cite[Proposition 1.9]{Moura}, \cite[Proposition 5.3]{CaetanoHaroske} and the references within. We note that since $H^\alpha_p(\R^d) = \L^{(\alpha,0)}_{p,\infty}(\R^d), \alpha > 0$, \eqref{ClasSobEmb} can be rewritten in terms of Lipschitz spaces as
\begin{equation*}
		\L^{(\alpha_0, 0)}_{p_0,\infty}(\R^d) \hookrightarrow \L^{(\alpha_1, 0)}_{p_1,\infty}(\R^d), \quad 0 < \alpha_1 < \alpha_0 < \infty.
	\end{equation*}
Next we extend this embedding to the full range of parameters. More precisely, we obtain the following Sobolev-type embedding.

\begin{thm}\label{ThmSobEmbLip}
	Let $1 < p_0 < p_1 < \infty, 0 < \alpha_1 < \alpha_0 < \infty$ with $\alpha_0 -d/p_0 = \alpha_1 -d/p_1$. Let $0 < q \leq \infty$ and $b > 1/q$. Then,
	\begin{equation*}
		\emph{Lip}^{(\alpha_0, -b)}_{p_0,q}(\R^d) \hookrightarrow 	\emph{Lip}^{(\alpha_1, -b)}_{p_1,q}(\R^d).
	\end{equation*}
	The corresponding result for periodic spaces also holds true.
\end{thm}

\begin{rem}
	The corresponding results in the endpoint cases $p_0 = 1$ and/or $p_1=\infty$ are more delicate and will be stated in Theorems \ref{ThmEmbLipschitz} and \ref{ThmEmbLipschitzinfty} below.
\end{rem}

\begin{proof}[Proof of Theorem \ref{ThmSobEmbLip}]
	We choose $\theta$ such that
	\begin{equation*}
	1 - \frac{\alpha_1}{\alpha_0} < \theta < \min\Big\{1, \frac{d}{\alpha_0 p_0}\Big\}
	\end{equation*}
	 and let $\lambda = \theta \alpha_0 - d \Big(\frac{1}{p_0} - \frac{1}{p_1}\Big)$. According to \eqref{ClasSobEmbBesov} and \eqref{ClasSobEmb}, we have
	\begin{equation*}
		B^{\theta \alpha_0}_{p_0,q}(\R^d) \hookrightarrow B^\lambda_{p_1,q}(\R^d) \quad \text{and} \quad H^{\alpha_0}_{p_0}(\R^d) \hookrightarrow H^{\alpha_1}_{p_1}(\R^d).
	\end{equation*}
	Then, by the interpolation property,
	\begin{equation}\label{ProofThmSobEmbLip1}
		(B^{\theta \alpha_0}_{p_0,q}(\R^d), H^{\alpha_0}_{p_0}(\R^d))_{(1,-b),q} \hookrightarrow (B^\lambda_{p_1,q}(\R^d), H^{\alpha_1}_{p_1}(\R^d))_{(1,-b),q}.
	\end{equation}
	Next we identify these interpolation spaces. Since $B^{\theta \alpha_0}_{p_0,q}(\R^d) = (L_{p_0}, H^{\alpha_0}_{p_0}(\R^d))_{\theta,q}$ (see \eqref{InterBesSobLp2}), it follows from \eqref{LemmaReiteration3} and \eqref{LipLimInter} that
	\begin{equation}\label{ProofThmSobEmbLip2}
	(B^{\theta \alpha_0}_{p_0,q}(\R^d), H^{\alpha_0}_{p_0}(\R^d))_{(1,-b),q}  = (L_{p_0}(\R^d), H^{\alpha_0}_{p_0}(\R^d))_{(1,-b),q} =  \L^{(\alpha_0,-b)}_{p_0,q}(\R^d).
	\end{equation}
	On the other hand, setting $\eta = 1 - (1-\theta) \alpha_0/\alpha_1 \in (0,1)$ then
	\begin{equation*}
		B^\lambda_{p_1,q}(\R^d) = (L_{p_1}(\R^d), H^{\alpha_1}_{p_1}(\R^d))_{\eta,q}
	\end{equation*}
	and applying again \eqref{LemmaReiteration3} and \eqref{LipLimInter}, we get
	\begin{equation}\label{ProofThmSobEmbLip3}
		(B^\lambda_{p_1,q}(\R^d), H^{\alpha_1}_{p_1}(\R^d))_{(1,-b),q} = \L^{(\alpha_1, -b)}_{p_1,q}(\R^d).
	\end{equation}
	Plugging \eqref{ProofThmSobEmbLip2} and \eqref{ProofThmSobEmbLip3} into \eqref{ProofThmSobEmbLip1}, we arrive at $\L^{(\alpha_0, -b)}_{p_0,q}(\R^d) \hookrightarrow 	\L^{(\alpha_1, -b)}_{p_1,q}(\R^d).$
\end{proof}

Next we turn our attention to the counterpart of Theorem \ref{ThmSobEmbLip} in the limiting case $p_0 = 1$.

\begin{thm}\label{ThmEmbLipschitz}
Let $1 < p < \infty,$ $0 < q \leq \infty$, and $b  > 1/q \, (b \geq 0  \text{ if }  q= \infty)$.
\begin{enumerate}[\upshape(i)]
	\item\label{ThmEmbLipschitzL1} Let $d \in \N,  0 < \alpha_ 1 < \alpha_0 < \infty$ with $\alpha_0 - d = \alpha_1-d/p$. Then, we have
	\begin{equation}\label{ThmEmbLipschitz*}
		\emph{Lip}^{(\alpha_0,-b)}_{1,q}(\R^d) \hookrightarrow \emph{Lip}^{(\alpha_1, -b-1/p)}_{p, q}(\R^d).
	\end{equation}
	
 \item\label{ThmEmbLipschitzL2} Let $d \geq 2, k \in \N, k \geq 2, 0 < \alpha_ 1 \leq k-1$ and $k - d = \alpha_1-d/p$. Then, we have
	\begin{equation}\label{ThmEmbLipschitz*new}
		\emph{Lip}^{(k,-b)}_{1,q}(\R^d) \hookrightarrow \emph{Lip}^{(\alpha_1, -b)}_{p, q}(\R^d).
	\end{equation}
	\end{enumerate}
	
	The corresponding results for periodic spaces also hold true.
\end{thm}

\begin{rem}\label{RemarkThmEmbLipschitz}
	One can observe that the target spaces in \eqref{ThmEmbLipschitz*} involve the additional logarithmic smoothness of order $1/p$. Such a phenomenon does not arise in the non-limiting case given in Theorem \ref{ThmSobEmbLip}. Furthermore,  we will see in Theorem \ref{ThmEmbLipschitzSharp} below that the embedding \eqref{ThmEmbLipschitz*} is, in general, sharp. However, \eqref{ThmEmbLipschitz*new} shows that under additional restrictions \eqref{ThmEmbLipschitz*} can be sharpened. Note that the assumptions given in \eqref{ThmEmbLipschitzL2} imply $\alpha_1 > k-d$.
\end{rem}

\begin{rem}
	The case $\alpha_0= 1, q=\infty$ and $b=0$ in \eqref{ThmEmbLipschitz*} is of special interest. Namely, the following embedding holds, recall \eqref{BV=Lip},
	\begin{equation}\label{RemarkThmEmbLipschitz1}
		\text{BV}(\R^d) \hookrightarrow \L^{(1-d+d/p,-1/p)}_{p,\infty}(\R), \quad 1 < p < \frac{d}{d-1} \quad (1 < p < \infty \quad \text{if} \quad d=1).
	\end{equation}
	Further, as already mentioned in Remark \ref{RemarkThmEmbLipschitz}, this result is optimal (see Theorem \ref{ThmEmbLipschitzSharp} below for the precise statement). This shows a striking difference between $\text{BV}(\R)$ and Sobolev spaces $H^{1/p}_p(\R), \, p > 1$. To be more precise, by \eqref{ClasSobEmb},
	 \begin{equation*}
	H^{1/p_0}_{p_0}(\R) = \L^{(1/p_0,0)}_{p_0,\infty}(\R) \hookrightarrow H^{1/p_1}_{p_1}(\R) = \L^{(1/p_1,0)}_{p_1,\infty}(\R) , \quad 1 < p_0 < p_1 < \infty.
\end{equation*}
However, the latter fails to be true if $p_0 = 1$ (that is, working with $\text{BV}(\R) = \L^{(1,0)}_{1,\infty}(\R)$) and the best possible embedding result involves additional logarithmic smoothness (see \eqref{RemarkThmEmbLipschitz1}). Note that $\L^{(1/p,0)}_{p,\infty} \subsetneq \L^{(1/p,-1/p)}_{p,\infty}(\R)$ (see Theorem \ref{ThmEmbLipIntegrability}).
\end{rem}

\begin{proof}[Proof of Theorem \ref{ThmEmbLipschitz}]

\eqref{ThmEmbLipschitzL1}: Let $q = \infty$ and $b \geq 0$. Since $\omega_\alpha(f,t)_1 \lesssim \omega_{\alpha_0}(f,t)_1, \, \alpha > \alpha_0$, we have $\L^{(\alpha_0,-b)}_{1,\infty}(\R^d) \hookrightarrow B^{\alpha_0, -b}_{1,\infty}(\R^d)$. On the other hand, by \eqref{ThmFrankeLip1}, $B^{\alpha_0, -b}_{1,\infty}(\R^d) \hookrightarrow \L^{(\alpha_1,-b -1/p)}_{p,\infty}(\R^d)$. Therefore,
\begin{equation}\label{ProofThmEmbLipschitz1}
 \L^{(\alpha_0,-b)}_{1,\infty}(\R^d) \hookrightarrow \L^{(\alpha_1, -b-1/p)}_{p, \infty}(\R^d), \quad b \geq 0.
 \end{equation}

Suppose now $q < \infty$ and $b -1/q> 0$. Let $0 < b_1 < b-1/q < b_0$. Then, there exists $\theta \in (0,1)$ such that $b -1/q = (1-\theta) b_0 + \theta b_1$. Further, it follows from \eqref{ProofThmEmbLipschitz1} that
\begin{equation*}
	\L^{(\alpha_0,-b_i)}_{1,\infty}(\R^d) \hookrightarrow \L^{(\alpha_1, -b_i-1/p)}_{p, \infty}(\R^d), \quad i = 0, 1,
\end{equation*}
and so, by the interpolation property,
\begin{equation}\label{ProofThmEmbLipschitz2}
	(\L^{(\alpha_0,-b_0)}_{1,\infty}(\R^d), \L^{(\alpha_0,-b_1)}_{1,\infty}(\R^d))_{\theta, q} \hookrightarrow ( \L^{(\alpha_1, -b_0-1/p)}_{p, \infty}(\R^d),  \L^{(\alpha_1, -b_1-1/p)}_{p, \infty}(\R^d))_{\theta,q}.
\end{equation}
According to Lemma \ref{LemmaInterpolationLipschitz},
\begin{equation*}
	(\L^{(\alpha_0,-b_0)}_{1,\infty}(\R^d), \L^{(\alpha_0,-b_1)}_{1,\infty}(\R^d))_{\theta, q} = \L_{1,q}^{(\alpha_0, -b)}(\R^d)
\end{equation*}
and
\begin{equation*}
	( \L^{(\alpha_1, -b_0-1/p)}_{p, \infty}(\R^d),  \L^{(\alpha_1, -b_1-1/p)}_{p, \infty}(\R^d))_{\theta,q} = \L_{p, q}^{(\alpha_1, -b-1/p)}(\R^d).
\end{equation*}
Inserting these formulas into \eqref{ProofThmEmbLipschitz2}, we arrive at
	\begin{equation*}
		\L^{(\alpha_0,-b)}_{1,q}(\R^d) \hookrightarrow \L^{(\alpha_1, -b-1/p)}_{p, q}(\R^d).
	\end{equation*}

\eqref{ThmEmbLipschitzL2}: By the classical Sobolev theorem and \eqref{ClasSobEmb},
	\begin{equation*}
		W^k_1(\R^d) \hookrightarrow W^{k-1}_{\frac{d}{d-1}}(\R^d) = H^{k-1}_{\frac{d}{d-1}}(\R^d) \hookrightarrow H^{\alpha_1}_p(\R^d).
	\end{equation*}
	Further, we will make use of the well-known embedding
	\begin{equation*}
		B^{k - \alpha_1}_{1,p}(\R^d) \hookrightarrow L_p(\R^d),
	\end{equation*}
cf. \cite{peetre-66} or \cite[Remark 11.8]{Triebel01} for further details on the history. We return to it in \eqref{LackSobEmbLebesgue} below. Applying now limiting interpolation and \eqref{LipLimInter}, we get
	\begin{equation*}
		(B^{k - \alpha_1}_{1,p}(\R^d), W^k_1(\R^d))_{(1,-b),q} \hookrightarrow (L_p(\R^d), H^{\alpha_1}_p(\R^d))_{(1,-b),q}  = \L^{(\alpha_1,-b)}_{p,q}(\R^d).
	\end{equation*}
	It remains to compute the domain space $(B^{k - \alpha_1}_{1,p}(\R^d), W^k_1(\R^d))_{(1,-b),q}$. It follows from \eqref{BesovInter9}, \eqref{LemmaReiteration3} and \eqref{LipLimInter*} that
	\begin{align*}
		(B^{k - \alpha_1}_{1,p}(\R^d), W^k_1(\R^d))_{(1,-b),q} & = ((L_1(\R^d), W^k_1(\R^d))_{\frac{k-\alpha_1}{k},p}, W^k_1(\R^d))_{(1,-b),q} \\
		& \hspace{-3cm}= (L_1(\R^d), W^k_1(\R^d))_{(1,-b),q} = \L^{(k,-b)}_{1,q}(\R^d).
	\end{align*}
	This completes the proof of \eqref{ThmEmbLipschitz*new}.
\end{proof}

\subsection{Br\'ezis-Wainger embeddings}
The Br\'ezis-Wainger embedding \cite{BrezisWainger} asserts that
\begin{equation}\label{BW1}
	H^{1 + d/p}_p(\R^d) \hookrightarrow \L^{(1,-1 + 1/p)}_{\infty, \infty}(\R^d), \quad 1 < p < \infty.
\end{equation}
Note that this embedding can be rewritten as
\begin{equation}\label{BW2}
	\L^{(1 + d/p,0)}_{p,\infty}(\R^d) \hookrightarrow \L^{(1,-1 + 1/p)}_{\infty, \infty}(\R^d),
\end{equation}
which corresponds to a special case of the limiting version of Theorem \ref{ThmSobEmbLip} with $p_1=\infty$.

The goal of this section is to provide the counterpart of Theorem \ref{ThmSobEmbLip} in the limiting case $p_1 = \infty$, or equivalently, to extend \eqref{BW1} and \eqref{BW2} to the full range of parameters.

 \begin{thm}\label{ThmEmbLipschitzinfty}
Let $b  > 1/q \, (b \geq 0 \, \text{ if } \, q=\infty)$.
 \begin{enumerate}[\upshape(i)]
	\item\label{ThmEmbLipschitzinftyL1} Let $1 < p < \infty$ and $0 < \alpha  < \infty$. Then, we have
	\begin{equation}\label{ThmEmbLipschitzinfty*}
		\emph{Lip}^{(\alpha + d/p,-b)}_{p,q}(\R^d) \hookrightarrow \emph{Lip}^{(\alpha, -b-1 +1/p)}_{\infty, q}(\R^d).
	\end{equation}
	\item\label{ThmEmbLipschitzinftyL2} Let $k \in \N$. Then, we have
	\begin{equation}\label{ThmEmbLipschitzinfty**}
		\emph{Lip}^{(k + d,-b)}_{1,q}(\R^d) \hookrightarrow \emph{Lip}^{(k, -b)}_{\infty, q}(\R^d).
	\end{equation}
	\end{enumerate}
	The corresponding results for periodic spaces also hold true.
\end{thm}

\begin{rem}
	It will be shown in Theorem \ref{ThmEmbLipschitzinftySharp} below that the embedding \eqref{ThmEmbLipschitzinfty*} is optimal in the sense that the shift $-1+1/p$ given in the logarithmic smoothness of the target space cannot be improved.
\end{rem}

\begin{rem}
	Setting $q=\infty$ and $b=0$ in \eqref{ThmEmbLipschitzinfty**}, we obtain
	\begin{equation*}
		\text{BV}^{k+d-1}(\R^d) \hookrightarrow \L^k(\R^d), \quad k \in \N.
	\end{equation*}
\end{rem}

\begin{proof}[Proof of Theorem \ref{ThmEmbLipschitzinfty}]
	\eqref{ThmEmbLipschitzinftyL1}: We shall divide the proof into several steps.
	
	\textsc{Step 1:} 	In virtue of Marchaud's inequality,
	\begin{equation*}
		 t^{-\alpha} \omega_{\alpha}(f,t)_\infty \lesssim \int_t^\infty \frac{\omega_{\alpha+d/p}(f,u)_\infty}{u^{\alpha}} \frac{\dint u}{u},
	\end{equation*}
	see, e.g., \cite{KolomoitsevTikhonov1}, we can apply H\"older's inequality to obtain
	\begin{equation*}
		 t^{-\alpha} \omega_{\alpha}(f,t)_\infty \lesssim (1-\log t)^{b -1/p} \|f\|_{B^{\alpha,-b+1}_{\infty,p}(\R^d)}, \quad b > 1/p,
	\end{equation*}
	or equivalently,
	\begin{equation}\label{ProofThmEmbLipschitzinfty2}
		B^{\alpha,-b+1}_{\infty,p}(\R^d) \hookrightarrow \L^{(\alpha,-b+1/p)}_{\infty,\infty}(\R^d), \quad b > 1/p.
	\end{equation}
	
	\textsc{Step 2:} We show \eqref{ThmEmbLipschitzinfty*} with $q=\infty$ and $b=0$, that is,
	\begin{equation*}
		H^{\alpha+d/p}_p(\R^d) \hookrightarrow \L^{(\alpha,-1+1/p)}_{\infty,\infty}(\R^d).
	\end{equation*}
	Indeed, by \eqref{FrankeMarschall} and \eqref{ProofThmEmbLipschitzinfty2}, we derive
	\begin{equation*}
		H^{\alpha+d/p}_p(\R^d) \hookrightarrow  B^\alpha_{\infty,p}(\R^d) \hookrightarrow \L^{(\alpha,-1+1/p)}_{\infty, \infty}(\R^d).
	\end{equation*}

	\textsc{Step 3:} We make the following assertion
	\begin{equation}\label{ProofThmEmbLipschitzinfty1}
		\L^{(\alpha + d/p,-b)}_{p,1}(\R^d) \hookrightarrow \L^{(\alpha,-b + 1/p)}_{\infty,\infty}(\R^d), \quad b > 1.
	\end{equation}
	Indeed, using \eqref{ThmFrankeLip2}, we have
	\begin{equation}\label{ProofThmEmbLipschitzinfty3}
		\L^{(\alpha+d/p,-b)}_{p,1}(\R^d) \hookrightarrow B^{\alpha, -b+1}_{\infty,p}(\R^d).
	\end{equation}
	Then, \eqref{ProofThmEmbLipschitzinfty1} follows from \eqref{ProofThmEmbLipschitzinfty2} and \eqref{ProofThmEmbLipschitzinfty3}.
	
	Let $0 < q \leq \infty$ and $b > 1/q$. We choose $b_0$ and $b_1$ satisfying $1 < b_1 < 1 + b -1/q < b_0$ and let $\theta \in (0,1)$ such that $ 1 + b -1/q = (1-\theta) b_1 + \theta b_0$. According to \eqref{ProofThmEmbLipschitzinfty1}, we have
	\begin{equation*}
		\L^{(\alpha+d/p,-b_i)}_{p,1}(\R^d) \hookrightarrow \L^{(\alpha,-b_i + 1/p)}_{\infty,\infty}(\R^d), \quad i=0, 1,
	\end{equation*}
	and thus
	\begin{equation}\label{ProofThmEmbLipschitzinfty4}
		(\L^{(\alpha+d/p,-b_0)}_{p,1}(\R^d), \L^{(\alpha+d/p,-b_1)}_{p,1}(\R^d))_{\theta,q} \hookrightarrow (\L^{(\alpha,-b_0 + 1/p)}_{\infty,\infty}(\R^d), \L^{(\alpha,-b_1 + 1/p)}_{\infty,\infty}(\R^d))_{\theta,q}.
	\end{equation}
	By Lemma \ref{LemmaInterpolationLipschitz},
	\begin{equation*}
		(\L^{(\alpha+d/p,-b_0)}_{p,1}(\R^d), \L^{(\alpha+d/p,-b_1)}_{p,1}(\R^d))_{\theta,q} = \L^{(\alpha+d/p, -b)}_{p,q}(\R^d)
	\end{equation*}
	and
	\begin{equation*}
		(\L^{(\alpha,-b_0 + 1/p)}_{\infty,\infty}(\R^d), \L^{(\alpha,-b_1 + 1/p)}_{\infty,\infty}(\R^d))_{\theta,q} = \L^{(\alpha,-b-1+1/p)}_{\infty,q}(\R^d).
	\end{equation*}
	Inserting these formulas into \eqref{ProofThmEmbLipschitzinfty4} we achieve \eqref{ThmEmbLipschitzinfty*}.
	
	\eqref{ThmEmbLipschitzinftyL2}: It follows from the trivial embedding $W^{k+d}_1(\R^d) \hookrightarrow W^k_\infty(\R^d)$ and the well-known result $B^d_{1,1}(\R^d) \hookrightarrow L_\infty(\R^d)$ (see \cite[Theorem 2.8.3]{Triebel83} or \eqref{PropEmbSobBesLinfty1} below) that
	\begin{equation*}
		(B^d_{1,1}(\R^d), W^{k+d}_1(\R^d))_{(1,-b),q} \hookrightarrow ( L_\infty(\R^d), W^k_\infty(\R^d))_{(1,-b),q} = \L_{\infty,q}^{(k,-b)}(\R^d)
	\end{equation*}
	where we have also used \eqref{LipLimInter*}. To find the   space $(B^d_{1,1}(\R^d), W^{k+d}_1(\R^d))_{(1,-b),q}$, we make  use of \eqref{BesovInter9}, \eqref{LemmaReiteration3} and \eqref{LipLimInter*}. There holds
	\begin{align*}
		(B^d_{1,1}(\R^d), W^{k+d}_1(\R^d))_{(1,-b),q} &= ((L_1(\R^d), W^{k+d}_1(\R^d))_{\frac{d}{k+d}, 1}, W^{k+d}_1(\R^d))_{(1,-b),q} \\
		&\hspace{-3cm} = (L_1(\R^d), W^{k+d}_1(\R^d))_{(1,-b),q} = \L^{(k+d,-b)}_{1,q}(\R^d).
	\end{align*}
	Therefore, $\L^{(k+d,-b)}_{1,q}(\R^d) \hookrightarrow  \L_{\infty,q}^{(k,-b)}(\R^d)$.
\end{proof}

\subsection{Embeddings into $\L$}\label{SectionEmbLipClass}

The study of embeddings of smooth function spaces into the Lipschitz class $\L$ has a long history. In particular, it plays a key role in the computation of continuity envelopes of function spaces as can be seen in \cite[Chapters 12 and 14]{Triebel01} and \cite[Chapter 9]{Haroske}. Here we shall only mention that if $1 \leq p \leq \infty$ then
\begin{equation}\label{PrelimEmbLipClass1}
	B^{1+d/p}_{p,q}(\R^d) \hookrightarrow \L(\R^d) \iff 0 < q \leq 1;
\end{equation}
for further extensions of this result, the reader is referred to \cite[Proposition 3.2]{CaetanoHaroske}.
Consequently, if $1 < p < \infty$ then
\begin{equation}\label{PrelimEmbLipClass2}
	H^{1+d/p}_p(\R^d) \quad \text{is not continuously embedded into} \quad \L(\R^d).
\end{equation}
Note that one can circumvent this obstruction using the Br\'ezis-Wainger inequality \cite{BrezisWainger}, which asserts that $H^{1+d/p}_p(\R^d)$ is formed by almost Lipschitz-continuous functions. More precisely, the following embedding holds true
\begin{equation*}
H^{1+d/p}_p(\R^d) \hookrightarrow \L^{(1,-1/p')}_{\infty, \infty}(\R^d), \quad 1 < p< \infty, \quad \frac{1}{p} + \frac{1}{p'} = 1.
\end{equation*}
Moreover, this  result is optimal within the scale of the spaces $\L^{(1,-b)}_{\infty,\infty}(\R^d)$. For further details, as well as generalizations to Besov and Triebel-Lizorkin spaces, we refer to \cite[Theorem 2.1]{EdmundsHaroske}, \cite[Theorem 11.4]{Triebel01} and \cite[Propositions 7.14 and 7.15]{Haroske} and the references therein.

Our next result gives a full characterization of the embeddings from $\L^{(\alpha,-b)}_{p,q}$ into $\L$.

\begin{thm}\label{ThmEmbLipLipClas}
	Let $\alpha > 0, 1 < p < \infty, 0 < q \leq \infty$, and $b > 1/q$. Then,
	\begin{equation*}
		\emph{Lip}^{(\alpha,-b)}_{p,q}(\R^d) \hookrightarrow \emph{Lip}(\R^d) \iff \alpha > 1 + \frac{d}{p}.
	\end{equation*}
\end{thm}
\begin{proof}
	Assume $\alpha > d/p +1$. Let $p < p_1 < \infty$. Then, according to \eqref{ThmFrankeLip1} and \eqref{PrelimEmbLipClass1}, we have
	\begin{equation*}
		\L^{(\alpha,-b)}_{p,q}(\R^d) \hookrightarrow B^{\alpha + d(1/p_1 -1/p), -b + 1/\max\{p,q\}}_{p_1,q}(\R^d) \hookrightarrow B^{1+d/p_1}_{p_1,1}(\R^d)  \hookrightarrow \L(\R^d),
	\end{equation*}
	where the second embedding follows from \eqref{ProofThmEmbLipIntegrability1}.
	
	The converse statement will be shown by contradiction. Suppose that
	\begin{equation}\label{ThmEmbLipLipClas1}
		\L^{(d/p + 1,-b)}_{p,q}(\R^d) \hookrightarrow \L(\R^d).
	\end{equation}
	It is an immediate consequence of \eqref{LipSob} that
	\begin{equation}\label{ThmEmbLipLipClas2}
		H^\alpha_p(\R^d) \hookrightarrow \L^{(\alpha,-b)}_{p,q}(\R^d).
	\end{equation}
	In particular, $H^{d/p +1}_p(\R^d) \hookrightarrow \L(\R^d)$, which is not true because of \eqref{PrelimEmbLipClass2}. Hence, \eqref{ThmEmbLipLipClas1} does not hold. It follows now from Theorem \ref{ThmEmbLipIntegrability} that if $\L^{(\alpha,-b)}_{p,q}(\R^d) \hookrightarrow \L(\R^d)$ then $\alpha > d/p + 1$.
\end{proof}

\subsection{Embeddings into $\text{BV}$}

Some technical problems of the space of functions of bounded variation can be overcome using its relationships with the scale of Besov spaces. See \cite{CohenDahmenDaubechiesDeVore, GLMV, Meyer}. In particular, the following embeddings hold
\begin{equation}\label{EmbBVBesov}
	B^1_{1,1}(\R^d) \hookrightarrow \text{BV}(\R^d) \hookrightarrow B^1_{1, \infty}(\R^d).
\end{equation}
The objective of this section is to characterize embeddings of Lipschitz spaces into $\text{BV}(\R^d) = \L^{(1,0)}_{1,\infty}(\R^d)$. This will complement those embeddings given in \eqref{RemarkThmEmbLipschitz1}.

\begin{thm}\label{ThmEmbBVLip}
	Let $\alpha > 0, 0 < q \leq \infty$ and $b > 1/q$. Then,
	\begin{equation*}
		\emph{Lip}^{(\alpha,-b)}_{1,q}(\R^d) \hookrightarrow \emph{\text{BV}}(\R^d) \iff
                            \alpha > 1.
	\end{equation*}
	The corresponding result for periodic spaces also holds true.
\end{thm}

\begin{proof}
	We claim that
	\begin{equation}\label{ProofThmEmbBVLip1}
	\L^{(\alpha,-b)}_{1,q}(\R^d) \hookrightarrow B^1_{1,1}(\R^d), \quad \alpha > 1.
		\end{equation}
	Indeed, this follows from the trivial embeddings
	\begin{equation*}		\L^{(\alpha,-b)}_{1,q}(\R^d) \hookrightarrow \L^{(\alpha,-b)}_{1,\infty}(\R^d), \quad \alpha > 0, \quad 0 < q < \infty, \quad b > 1/q,
		\end{equation*}
		and
		\begin{equation*}
		\L^{(\alpha,-b)}_{1,\infty}(\R^d) \hookrightarrow B^1_{1,1}(\R^d).
		\end{equation*}
	Combining \eqref{ProofThmEmbBVLip1} and \eqref{EmbBVBesov}, we arrive at $\L^{(\alpha,-b)}_{1,q}(\R^d) \hookrightarrow \text{BV}(\R^d)$.
	
	Let us prove that the condition $\alpha > 1$ is necessary. We shall proceed by contradiction, that is, assume that there exists $\alpha \leq 1$ such that
	\begin{equation}\label{ProofThmEmbBVLip3}
		\L^{(\alpha,-b)}_{1,q}(\R^d) \hookrightarrow \text{BV}(\R^d).
	\end{equation}
	We observe that it is enough to disprove \eqref{ProofThmEmbBVLip3} with $\alpha = 1$ because
	\begin{equation*}
	\L^{(1,-b)}_{1,\infty}(\R^d) \hookrightarrow \L^{(\alpha,-b)}_{1,q}(\R^d), \quad \alpha < 1.
	\end{equation*}
	This embedding is an immediate consequence of the Marchaud inequality for moduli of smoothness
	\begin{equation*}
		\omega_\alpha(f,t)_1 \lesssim t^\alpha \int_t^\infty \frac{\omega_1(f,u)_1}{u^\alpha} \frac{\dint u}{u},
	\end{equation*}
see, e.g., \cite{KolomoitsevTikhonov1}. Assume that \eqref{ProofThmEmbBVLip3} holds with $\alpha =1$. For $\theta \in (0,1)$, we have
	\begin{equation}\label{ProofThmEmbBVLip4}
		(L_1(\R^d), \L^{(1,-b)}_{1,q}(\R^d))_{\theta,q} \hookrightarrow (L_1(\R^d), \text{BV}(\R^d))_{\theta,q}.
	\end{equation}
	Next we compute these interpolation spaces. Concerning the target space, we have
	\begin{equation*}
		(L_1(\R^d), \text{BV}(\R^d))_{\theta,q} = B^\theta_{1,q}(\R^d)
	\end{equation*}
	(cf. \cite[Corollary 11.7]{DominguezTikhonov}). On the other hand, by \eqref{LipLimInter*}, \eqref{LemmaReiteration2} and \eqref{BesovInter9},
	\begin{align*}
	(L_1(\R^d), \L^{(1,-b)}_{1,q}(\R^d))_{\theta,q} &= (L_1(\R^d), (L_1(\R^d), W^1_1(\R^d))_{(1,-b), q})_{\theta,q} \\
	&\hspace{-3cm}= (L_1(\R^d), W^1_1(\R^d))_{\theta,q;\theta (-b + 1/q)} = B^{\theta, \theta (-b + 1/q)}_{1,q}(\R^d).
	\end{align*}
	Therefore, \eqref{ProofThmEmbBVLip4} results in
	\begin{equation*}
	B^{\theta, \theta (-b + 1/q)}_{1,q}(\R^d) \hookrightarrow B^\theta_{1,q}(\R^d),
	\end{equation*}
	which implies $-b + 1/q \geq 0$. By assumptions, this is not possible.
\end{proof}

\section{Characterization of Lipschitz spaces via Fourier transform}

The goal of this section is to obtain the Fourier-analytical description of the spaces $\L^{(\alpha,-b)}_{p,q}(\R^d)$.

\begin{thm}\label{ThmLipFourier}
	Let $\alpha > 0, 1 < p < \infty, 0 < q \leq \infty$ and $b > 1/q$. Then,
	\begin{equation}\label{ThmLipFourier*}
		\|f\|_{\emph{\L}^{(\alpha,-b)}_{p,q}(\R^d)} \asymp \left(\sum_{k=0}^\infty (1 + k)^{- b q} \Big\| \Big(\sum_{j=0}^k 2^{j \alpha 2} |(\varphi_j \widehat{f})^\vee (\cdot)|^2 \Big)^{1/2} \Big\|_{L_p(\R^d)}^q \right)^{1/q}.
	\end{equation}
	The corresponding result for periodic functions also holds true.
\end{thm}

Before providing the proof of this theorem, we make the following remark, which enables us to better understand the Fourier decomposition \eqref{ThmLipFourier*}.

\begin{rem}\label{RemThmLipFourier}
	By the Littlewood-Paley theorem (see, e.g., \cite[Theorem 2.5.6]{Triebel83})
	\begin{equation}\label{RemThmLipFourier1new*}
	\|f\|_{L_p(\R^d)} \asymp \Big\| \Big(\sum_{j=0}^\infty |(\varphi_j \widehat{f})^\vee(\cdot)|^2 \Big)^{1/2}\Big\|_{L_p(\R^d)},
	\end{equation}
	and more generally,
	\begin{equation}\label{RemThmLipFourier1new}
	\|f\|_{H^\alpha_p(\R^d)} \asymp \Big\| \Big(\sum_{j=0}^\infty 2^{j \alpha 2} |(\varphi_j \widehat{f})^\vee(\cdot)|^2 \Big)^{1/2}\Big\|_{L_p(\R^d)}, \quad \alpha \geq 0.
	\end{equation}
	On the other hand, we make the following claim
	\begin{equation}\label{RemThmLipFourier1}
		\|f\|_{B^{\alpha,-b}_{p,q}(\R^d)} \asymp \left(\sum_{k=0}^\infty 2^{k(\alpha-\beta) q} (1 + k)^{- b q} \Big\| \Big(\sum_{j=0}^k 2^{j \beta 2} |(\varphi_j \widehat{f})^\vee (\cdot)|^2 \Big)^{1/2} \Big\|_{L_p(\R^d)}^q \right)^{1/q}, \quad \beta > \alpha.
	\end{equation}
	We postpone the proof of this characterization until a little later, and meanwhile point out certain similarities between Besov, Lebesgue, Sobolev and Lipschitz spaces. Namely, it turns out that the family of norms given by
\begin{equation*}
		\left(\sum_{k=0}^\infty 2^{k(\alpha-\beta) q} (1 + k)^{- b q} \Big\| \Big(\sum_{j=0}^k 2^{j \beta 2} |(\varphi_j \widehat{f})^\vee (\cdot)|^2 \Big)^{1/2} \Big\|_{L_p(\R^d)}^q \right)^{1/q}, \quad \beta \geq \alpha,
	\end{equation*}
	(as usual, the sum should be replaced by the supremum if $q=\infty$) allows us to introduce in a unifying way the Besov spaces (see \eqref{RemThmLipFourier1}), Lipschitz spaces (taking $\beta = \alpha$; see \eqref{ThmLipFourier*}), Lebesgue spaces (taking $\beta = \alpha = 0, q= \infty$ and $b=0$; see \eqref{RemThmLipFourier1new*}) and Sobolev spaces (taking $\beta = \alpha, q= \infty$ and $b=0$; see \eqref{RemThmLipFourier1new}).
	
	
	Now let us show \eqref{RemThmLipFourier1}. We have
	\begin{align}
	 \left(\sum_{k=0}^\infty 2^{k(\alpha-\beta) q} (1 + k)^{- b q} \Big\| \Big(\sum_{j=0}^k 2^{j \beta 2} |(\varphi_j \widehat{f})^\vee (\cdot)|^2 \Big)^{1/2} \Big\|_{L_p(\R^d)}^q \right)^{1/q} \nonumber\\
	   & \hspace{-8cm}  \leq  \left(\sum_{k=0}^\infty 2^{k(\alpha-\beta) q} (1 + k)^{- b q} \Big(\sum_{j=0}^k 2^{j \beta} \|(\varphi_j \widehat{f})^\vee\|_{L_p(\R^d)} \Big)^q \right)^{1/q}. \label{RemThmLipFourier2}
	\end{align}
	We distinguish two possible cases. If $q \geq 1$ then we apply Hardy's inequality (noting that $\beta > \alpha$) to get
	\begin{equation}\label{RemThmLipFourier3}
		\left(\sum_{k=0}^\infty 2^{k(\alpha-\beta) q} (1 + k)^{- b q} \Big(\sum_{j=0}^k 2^{j \beta} \|(\varphi_j \widehat{f})^\vee \|_{L_p(\R^d)} \Big)^q \right)^{1/q} \lesssim  \|f\|_{B^{\alpha,-b}_{p,q}(\R^d)}.
	\end{equation}
	On the other hand, if $q < 1$ then
	\begin{align}
		\left(\sum_{k=0}^\infty 2^{k(\alpha-\beta) q} (1 + k)^{- b q} \Big(\sum_{j=0}^k 2^{j \beta} \|(\varphi_j \widehat{f})^\vee\|_{L_p(\R^d)} \Big)^q \right)^{1/q} \nonumber \\
		 & \hspace{-8cm} \lesssim \left(\sum_{k=0}^\infty 2^{k (\alpha - \beta) q} (1 + k)^{-b q} \sum_{j=0}^k 2^{j \beta q} \|(\varphi_j \widehat{f})^\vee \|_{L_p(\R^d)}^q  \right)^{1/q}  \asymp \|f\|_{B^{\alpha,-b}_{p,q}(\R^d)} \label{RemThmLipFourier4}
	\end{align}
	where we have also used that $\beta > \alpha$ in the last step.
	
	Therefore, it follows from \eqref{RemThmLipFourier2}, \eqref{RemThmLipFourier3} and \eqref{RemThmLipFourier4} that
	\begin{equation*}
		 \left(\sum_{k=0}^\infty 2^{k(\alpha-\beta) q} (1 + k)^{- b q} \Big\| \Big(\sum_{j=0}^k 2^{j \beta 2} |(\varphi_j \widehat{f})^\vee (\cdot)|^2 \Big)^{1/2} \Big\|_{L_p(\R^d)}^q \right)^{1/q} \lesssim \|f\|_{B^{\alpha,-b}_{p,q}(\R^d)}.
	\end{equation*}
	Since the converse estimate holds trivially, we arrive at the desired claim \eqref{RemThmLipFourier1}.

\end{rem}

\begin{proof}[Proof of Theorem \ref{ThmLipFourier}]
	We have
	\begin{equation}\label{ThmLipFourier1}
		\|f\|_{(L_p(\R^d), H^\alpha_p(\R^d))_{(1,-b),q}} \asymp \|((\varphi_j \widehat{f})^\vee(\cdot))\|_{(L_p(\R^d; \ell_2), L_p(\R^d; \ell_2^\alpha))_{(1,-b),q}}.
	\end{equation}
	Indeed, this interpolation formula follows from the well-known fact that $L_p(\R^d)$ and $H^\alpha_p(\R^d)$ are retracts of $L_p(\R^d; \ell_2)$ and $L_p(\R^d; \ell_2^\alpha)$, respectively, with coretraction operator $\mathfrak{J}(f) = ((\varphi_j \widehat{f})^\vee(\cdot))$ (see \cite[p. 185]{Triebel78}), and then invoke the retraction method of interpolation \cite[Sections 1.2.4 and 2.4.1]{Triebel78}.
	
	Now according to \eqref{ThmLipFourier1} and Lemma \ref{LemLimintingVectorLp} with $r=2$ and $A=\mathbb{C}$, we derive
	\begin{equation*}
		\|f\|_{(L_p(\R^d), H^\alpha_p(\R^d))_{(1,-b),q}} \asymp \left(\sum_{k=0}^\infty (1 + k)^{- b q} \Big\| \Big(\sum_{j=0}^k 2^{j \alpha 2} |(\varphi_j \widehat{f})^\vee(\cdot)|^2 \Big)^{1/2} \Big\|_{L_p(\R^d)}^q \right)^{1/q}.
	\end{equation*}
	Hence, the desired result follows from \eqref{LipLimInter}.
\end{proof}

It is well known that Besov spaces can be viewed as retracts of weighted sequence spaces taking values in Lebesgue spaces (cf. \cite{BerghLofstrom} and \cite{Triebel78}). The next result proves that the retraction property also holds for Lipschitz spaces when Lebesgue spaces are replaced by Sobolev spaces.

\begin{thm}\label{ThmRetractLip}
	Let $\alpha > 0, 1 < p < \infty, 0 < q \leq \infty$ and $b > 1/q$. The space $\emph{\L}^{(\alpha,-b)}_{p,q}(\R^d)$ is a retract of $\ell_q (2^{k(-b + 1/q)} H^\alpha_p(\R^d))$, that is, there exist linear operators $\mathfrak{I}: \emph{\L}^{(\alpha,-b)}_{p,q}(\R^d) \to \ell_q (2^{k(-b + 1/q)} H^\alpha_p(\R^d))$ and $\mathfrak{R} : \ell_q (2^{k(-b + 1/q)} H^\alpha_p(\R^d)) \to  \emph{\L}^{(\alpha,-b)}_{p,q}(\R^d)$ such that $\mathfrak{R} (\mathfrak{I} f) = f$. Specifically, these operators are defined by
	\begin{equation*}
		\mathfrak{I} f = \Big(\sum_{j=2^k - 1}^{2^{k+1} -2} (\varphi_j \widehat{f})^\vee \Big)_{k \in \N_0} \quad \text{and} \quad \mathfrak{R} ((f_k)_{k \in \N_0}) = \sum_{k=0}^\infty (\tilde{\varphi}_k \widehat{f_k})^\vee
	\end{equation*}
	where $\varphi_{-1} =0$ and $\tilde{\varphi}_k = \sum_{j=2^k - 2}^{2^{k+1}-1} \varphi_j, \, k \in \N_0$.
\end{thm}

\begin{proof}
	Firstly, we claim that
	\begin{equation}\label{ThmRetractLip1}
	\|f\|_{\L^{(\alpha,-b)}_{p,q}(\R^d)} \asymp  \left(\sum_{k=0}^\infty 2^{k(- b + 1/q) q} \Big\| \sum_{j=2^k-1}^{2^{k+1}-2}  (\varphi_j \widehat{f})^\vee \Big\|_{H^\alpha_p(\R^d)}^q \right)^{1/q}.
	\end{equation}
	Indeed, by monotonicity properties, we can rewrite \eqref{ThmLipFourier*} as
	\begin{equation}\label{ThmRetractLip2}
		\|f\|_{\L^{(\alpha,-b)}_{p,q}(\R^d)} \asymp \left(\sum_{k=0}^\infty 2^{k(- b+ 1/q) q} \Big\| \Big(\sum_{j=0}^{2^k} 2^{j \alpha 2} |(\varphi_j \widehat{f})^\vee (\cdot)|^2 \Big)^{1/2} \Big\|_{L_p(\R^d)}^q \right)^{1/q}.
	\end{equation}
	Thus, using that $\ell_1 \hookrightarrow \ell_2$ and Hardy's inequality (noting that $-b+1/q < 0$), we have
	\begin{align*}
		\|f\|_{\L^{(\alpha,-b)}_{p,q}(\R^d)}
		& \lesssim \left(\sum_{k=0}^\infty 2^{k(- b+ 1/q) q} \bigg( \sum_{l=0}^k \Big\|  \Big(\sum_{j=2^l-1}^{2^{l+1}-2} 2^{j \alpha 2} |(\varphi_j \widehat{f})^\vee (\cdot)|^2 \Big)^{1/2} \Big\|_{L_p(\R^d)}\bigg)^q \right)^{1/q}  \\
		& \lesssim \left(\sum_{k=0}^\infty 2^{k(- b+ 1/q) q} \Big\|  \Big(\sum_{j=2^k-1}^{2^{k+1}-2} 2^{j \alpha 2} |(\varphi_j \widehat{f})^\vee (\cdot)|^2 \Big)^{1/2} \Big\|_{L_p(\R^d)}^q \right)^{1/q}.
	\end{align*}
	Obviously, the converse estimate also holds true (see \eqref{ThmRetractLip2}) and thus
	\begin{equation*}
		\|f\|_{\L^{(\alpha,-b)}_{p,q}(\R^d)}  \asymp  \left(\sum_{k=0}^\infty 2^{k(- b+ 1/q) q} \Big\|  \Big(\sum_{j=2^k-1}^{2^{k+1}-2} 2^{j \alpha 2} |(\varphi_j \widehat{f})^\vee (\cdot)|^2 \Big)^{1/2} \Big\|_{L_p(\R^d)}^q \right)^{1/q}.
	\end{equation*}
	Now, the desired equivalence \eqref{ThmRetractLip1} follows from Littlewood-Paley theorem for Sobolev spaces, that is, $\|f\|_{H^\alpha_p(\R^d)} \asymp \big\|\big(\sum_{j=0}^\infty 2^{j \alpha 2} |(\varphi_j \widehat{f})^\vee (\cdot)|^2 \big)^{1/2} \big\|_{L_p(\R^d)}$.
	
	According to \eqref{ThmRetractLip1}, we infer that $\mathfrak{I}: \L^{(\alpha,-b)}_{p,q}(\R^d) \to \ell_q (2^{k(-b + 1/q)} H^\alpha_p(\R^d))$. Furthermore, using the properties of the system $\{\varphi_j\}_{j \in \N_0}$, the fact that $\tilde{\varphi}_k \sum_{j=2^k - 1}^{2^{k+1}-2} \varphi_j = \sum_{j=2^k - 1}^{2^{k+1}-2} \varphi_j$, and multiplier theorems, we have that
	\begin{align*}
		\|  \mathfrak{R} ((f_k)_{k \in \N_0}) \|_{\L^{(\alpha,-b)}_{p,q}(\R^d)} &  \asymp  \left(\sum_{l=0}^\infty 2^{l(- b + 1/q) q} \Big\| \sum_{j=2^l-1}^{2^{l+1}-2}  \Big(\varphi_j \sum_{k=0}^\infty \tilde{\varphi}_k \widehat{f_k}\Big)^\vee \Big\|_{H^\alpha_p(\R^d)}^q \right)^{1/q} \\
		& \lesssim \left(\sum_{l=0}^\infty 2^{l(- b + 1/q) q}  \Big\| \sum_{j=2^l-1}^{2^{l+1}-2}  (\varphi_j \tilde{\varphi}_l \widehat{f_l})^\vee \Big\|_{H^\alpha_p(\R^d)}^q   \right)^{1/q} \\
		& = \left(\sum_{l=0}^\infty 2^{l(- b + 1/q) q}  \Big\| \big(\sum_{j=2^l-1}^{2^{l+1}-2} \varphi_j  \widehat{f_l}\big)^\vee \Big\|_{H^\alpha_p(\R^d)}^q   \right)^{1/q} \\
		& = \left(\sum_{l=0}^\infty 2^{l(- b + 1/q) q}  \Big\| \big(\sum_{j=2^l-1}^{2^{l+1}-2} \varphi_j  (1 + |x|^2)^{\alpha/2} \widehat{f_l}\big)^\vee \Big\|_{L_p(\R^d)}^q   \right)^{1/q} \\
		& \leq \left(\sum_{l=0}^\infty 2^{l(- b + 1/q) q} \Big\|\sum_{j=2^l-1}^{2^{l+1}-2} \varphi_j  \Big\|_{W^{[d/2] + 1}_2(\R^d)}^q \|f_l\|_{H^\alpha_p(\R^d)}^q  \right)^{1/q} \\
		& \lesssim \|(f_l)_{l \in \N_0}\|_{\ell_q(2^{l(-b+1/q) q} H^\alpha_p(\R^d))}.
	\end{align*}
	Moreover, it is plain to see that $\mathfrak{R} (\mathfrak{I} f) = f$.
\end{proof}
	
We finish this section presenting 
  another characterization of Lipschitz norm in terms of
  Fourier means. Let $\Psi_n(f)$ stand for any  of the following
 means:

1) the $\ell_r$-Fourier means given by
 $$
\widehat{S_{n,r}} f(\xi)=\chi_{\{\xi\in\R^d\,:\,\Vert \xi\Vert_{\ell_r}\le n\}}(\xi) \widehat{f}(\xi),\quad r=1,\infty;
$$

2) the de~la~Vall\'ee Poussin-type means $\eta_{n}f$  (see \cite{KolomoitsevTikhonov*});

3) the Riesz spherical  means $R_n^{\b,\d} f$ given by
$$
\widehat{R_n^{\b,\d} f} (\xi)=\bigg(1- \bigg(\frac{\left|\xi\right|}n\bigg)^\b\bigg)_+^\delta \widehat{f}(\xi)
$$
for $\b>0$ and $\d>d|\frac1p-\frac12 |-\frac12$.

\begin{thm}\label{ThmLipFourier+}
	Let $\alpha > 0, 1 < p < \infty, 0 < q \leq \infty$ and $b > 1/q$. Then,
	\begin{equation*}
		\|f\|_{\emph{\L}^{(\alpha,-b)}_{p,q}(\R^d)} \asymp \|f\|_{L_p(\R^d)} + \left(\sum_{k=0}^\infty (1 + k)^{- b q} \Big\|
 (-\Delta)^{\a/2}{\Psi_{2^k}} f
\Big\|_{L_p(\R^d)}^q \right)^{1/q}.
	\end{equation*}
	The corresponding result for periodic functions
 also holds true.
\end{thm}

The proof immediately  follows  from Hardy's inequalities and the following estimates
  \begin{multline*}
    \(\sum_{k=n+1}^\infty 2^{-k\a\tau}\Vert (-\Delta)^{\a/2}{\Psi_{2^k}} f\Vert_{L_p(\R^d)}^\tau\)^\frac1\tau\lesssim \omega_\a(f,2^{-n})_{p}\\
    \lesssim  \(\sum_{k=n+1}^\infty 2^{-k\a\t}\Vert (-\Delta)^{\a/2}{\Psi_{2^k}} f\Vert_{L_p(\R^d)}^\t\)^\frac1\t,
  \end{multline*}
  where $\tau=\max(2,p)$ and $\t=\min(2,p)$. See
   \cite[Theorem 6.3]{KolomoitsevTikhonov*}.

\section{Characterization of Lipschitz spaces via wavelets}

In order to describe our results, we briefly discuss wavelet bases. For full treatment, we refer the reader to \cite{Daubechies}, \cite{Meyer} and \cite{Triebel08}. As usual, $C^{u}
(\mathbb{R})$ with $u \in \mathbb{N}$ collects all (complex-valued)
continuous functions on $\mathbb{R}$ having continuous bounded
derivatives up to order $u$. Let
\begin{equation}\label{4.1}
\psi_F \in C^{u} (\mathbb{R}), \quad \psi_M \in C^{u} (\mathbb{R}),
\qquad u \in \mathbb{N},
\end{equation}
be real compactly supported Daubechies wavelets with
\begin{equation*}   
\int_{\mathbb{R}} \psi_M (x) \, x^v \, \dint x =0 \quad \text{for all
$v\in \mathbb{N}_0$ with $v<u$.}
\end{equation*}
Recall that $\psi_F$ is called the \emph{scaling function} (\emph{father wavelet})
and $\psi_M$ the \emph{associated wavelet} (\emph{mother wavelet}). The extension
of these wavelets from $\mathbb{R}$ to $\mathbb{R}^d, \, d \geq 2$, is based on the usual tensor procedure. Let
\begin{equation*}   
G = (G_1, \ldots, G_n ) \in G^0 = \{F,M \}^d,
\end{equation*}
which means that $G_r$ is either $F$ or $M$. Let
\begin{equation*}   
G = (G_1, \ldots, G_n ) \in G^j = \{F,M \}^{d *}, \qquad j \in
\mathbb{N},
\end{equation*}
which means that $G_r$ is either $F$ or $M$ where * indicates that
at least one of the components of $G$ must be an $M$. Hence $G^0$
has $2^d$ elements, whereas $G^j$ with $j \in \mathbb{N}$ has $2^d
-1$ elements. Let
\begin{equation}\label{4.2}
\Psi^j_{G,m} (x) = 2^{jd/2} \prod^d_{r=1} \psi_{G_r} (2^j x_r - m_r
), \quad G \in G^j, \quad m \in \mathbb{Z}^d, \quad j \in \N_0.
\end{equation}
We shall assume that $\psi_F$ and
$\psi_M$ in (\ref{4.1}) are normalized with respect to $L_2(\R)$. Then, the system
\begin{equation*}   
\Psi =  \Big\{ \Psi^j_{G,m}: \ j \in \mathbb{N}_0, \ G\in G^j, \ m
\in \mathbb{Z}^d \Big\}
\end{equation*}
is an orthonormal basis in $L_2 (\mathbb{R}^d)$ and
\begin{equation*} 
f = \sum^\infty_{j=0} \sum_{G\in G^j} \sum_{m \in \mathbb{Z}^d}
\lambda^{j,G}_m \, 2^{-jd/2} \, \Psi^j _{G,m}
\end{equation*}
with
\begin{equation}\label{Wav1}   
\lambda^{j,G}_m = \lambda^{j,G}_m (f)=2^{jd/2} \int_{\mathbb{R}^d}
f(x) \, \Psi^j_{G,m} (x) \, \dint x,
\end{equation}
where
$2^{-jd/2} \Psi^j_{G,m}$ are uniformly bounded functions (with
respect to $j$ and $m$).

Under certain conditions on the smoothness parameter $u$ (see \eqref{4.1}), Besov spaces and Lebesgue and Sobolev spaces admit characterizations via wavelet decompositions. Next we introduce the related sequence spaces.

Let  $\chi_{j,m}$ be the characteristic function of the dyadic cube
$Q_{j,m} = 2^{-j}m + 2^{-j}(0,1)^d$ in $\mathbb{R}^d$  with sides of
length $2^{-j}$ parallel to the axes of coordinates and $2^{-j}m$ as
the lower left corner. Let $-\infty < s, \xi < \infty, 1 < p < \infty$ and $0 < q \leq \infty$. The space $b^{s,\xi}_{p,q}$ is the collection of all sequences $\lambda=(\lambda^{j,G}_m)$ with $j \in \mathbb{N}_0, G \in G^j$ and
$m \in \mathbb{Z}^d$ such that
\begin{equation*}
	\|\lambda\|_{b^{s,\xi}_{p,q}} = \left(\sum_{j=0}^\infty 2^{j(s-d/p) q} (1 + j)^{\xi q} \sum_{G \in G^j} \Big(\sum_{m \in \Z^d} |\lambda^{j, G}_m|^p \Big)^{q/p} \right)^{1/q} < \infty
\end{equation*}
with the usual modification if $q=\infty$. We write
$f^s_{p,2}$ for the space of all sequences
$\lambda=(\lambda^{j,G}_m)$ with $j \in \mathbb{N}_0, G \in G^j$ and
$m \in \mathbb{Z}^d$ such that
\begin{equation}\label{Deffspaces}
    \|\lambda\|_{f^s_{p,2}} = \bigg\|\Big(\sum_{j,G,m} 2^{js2} |\lambda^{j,G}_m
    \chi_{j,m}(\cdot)|^2\Big)^{1/2}\bigg\|_{L_p(\R^d)} < \infty.
\end{equation}
It turns out that $f^s_{p,2}$ can be identified with a complemented subspace of $L_p(\R^d; \ell^s_2(\ell_2))$, where
\begin{equation*}
    \|\lambda\|_{\ell^s_2(\ell_2)}= \Big(\sum_{j = 0}^\infty 2^{js2} \sum_{G \in G^j,m \in \Z^d} |\lambda^{j,G}_m|^2\Big)^{1/2} < \infty
\end{equation*}
(see \eqref{DefVectorValuedSeqSpaces}). Indeed, since
\begin{equation}\label{fLp}
	\|\lambda\|_{f^s_{p,2}} = \|(\lambda^{j,G}_m \chi_{j,m} (\cdot))\|_{L_p(\R^d; \ell_2^s(\ell_2))},
\end{equation}
it is plain to check that
\begin{equation*}
	P f (x) = P ((f^{j, G}_m))(x) = \bigg(\Big(2^{j d} \int_{Q_{j,m}} f^{j,G}_m(y) \, \dint y \, \chi_{j,m}(x) \Big)_{\substack {G\in
G^j \\ m\in \mathbb{Z}^d}} \bigg)_{j \in \N_0}, \quad x \in \R^d,
\end{equation*}
defines a projection operator from $L_p(\R^d; \ell^s_2(\ell_2))$ onto a subspace of $L_p(\R^d; \ell^s_2(\ell_2))$, which is isometric to $f^s_{p,2}$.

We are now in a position to state the well-known characterizations of Lebesgue, Sobolev and Besov spaces via wavelets. Recall that $L_p(\R^d) = H^0_p(\R^d)$. For the proof and more general statements covering not only Besov and Triebel-Lizorkin spaces but also Besov spaces of generalized smoothness, the reader is referred to \cite[Theorem 1.20]{Triebel08} and \cite{Almeida}.

\begin{thm}\label{ThmWaveletsLp}
\begin{enumerate}[\upshape(i)]
	\item Let $1 < p < \infty$ and $-\infty < s < \infty$. Assume that \eqref{4.1} holds with $u > |s|$. Then, $f \in H^s_p(\R^d)$ if and only if
	\begin{equation*}
	 f = \sum_{j \in \N_0,G \in G^j,m \in \Z^d} \lambda^{j,G}_m 2^{-jd/2}
    \Psi^j_{G,m},  \quad (\lambda^{j,G}_m) \in f^s_{p,2}
    \end{equation*}
     (unconditional convergence being in $H^s_p(\R^d)$). This representation is unique, that is, the wavelet coefficients $(\lambda^{j, G}_m)$ are given by \eqref{Wav1}, and the operator
     \begin{equation*}
     	I : f \mapsto (\lambda^{j,G}_m)
     \end{equation*}
     defines an isomorphism from $H^s_p(\R^d)$ onto $f^s_{p,2}$.

     \item Let $1 < p < \infty, 0 < q \leq \infty,$ and $- \infty < s, \xi < \infty$. Assume that \eqref{4.1} holds with $u > |s|$. Then, $f \in B^{s,\xi}_{p,q}(\R^d)$ if and only if
	\begin{equation*}
	 f = \sum_{j \in \N_0,G \in G^j,m \in \Z^d} \lambda^{j,G}_m 2^{-jd/2}
    \Psi^j_{G,m},  \quad (\lambda^{j,G}_m) \in b^{s,\xi}_{p,q}
    \end{equation*}
     (unconditional convergence being in $\mathcal{S}'(\R^d)$). This representation is unique, that is, the wavelet coefficients $(\lambda^{j, G}_m)$ are given by \eqref{Wav1}, and the operator
     \begin{equation*}
     	I : f \mapsto (\lambda^{j,G}_m)
     \end{equation*}
     defines an isomorphism from $B^{s,\xi}_{p,q}(\R^d)$ onto $b^{s,\xi}_{p,q}$. If, in addition, $q < \infty$, then $\{\Psi^j_{G,m}\}$ is an unconditional basis in $B^{s,\xi}_{p,q}(\R^d)$.
     \end{enumerate}
\end{thm}	

The goal of this section is to provide the wavelet description of Lipschitz spaces. With this in mind, we define the sequence spaces $\text{lip}^{(\alpha,-b)}_{p,q}$ as follows.

\begin{defn}
	Let $\alpha > 0, 1 < p < \infty, 0 < q \leq \infty$, and $b > 1/q \, (b \geq 0 \text{ if } q=\infty)$. The space $\emph{lip}^{(\alpha,-b)}_{p,q}$ is the collection of all $\lambda = (\lambda^{j, G}_m)$ such that
	\begin{equation}\label{DefLipSeq}
		\|\lambda\|_{\emph{lip}^{(\alpha,-b)}_{p,q}} = \bigg(\sum_{k=0}^\infty (1 + k)^{-b q} \bigg\| \Big(\sum_{j=0}^k \sum_{G \in G^j, m \in \Z^d} 2^{j \alpha 2} |\lambda^{j,G}_m \chi_{j,m}(\cdot)|^2 \Big)^{1/2} \bigg\|_{L_p(\R^d)}^q \bigg)^{1/q} < \infty
	\end{equation}
	(with the usual modification if $q=\infty$).
\end{defn}

It is not hard to show that $f^\alpha_{p,2} \hookrightarrow \text{lip}^{(\alpha,-b)}_{p,q} \hookrightarrow f^0_{p,2}$. In fact, our next result shows that $\text{lip}^{(\alpha,-b)}_{p,q}$ can be characterized in terms of the classical spaces $(f^0_{p,2}, f^\alpha_{p,2})$ via limiting interpolation.

\begin{lem}\label{LemmaWavLip}
	Let $\alpha > 0, 1 < p < \infty, 0 < q \leq \infty$ and $b > 1/q \, (b \geq 0 \, \text{if} \, q=\infty)$. Then,
	\begin{equation*}
		\emph{lip}^{(\alpha,-b)}_{p,q} = (f^0_{p,2}, f^\alpha_{p,2})_{(1,-b),q}.
	\end{equation*}
\end{lem}
\begin{proof}
	For $f = ((f^{j,G}_m)_{\substack {G\in
G^j \\ m\in \mathbb{Z}^d}})_{j \in \N_0} \in L_p(\R^d; \ell_2(\ell_2))$, we can invoke Lemma \ref{LemLimintingVectorLp} to get
	\begin{align*}
		\|f\|_{(L_p(\R^d; \ell_2(\ell_2)), L_p(\R^d; \ell^\alpha_2(\ell_2)))_{(1,-b),q}} & \nonumber \\
		& \hspace{-4cm}\asymp  \left(\sum_{k=0}^\infty (1 + k)^{- b q} \Big\| \Big(\sum_{j=0}^k \sum_{G \in G^j, m \in \Z^d} 2^{j \alpha 2} |f^{j,G}_m(\cdot)|^2  \Big)^{1/2} \Big\|_{L_p(\R^d)}^q \right)^{1/q}.
	\end{align*}
	Since $f^0_{p,2}$ and $f^\alpha_{p,2}$ are isometric to complemented subspaces of $L_p(\R^d; \ell_2(\ell_2))$ and $ L_p(\R^d; \ell^\alpha_2(\ell_2))$, respectively, via \eqref{fLp}, we can apply the theorem on interpolation of complemented subspaces \cite[Theorem 1.17.1]{Triebel78} to obtain
	\begin{align*}
		\|\lambda\|_{(f^0_{p,2}, f^\alpha_{p,2})_{(1,-b),q}}  \\
		&  \hspace{-2.5cm}\asymp  \left(\sum_{k=0}^\infty (1 + k)^{- b q} \Big\| \Big(\sum_{j=0}^k \sum_{G \in G^j, m \in \Z^d} 2^{j \alpha 2} |\lambda^{j,G}_m \chi_{j,m}(\cdot)|^2  \Big)^{1/2} \Big\|_{L_p(\R^d)}^q \right)^{1/q} = \|\lambda\|_{\text{lip}^{(\alpha,-b)}_{p,q}}.
	\end{align*}
\end{proof}

Now we are ready to establish the wavelet decomposition of Lipschitz spaces.

\begin{thm}\label{ThmWaveletsLipschitz}
	Let $\alpha > 0, 1 < p < \infty, 0 < q \leq \infty$ and $b > 1/q$. Assume that \eqref{4.1} holds with $u > \alpha$. Then, $f \in \emph{Lip}^{(\alpha,-b)}_{p,q}(\R^d)$ if and only if
	\begin{equation}\label{ThmWaveletsLipschitz1}
	 f = \sum_{j \in \N_0,G \in G^j,m \in \Z^d} \lambda^{j,G}_m 2^{-jd/2}
    \Psi^j_{G,m}, \quad (\lambda^{j,G}_m) \in \emph{lip}^{(\alpha,-b)}_{p,q}
    \end{equation}
     (unconditional convergence being in $L_p(\R^d)$). This representation is unique, that is, the wavelet coefficients $(\lambda^{j, G}_m)$ are given by \eqref{Wav1}, and the operator
     \begin{equation*}
     	I : f \mapsto (\lambda^{j,G}_m)
     \end{equation*}
     defines an isomorphism from $\emph{Lip}^{(\alpha,-b)}_{p,q}(\R^d)$ onto $\emph{lip}^{(\alpha,-b)}_{p,q}$. If, in addition, $q < \infty$, then $\{\Psi^j_{G,m}\}$ is an unconditional basis in $\emph{Lip}^{(\alpha,-b)}_{p,q}(\R^d)$.
\end{thm}	

\begin{rem}
	Let $\beta > \alpha$. Following similar arguments as those used to show \eqref{RemThmLipFourier1}, one can prove that
	\begin{equation*}
		\|\lambda\|_{b^{\alpha,-b}_{p,q}} \asymp \left(\sum_{k=0}^\infty 2^{k(\alpha-\beta) q} (1 + k)^{-b q} \Big\| \Big(\sum_{j=0}^k \sum_{G \in G^j, m \in \Z^d} 2^{j \beta 2} |\lambda^{j, G}_m \chi_{j,m}(\cdot)|^2 \Big)^{1/2} \Big\|^q_{L_p(\R^d)} \right)^{1/q}.
	\end{equation*}
	Thus, according to Theorem \ref{ThmWaveletsLp}(ii), we derive
	\begin{equation}\label{RemWaveletsLipschitz1}
	\|f\|_{B^{\alpha,-b}_{p,q}(\R^d)} \asymp \left(\sum_{k=0}^\infty 2^{k(\alpha-\beta) q} (1 + k)^{-b q} \Big\| \Big(\sum_{j=0}^k \sum_{G \in G^j, m \in \Z^d} 2^{j \beta 2} |\lambda^{j, G}_m \chi_{j,m}(\cdot)|^2 \Big)^{1/2} \Big\|^q_{L_p(\R^d)} \right)^{1/q}.
	\end{equation}
	
	It is remarkable that Besov, Sobolev, Lebesgue and Lipschitz spaces come together through the family of norms given by
	\begin{equation*}
		 \left(\sum_{k=0}^\infty 2^{k(\alpha-\beta) q} (1 + k)^{-b q} \Big\| \Big(\sum_{j=0}^k \sum_{G \in G^j, m \in \Z^d} 2^{j \beta 2} |\lambda^{j, G}_m \chi_{j,m}(\cdot)|^2 \Big)^{1/2} \Big\|^q_{L_p(\R^d)} \right)^{1/q}, \quad \beta \geq \alpha,
	\end{equation*}
	with the usual modification if $q=\infty$. More explicitly, if $\beta > \alpha$ then we recover  $B^{\alpha,-b}_{p,q}(\R^d)$ (see \eqref{RemWaveletsLipschitz1}); if $\beta = \alpha > 0, q= \infty$ and $b=0$ then we get $H^\alpha_p(\R^d)$ (see Theorem \ref{ThmWaveletsLp}(i)); setting $\beta = \alpha = 0, q= \infty$ and $b=0$ we obtain $L_p(\R^d)$ (see Theorem \ref{ThmWaveletsLp}(i)); and if $\beta = \alpha > 0, 0 < q \leq \infty$ and $b > 1/q$ then we have $\L^{(\alpha,-b)}_{p,q}(\R^d)$ (see Theorem \ref{ThmWaveletsLipschitz}).
\end{rem}

\begin{proof}[Proof of Theorem \ref{ThmWaveletsLipschitz}]
	According to Theorem \ref{ThmWaveletsLp}, $L_p(\R^d)$ and $H^\alpha_p(\R^d)$ are isomorphic to $f^0_{p,2}$ and $f^\alpha_{p,2}$ via $I : f \mapsto (\lambda^{j,G}_m)$. Hence, applying \eqref{LipLimInter} and Lemma \ref{LemmaWavLip}, we derive that
	\begin{equation*}
		I : \text{Lip}^{(\alpha,-b)}_{p,q}(\R^d) = (L_p(\R^d), H^\alpha_p(\R^d))_{(1,-b),q} \to (f^0_{p,2}, f^\alpha_{p,2})_{(1,-b),q} = \text{lip}^{(\alpha,-b)}_{p,q}(\R^d)
	\end{equation*}
	is also an isomorphism.
	
	The uniqueness and the unconditional convergence in $L_p(\R^d)$ of the representation \eqref{ThmWaveletsLipschitz1} are  immediate consequences of the embeddings
	\begin{equation*}
	 \L^{(\alpha,-b)}_{p,q}(\R^d) \hookrightarrow L_p(\R^d) \quad \text{and} \quad \text{lip}^{(\alpha,-b)}_{p,q} \hookrightarrow f^0_{p,2}
	 \end{equation*}
	 together with the corresponding assertion for $L_p(\R^d)$ spaces given in Theorem \ref{ThmWaveletsLp}. The fact that $\{\Psi^j_{G,m}\}$ is an unconditional basis in $\L^{(\alpha,-b)}_{p,q}(\R^d), q < \infty,$ follows easily from \eqref{DefLipSeq}.
\end{proof}

\begin{rem}
	Our method also works with periodic wavelets. Before we state the periodic analogue of Theorem \ref{ThmWaveletsLipschitz}, let us fix some notation. Let $L \in \N$ be fixed such that
	\begin{equation*}
	\Psi^{j, L}_{G,m}(x) = 2^{(j  + L) d/2} \prod^d_{r=1} \psi_{G_r} (2^{j + L} x_r - m_r
)
\end{equation*}
	satisfies
	\begin{equation*}
		\text{supp } \Psi^{j, L}_{G,m} \subset \{x \in \R^d : |x| < 1/2\}
	\end{equation*}
	and let $\{\Psi^{j, L, \text{per}}_{G,m}\}$ be the orthonormal system in $L_2(\T^d)$ obtained from $\{\Psi^{j, L}_{G,m}\}$ via standard periodization arguments. See \cite[Section 1.3.2]{Triebel08} for full details. The periodic counterparts of the sequence spaces $f^\alpha_{p,2}$ (see \eqref{Deffspaces}) and $\text{lip}^{(\alpha,-b)}_{p,q}$ (see \eqref{DefLipSeq}) are introduced as follows. Let
	\begin{equation*}
		\mathbb{P}^d_j = \{m \in \Z^d : 0 \leq m_r < 2^{j+L}\}, \quad j \in \N_0.
	\end{equation*}
	Let $\alpha > 0, 1 < p < \infty, 0 < q \leq \infty$, and $b > 1/q \, (b \geq 0 \text{ if } q=\infty)$. The space $f^{\alpha, \text{per}}_{p,2}$ is formed by all $\lambda = (\lambda^{j,G}_m)$ such that
	\begin{equation*}
    \|\lambda\|_{f^{\alpha, \text{per}}_{p,2}} = \bigg\|\Big(\sum_{j \in \N_0,G \in G^j,m \in \mathbb{P}^d_j} 2^{j \alpha 2} |\lambda^{j,G}_m
    \chi_{j,m}(\cdot)|^2\Big)^{1/2}\bigg\|_{L_p(\T^d)} < \infty
\end{equation*}
	where $\chi_{j,m}$ is the characteristic function of a cube with the left corner $2^{-j-L} m$ and of side-length $2^{-j-L}$. The space $\text{lip}^{(\alpha,-b), \text{per}}_{p,q}$ is the collection of all $\lambda$ such that
	\begin{equation*}
		\|\lambda\|_{\text{lip}^{(\alpha,-b), \text{per}}_{p,q}} = \bigg(\sum_{k=0}^\infty (1 + k)^{-b q} \bigg\| \Big(\sum_{j=0}^k \sum_{G \in G^j, m \in \mathbb{P}^d_j} 2^{j \alpha 2} |\lambda^{j,G}_m \chi_{j,m}(\cdot)|^2 \Big)^{1/2} \bigg\|_{L_p(\T^d)}^q \bigg)^{1/q} < \infty
	\end{equation*}
	(with the usual modification if $q=\infty$).
	
	The periodic counterpart of Theorem \ref{ThmWaveletsLipschitz} reads as follows.
	
\begin{thm}\label{ThmWaveletsLipschitzPer}
	Let $\alpha > 0, 1 < p < \infty, 0 < q \leq \infty$ and $b > 1/q$. Assume that \eqref{4.1} holds with $u > \alpha$. Then, $f \in \emph{Lip}^{(\alpha,-b)}_{p,q}(\T^d)$ if and only if
	\begin{equation*}
	 f = \sum_{j \in \N_0,G \in G^j,m \in \mathbb{P}^d_j} \mu^{j,G}_m 2^{-(j+L)d/2}
    \Psi^{j, L, \emph{per}}_{G,m}, \quad (\mu^{j,G}_m) \in \emph{lip}^{(\alpha,-b), \emph{per}}_{p,q}
    \end{equation*}
     (unconditional convergence being in $L_p(\T^d)$). This representation is unique, that is, the wavelet coefficients $(\mu^{j, G}_m)$ are given by
     \begin{equation*}
     	\mu^{j, G}_m = 2^{(j+L)d/2} \int_{\T^d} f(x) \Psi^{j,L,\emph{per}}_{G,m}(x) \, \dint x,
     \end{equation*}
     and the operator
     \begin{equation*}
     	I : f \mapsto (\mu^{j,G}_m)
     \end{equation*}
     defines an isomorphism from $\emph{Lip}^{(\alpha,-b)}_{p,q}(\T^d)$ onto $\emph{lip}^{(\alpha,-b), \emph{per}}_{p,q}$. If, in addition, $q < \infty$, then $\{\Psi^{j, L, \emph{per}}_{G,m}\}$ is an unconditional basis in $\emph{Lip}^{(\alpha,-b)}_{p,q}(\T^d)$.
\end{thm}	
	
	The proof of Theorem \ref{ThmWaveletsLipschitzPer} follows the same arguments used to show Theorem \ref{ThmWaveletsLipschitz} except for some minor modifications. Specifically, under assumptions of Theorem \ref{ThmWaveletsLipschitzPer}, we have that $I$ is an isomorphism from $L_p(\T^d)$ onto $f^{0, \text{per}}_{p,2}$ and from $H^\alpha_p(\T^d)$ onto $f^{\alpha, \text{per}}_{p,2}$ (see \cite[Theorem 1.37]{Triebel08}). On the other hand, the periodic analogue of Lemma \ref{LemmaWavLip}, that is,
		\begin{equation*}
		\text{lip}^{(\alpha,-b), \text{per}}_{p,q} = (f^{0, \text{per}}_{p,2}, f^{\alpha, \text{per}}_{p,2})_{(1,-b),q},
	\end{equation*}
	also holds true. This is a consequence of the facts that $f^{0, \text{per}}_{p,2}$ and $f^{\alpha, \text{per}}_{p,2}$ are complemented subspaces of $L_p(\T^d; \ell_2(\N_0; \R^{2^{(j + L)d}}))$ and $L_p(\T^d; \ell^\alpha_2(\N_0; \R^{2^{(j + L)d}}))$, respectively, together with the generalization of the periodic counterpart of Lemma \ref{LemLimintingVectorLp} to a family of Banach spaces $(A_j)_{j \in \N_0}$, i.e.,
	\begin{equation}\label{LemLimintingVectorLp*Per}
		\|(f_j)\|_{(L_p(\T^d; \ell_r(A_j)), L_p(\T^d; \ell_r^\alpha(A_j)))_{(1,-b),q}} \asymp  \left(\sum_{k=0}^\infty (1 + k)^{- b q} \Big\| \Big(\sum_{j=0}^k 2^{j \alpha r} \|f_j (\cdot)\|_{A_j}^r \Big)^{1/r} \Big\|_{L_p(\T^d)}^q \right)^{1/q}.
	\end{equation}
	
\end{rem}

\begin{rem}
	The method of proof of Theorem \ref{ThmWaveletsLipschitz} can also be carried out to obtain the characterization of Lipschitz spaces in terms of Haar wavelet bases. We shall not record here the construction of Haar wavelet bases and we refer the interested reader to \cite[Section 2.5.1]{Triebel08} and \cite[Section 2.3]{Triebel10} for further explanations and related literature. Let $\{H^j_{G,m}\}$ be an orthonormal Haar wavelet basis in $L_2(\R^d)$. The characterization for Lipschitz spaces in terms of $\{H^j_{G,m}\}$ reads as follows.
	
\begin{thm}
	Let $1 < p < \infty, 0 < \alpha < \min\{1/p,1/2\}, 0 < q \leq \infty$ and $b > 1/q$. Then, $f \in \emph{Lip}^{(\alpha,-b)}_{p,q}(\R^d)$ if and only if
	\begin{equation*}
	 f = \sum_{j \in \N_0,G \in G^j,m \in \Z^d} \lambda^{j,G}_m 2^{-j d/2}
    H^{j}_{G,m}, \quad (\lambda^{j,G}_m) \in \emph{lip}^{(\alpha,-b)}_{p,q}
    \end{equation*}
     (unconditional convergence being in $L_p(\R^d)$). This representation is unique, that is, the wavelet coefficients $(\lambda^{j, G}_m)$ are given by
     \begin{equation*}
     	\lambda^{j, G}_m = 2^{j d/2} \int_{\R^d} f(x) H^{j}_{G,m}(x) \, \dint x,
     \end{equation*}
     and the operator
     \begin{equation*}
     	I : f \mapsto (\lambda^{j,G}_m)
     \end{equation*}
     defines an isomorphism from $\emph{Lip}^{(\alpha,-b)}_{p,q}(\R^d)$ onto $\emph{lip}^{(\alpha,-b)}_{p,q}$. If, in addition, $q < \infty$, then $\{H^{j}_{G,m}\}$ is an unconditional basis in $\emph{Lip}^{(\alpha,-b)}_{p,q}(\R^d)$.
\end{thm}

The proof of this theorem follows line by line the arguments of the proof of Theorem \ref{ThmWaveletsLipschitz} and is safely left to the reader. However, note that unlike Theorem \ref{ThmWaveletsLipschitz}, we obtain unconditional Haar wavelet bases in $\L^{(\alpha,-b)}_{p,q}(\R^d)$ under the additional assumption $\alpha < \min\{1/p,1/2\}$. This restriction comes from the fact that if $\alpha > 0$ and $1 < p < \infty$ then $\{H^{j}_{G,m}\}$ is an unconditional basis in $H^\alpha_p(\R^d)$, if and only if, $\alpha < \min\{1/p,1/2\}$. See \cite[Corollary 2.23]{Triebel10} and \cite{SeegerUllrich, SeegerUllrich2}.

\end{rem}

\section{Characterizations of Lipschitz spaces for special classes of functions}

\subsection{Lacunary Fourier series} We fix $\psi \in \mathcal{S}(\R^d) \backslash \{0\}$ such that
\begin{equation}\label{LacCondition}
	\text{supp } \widehat{\psi} \subset \{\xi : |\xi| \leq 2\}.
\end{equation}
Define the class
\begin{equation*}
	\mathfrak{L} = \Big\{\psi W : W(x) \sim \sum_{j=3}^\infty a_j e^{i (2^j - 2) x_1} \quad \text{for} \quad x = (x_1, \ldots, x_d) \in \R^d \quad \text{and} \quad (a_j) \subset \mathbb{C}\Big\}.
\end{equation*}
In light of \eqref{LacCondition}, it is plain to check that
\begin{equation*}
	(\varphi_j \widehat{f})^\vee (x) = a_j e^{i (2^j - 2) x_1}  \psi(x), \quad x \in \R^d, \quad f \in \mathfrak{L}.
\end{equation*}
Therefore, according to \eqref{DefBesov} and Theorem \ref{ThmLipFourier}, Besov and Lipschitz norms of functions from the class $\mathfrak{L}$ can be described in terms of the sequence of their Fourier coefficients. More precisely, we have

\begin{thm}\label{ThmBesovLac}
	Let $1 \leq p \leq \infty, 0 < q \leq \infty$ and $-\infty < s, b < \infty$. Then,
	\begin{equation}\label{ThmBesovLac1}
		\|f\|_{B^{s,b}_{p,q}(\R^d)} \asymp  \left(\sum_{j=3}^\infty (2^{j s} (1 + j)^b |a_j|)^q  \right)^{1/q}, \quad f \in \mathfrak{L}.
	\end{equation}
	The corresponding result for periodic functions also holds true.
\end{thm}

\begin{thm}\label{ThmLipLac}
	Let $\alpha>0, 1 < p < \infty, 0 < q \leq \infty$ and $b > 1/q \, (b \geq 0 \text{ if } q=\infty)$. Then,
	\begin{equation}\label{ThmLipLac1}
		\|f\|_{\emph{\L}^{(\alpha,-b)}_{p,q}(\R^d)} \asymp \left(\sum_{k=3}^\infty (1 + k)^{- b q} \Big(\sum_{j=3}^k 2^{j \alpha 2} |a_j|^2 \Big)^{q/2} \right)^{1/q}, \quad f \in \mathfrak{L}.
	\end{equation}
	The corresponding result for periodic functions also holds true.
\end{thm}
\begin{rem}
	The hidden constants in the equivalences \eqref{ThmBesovLac1} and \eqref{ThmLipLac1} depend on $\|\psi\|_{L_p(\R^d)}$.
\end{rem}

As an immediate consequence of Theorems \ref{ThmBesovLac} and \ref{ThmLipLac} we obtain the following

\begin{cor}
Let $\alpha > 0, 1 < p < \infty$ and $b > 1/2$. Then,
\begin{equation*}
	\emph{Lip}^{(\alpha,-b)}_{p,2}(\R^d) \cap \mathfrak{L} = B^{\alpha,-b+1/2}_{p,2}(\R^d) \cap \mathfrak{L}.
\end{equation*}
\end{cor}

\subsection{Monotone functions}

We start by recalling the definition of the general monotone functions given in \cite{LiflyandTikhonov, Tikhonov}. A complex-valued function $\varphi (z), z
>0,$ is called \emph{general monotone}  if it is locally of bounded variation and for some  constant $C > 1$ the following is true
\begin{equation}\label{3.1}
    \int_z^{2z}  |\dint \varphi (u)| \leq C |\varphi (z)|
\end{equation}
for all $z > 0$. Constant $C$ in (\ref{3.1}) is independent of $z$.
The set of all general monotone functions is denoted by $GM$. Basic examples of general monotone functions consist of:
decreasing functions,
 quasi-monotone
functions $\varphi$ (i.e, $\varphi(t) t^{-\alpha}$ is non-increasing for some $\alpha \geq 0$), and increasing functions $\varphi$ such that
$ \varphi(2z) \lesssim \varphi(z)$. Note also that (\ref{3.1}) implies
\begin{equation}\label{3.2}
    |\varphi (u)| \lesssim |\varphi (z)|\quad \text{ for any } \quad z \leq u \leq
    2z.
\end{equation}

It is well known that
the Fourier transform of a radial function $f(x) = f_0(|x|)$ is also radial, $\widehat{f}(\xi)=F_{0}(|\xi|)$
(see \cite[Chapter 4]{SteinWeiss}) and it can be written as the Fourier--Hankel transform
\begin{equation*}
F_{0}(s)= \frac{2 \pi^{d/2}}{\Gamma\left(\frac{d}{2}\right)} \int_{0}^{\infty}f_{0}(t)j_{d/2-1}(st)t^{d-1}\,\dint t,
\end{equation*}
where $j_{\alpha}(t)=\Gamma(\alpha+1)(t/2)^{-\alpha}J_{\alpha}(t)$ is the
normalized Bessel function ($j_{\alpha}(0)=1$), $\alpha\ge -1/2$, with $J_\alpha(t)$ the classical Bessel function of the first kind of order $\alpha$.

Let $\widehat{GM}^d$ be the collection of all radial functions such that the radial component $F_0$ of the Fourier transform of $f$, that is, $\widehat{f}(\xi)= F_0(|\xi|)$, belongs to the class $GM$, is non-negative and satisfies the condition
\begin{equation}\label{3.4new}
    \int_0^1 u^{d-1} F_0(u) \dint u + \int_1^\infty u^{(d-1)/2} |\dint
    F_0(u)| < \infty;
\end{equation}
see \cite{GorbachevTikhonov}.
In other words,
 $\widehat{GM}^{d}$ consists of radial functions
$f(x)=f_{0}(|x|)$, $x\in \mathbb{R}^{d}$, which are defined in terms of the inverse Fourier--Hankel transform
\begin{equation}\label{3.4new+}
f_{0}(z)=\frac{2}{\Gamma\left(\frac{d}{2}\right) (2 \sqrt{\pi})^d}\int_{0}^{\infty}F_{0}(s)j_{d/2-1}(zs)s^{d-1}\, \dint s,
\end{equation}
where the function $F_{0}\in GM$ and satisfies  condition (\ref{3.4new}).
We note that, by  \cite[Lemma 1]{GorbachevLiflyandTikhonov}, the integral in \eqref{3.4new}
converges in the improper sense and therefore $f_{0}(z)$ is continuous for $z>0$.

It turns out that Besov norms of functions from the class $\widehat{GM}^{d}$ can be fully characterized in terms of the growth properties of the Fourier transform. More precisely, the following result was obtained in \cite[Theorem 4.6]{DominguezTikhonov}.

\begin{thm}[\bf{Characterization of Besov norms for $\widehat{GM}^{d}$ functions}]\label{TheoremGMBesov}
	Let $\frac{2d}{d+1} < p < \infty, 0 < q \leq \infty$, and $-\infty < s,b < \infty$. Assume that $f \in \widehat{GM}^d$. Then
	\begin{equation}\label{BesovGM}
		\|f\|_{B^{s,b}_{p,q}(\mathbb{R}^d)} \asymp \left(\int_0^1 t^{d p-d - 1} F_0^p(t) \dint t\right)^{1/p} + \left(\int_1^\infty t^{s q +  dq -d q/p - 1} (1 + \log t)^{b q} F_0^q(t) \dint t\right)^{1/q}.
	\end{equation}
\end{thm}

\begin{rem}\label{RemGMBesov}
	It is well known that Besov spaces with positive smoothness (i.e., $s > 0$) are formed by locally integrable functions. This is illustrated in the above characterization because the finiteness of the right-hand side of (\ref{BesovGM}) implies $f \in L_p(\R^d)$. Indeed, this follows from the Hardy-Littlewood theorem for functions $f \in \widehat{GM}^d$ (see \cite[Theorem 1]{GorbachevLiflyandTikhonov}), which asserts that
	\begin{equation}\label{HLGM}
		\|f\|_{L_p(\R^d)} \asymp \left(\int_0^\infty t^{d p - d -1} F_0^p(t) \dint t \right)^{1/p}, \quad \frac{2 d}{d+1} < p < \infty,
	\end{equation}
	together with the estimate
	\begin{equation}\label{AuxEst}
		 \left(\int_1^\infty t^{d p - d -1} F_0^p(t) \dint t \right)^{1/p} \lesssim \left(\int_1^\infty t^{s q +  dq -d q/p - 1} (1 + \log t)^{b q} F_0^q(t) \dint t\right)^{1/q}, \quad s > 0.
	\end{equation}
\end{rem}

The goal of this section is to establish the corresponding characterization for Lipschitz spaces.

\begin{thm}[\bf{Characterization of Lipschitz norms for $\widehat{GM}^{d}$ functions}]\label{TheoremGMLip}
	Let $\frac{2d}{d+1} < p < \infty, 0 < q \leq \infty, \alpha > 0$, and $b > 1/q \, (b \geq 0 \text{ if } q = \infty)$ . Assume that $f \in \widehat{GM}^d$. Then
	\begin{align}
		\|f\|_{\emph{Lip}^{(\alpha,-b)}_{p,q}(\mathbb{R}^d)} & \asymp \left(\int_0^1 t^{d p-d - 1} F_0^p(t) \dint t\right)^{1/p} \nonumber \\
		&\hspace{1cm}+ \left(\int_1^\infty  (1 + \log t)^{- b q} \left(\int_1^t u^{\alpha p  + d p - d} F_0^p(u) \frac{\dint u}{u} \right)^{q/p} \frac{\dint t}{t}\right)^{1/q}. \label{LipGM}
	\end{align}
\end{thm}

\begin{rem}
(i) Similarly to Remark \ref{RemGMBesov}, if the right-hand side of (\ref{LipGM}) is finite then $f \in L_p(\R^d)$. Indeed, by (\ref{3.2}),	
\begin{align*}
		\left( \int_1^\infty t^{\alpha q + d q - d q/p - 1} (1 + \log t)^{-b q} F_0^q(t) \dint t\right)^{1/q} & \\
		&\hspace{-4cm}\lesssim \left(\int_1^\infty  (1 + \log t)^{- b q} \left(\int_1^t u^{\alpha p  + d p - d} F_0^p(u) \frac{\dint u}{u} \right)^{q/p} \frac{\dint t}{t}\right)^{1/q},
	\end{align*}
	and so, by (\ref{AuxEst}),
	\begin{equation}\label{TheoremGMLip4}
	\left(\int_1^\infty t^{d p - d -1} F_0^p(t) \dint t \right)^{1/p} \lesssim \left(\int_1^\infty  (1 + \log t)^{- b q} \left(\int_1^t u^{\alpha p  + d p - d} F_0^p(u) \frac{\dint u}{u} \right)^{q/p} \frac{\dint t}{t}\right)^{1/q}.
	\end{equation}
	Hence, $f \in L_p(\R^d)$ (see (\ref{HLGM})).

	(ii) In the special case $b=0$ and $q=\infty$ we obtain the characterization of the Sobolev space $H^\alpha_p(\R^d) = \L^{(\alpha,0)}_{p,\infty}(\R^d)$ (see (\ref{LipSob})) given in \cite[Theorem 4.8]{DominguezTikhonov},
	\begin{equation*}
		\|f\|_{H^\alpha_p(\R^d)} \asymp \left(\int_0^1 t^{d p-d - 1} F_0^p(t) \dint t\right)^{1/p}  + \left(\int_1^\infty t^{\alpha p + d p-d - 1} F_0^p(t) \dint t\right)^{1/p}, \quad f \in \widehat{GM}^d.
	\end{equation*}
\end{rem}

\begin{proof}[Proof of Theorem \ref{TheoremGMLip}]
	    We  make use of the following description (see \cite[Corollary 4.1]{GorbachevTikhonov}) of the modulus of smoothness in terms of the Fourier transform for functions in the $\widehat{GM}^d$ class
    \begin{equation*}
        \omega_{\alpha}(f,t)_p \asymp t^{\alpha} \left(\int_0^{1/t} u^{ \alpha p + d p - d} F_0^p(u) \frac{\dint u}{u} \right)^{1/p} + \left(\int_{1/t}^\infty u^{d p - d} F^p_0(u)
        \frac{\dint u}{u}\right)^{1/p}.
    \end{equation*}
    Therefore,
    \begin{equation}\label{TheoremGMLip1}
    	\left(\int_0^{1} (t^{-\alpha} (1 - \log t)^{-b} \omega_\alpha(f,t)_p)^q \frac{\dint t}{t} \right)^{1/q} = I + II
\end{equation}
	 where
	 \begin{equation*}
	I =  \left(\int_0^1 (1-\log t)^{-b q} \left( \int_0^{1/t} u^{ \alpha p + d p - d} F_0^p(u) \frac{\dint u}{u} \right)^{q/p} \frac{\dint t}{t} \right)^{1/q}
	\end{equation*}
	and
	\begin{equation*}
	II = \left(\int_0^1 t^{-\alpha q} (1 - \log t)^{-b q} \left(\int_{1/t}^\infty u^{d p - d} F_0^p(u) \frac{\dint u}{u} \right)^{q/p} \frac{\dint t}{t} \right)^{1/q}.
    \end{equation*}
    A simple change of variables yields that
    \begin{align*}
    	I & =  \left(\int_1^\infty  (1 + \log t)^{- b q} \left(\int_0^t u^{\alpha p  + d p - d} F_0^p(u) \frac{\dint u}{u} \right)^{q/p} \frac{\dint t}{t}\right)^{1/q} \\
	& \asymp \left(\int_0^1 u^{\alpha p  + d p - d} F_0^p(u) \frac{\dint u}{u} \right)^{1/p} \\
	& \hspace{1cm} +  \left(\int_1^\infty  (1 + \log t)^{- b q} \left(\int_1^t u^{\alpha p  + d p - d} F_0^p(u) \frac{\dint u}{u} \right)^{q/p} \frac{\dint t}{t}\right)^{1/q}
    \end{align*}
    where we have also used that $b > 1/q$. Hence, we have
    \begin{align}
    	\left(\int_1^\infty  (1 + \log t)^{- b q} \left(\int_1^t u^{\alpha p  + d p - d} F_0^p(u) \frac{\dint u}{u} \right)^{q/p} \frac{\dint t}{t}\right)^{1/q}  & \lesssim I \nonumber\\
	& \hspace{-7cm} \lesssim  \left(\int_0^1 u^{d p - d} F_0^p(u) \frac{\dint u}{u} \right)^{1/p} \nonumber\\
	&\hspace{-6cm}+  \left(\int_1^\infty  (1 + \log t)^{- b q} \left(\int_1^t u^{\alpha p  + d p - d} F_0^p(u) \frac{\dint u}{u} \right)^{q/p} \frac{\dint t}{t}\right)^{1/q}. \label{TheoremGMLip2}
    \end{align}

   We claim that
    \begin{align}
    	II  & = \left(\int_1^\infty t^{\alpha q} (1 + \log t)^{-b q} \left(\int_t^\infty u^{d p - d} F_0^p(u) \frac{\dint u}{u} \right)^{q/p} \frac{\dint t}{t} \right)^{1/q} \nonumber \\
	& \lesssim \left(\int_1^\infty (1 + \log t)^{-b q} \left(\int_1^{t} u^{\alpha p + d p - d} F_0^p(u) \frac{\dint u}{u} \right)^{q/p} \frac{\dint t}{t} \right)^{1/q}.  \label{TheoremGMLip3}
    \end{align}
    To prove (\ref{TheoremGMLip3}), we shall distinguish two possible cases. Firstly, assume $q \geq p$. Then, applying Hardy's inequality and monotonicity properties of $GM$ functions (see (\ref{3.2})), we get
    \begin{align*}
    	II & \lesssim \left(\int_1^\infty t^{\alpha q + d q - d q/p} (1 + \log t)^{-b q} F_0^q(t) \frac{\dint t}{t} \right)^{1/q} \\
	& \lesssim \left(\int_1^\infty (1 + \log t)^{-b q} \left(\int_1^{t} u^{\alpha p + d p - d} F_0^p(u) \frac{\dint u}{u} \right)^{q/p} \frac{\dint t}{t} \right)^{1/q}.
    \end{align*}
  Suppose now $q < p$. By monotonicity properties,
    \begin{align*}
    	II & \asymp \left(\sum_{i=0}^\infty 2^{i \alpha q} (1 + i)^{-b q} \left(\sum_{j=i}^\infty 2^{j(d - d/p)p} F_0^p(2^j) \right)^{q/p} \right)^{1/q}  \\
	& \leq \left(\sum_{i=0}^\infty 2^{i \alpha q} (1 + i)^{-b q} \sum_{j=i}^\infty 2^{j(d-d/p)q} F_0^q(2^j) \right)^{1/q}  \\
	& \asymp \left(\sum_{j=0}^\infty 2^{j(d-d/p)q} F_0^q(2^j) 2^{j \alpha q} (1 + j)^{-b q} \right)^{1/q}  \\
	& \asymp \left(\int_1^\infty t^{\alpha q + d q - d q/p} (1 + \log t)^{-b q} F_0^q(t) \frac{\dint t}{t} \right)^{1/q} \\
	& \lesssim \left(\int_1^\infty (1 + \log t)^{-b q} \left(\int_1^{t} u^{\alpha p + d p - d} F_0^p(u) \frac{\dint u}{u} \right)^{q/p} \frac{\dint t}{t} \right)^{1/q}.
    \end{align*}

	According to (\ref{TheoremGMLip1}), (\ref{TheoremGMLip2}), (\ref{TheoremGMLip3}), we estimate
	\begin{align}
			\left(\int_1^\infty  (1 + \log t)^{- b q} \left(\int_1^t u^{\alpha p  + d p - d} F_0^p(u) \frac{\dint u}{u} \right)^{q/p} \frac{\dint t}{t}\right)^{1/q}& \nonumber\\
			& \hspace{-7cm}\lesssim \left(\int_0^{1} (t^{-\alpha} (1 - \log t)^{-b} \omega_\alpha(f,t)_p)^q \frac{\dint t}{t} \right)^{1/q}   \label{TheoremGMLip3*}\\
			& \hspace{-7cm} \lesssim  \left(\int_0^1 u^{d p - d} F_0^p(u) \frac{\dint u}{u} \right)^{1/p} \nonumber \\
	&\hspace{-6cm}+  \left(\int_1^\infty  (1 + \log t)^{- b q} \left(\int_1^t u^{\alpha p  + d p - d} F_0^p(u) \frac{\dint u}{u} \right)^{q/p} \frac{\dint t}{t}\right)^{1/q}. \nonumber
	\end{align}
	
	Finally, the desired characterization (\ref{LipGM}) follows from (\ref{DefLip}), (\ref{TheoremGMLip3*}), (\ref{HLGM}) and (\ref{TheoremGMLip4}).
\end{proof}

The next result is a direct consequence of Theorems \ref{TheoremGMBesov} and \ref{TheoremGMLip}.

\begin{cor}
	Let $\alpha > 0, \frac{2 d}{d + 1} < p < \infty$ and $b > 1/p$. Then,
	\begin{equation*}
		\emph{Lip}^{(\alpha,-b)}_{p,p}(\R^d) \cap \widehat{GM}^{d} = B^{\alpha, -b + 1/p}_{p,p}(\R^d) \cap \widehat{GM}^{d}.
	\end{equation*}
\end{cor}

We now turn our attention to characterizations of Besov and Lipschitz norms for periodic functions. This will be done with the help of general monotone sequences. Recall that the sequence $a=\{a_n\}_{n \in \mathbb{N}}$
is called \emph{general monotone}, written $a \in GM$, if there is a
constant $C > 0$ such that
\begin{equation*}
    \sum_{k=n}^{2n-1} |\Delta a_k| \leq C |a_n| \text{ for all } n
    \in \mathbb{N}.
\end{equation*}
Here $\Delta a_k = a_k - a_{k+1}$ and the constant $C$ is
independent of $n$. It is proved in \cite[p. 725]{Tikhonov} that $a \in GM$
if and only if
\begin{equation}\label{3.16}
    |a_\nu| \lesssim |a_n| \text{ for } n \leq \nu \leq 2n
\end{equation}
and
\begin{equation*}
    \sum_{k=n}^N |\Delta a_k| \lesssim |a_n| + \sum_{k=n+1}^N
    \frac{|a_k|}{k} \text{ for any } n < N.
\end{equation*}


The characterization of Besov norms for Fourier series with monotone coefficients was obtained in \cite[Theorem 4.22]{DominguezTikhonov}. It reads as follows.

\begin{thm}[\bf{Characterization of Besov norms for Fourier series with $GM$ coefficients}]\label{TheoremGMBesovPer}
	Let $1 < p < \infty, 0 < q \leq \infty$, and $-\infty < s,b < \infty$.
Let the Fourier series of
    $f \in L_1(\mathbb{T})$ be given by
	 \begin{equation*}
	 f(x) \sim \sum_{n=1}^\infty (a_n \cos n x + b_n \sin nx)
	 \end{equation*}
    where  $\{a_n\}_{n \in \mathbb{N}}, \{b_n\}_{n \in \mathbb{N}}$ are  non-negative general monotone sequences.
 Then
	\begin{equation}\label{BesovGMPer}
		\|f\|_{B^{s,b}_{p,q}(\mathbb{T})} \asymp \left(\sum_{n=1}^\infty n^{sq + q -q/p -1} (1 + \log n)^{b q} (a_{n}^q + b_n^q)\right)^{1/q}.
	\end{equation}
\end{thm}

\begin{rem}\label{RemGMBesovPer}
	The Hardy-Littlewood theorem for Fourier series whose sequence of Fourier coefficients is non-negative general monotone was obtained in \cite[Theorem 4.2]{Tikhonov}. Namely, if
	 \begin{equation*}
	 f(x) \sim \sum_{n=1}^\infty (a_n \cos n x + b_n \sin nx), \quad \{a_n\}_{n \in \N}, \{b_n\}_{n \in \N} \in GM, \quad a_n, b_n \geq 0,
	 \end{equation*}
	 then
	 \begin{equation}\label{HLGMPer}
	 	\|f\|_{L_p(\T)} \asymp \left(\sum_{n=1}^\infty n^{p-2} (a_n^p + b_n^p) \right)^{1/p}, \quad 1 < p < \infty.
	 \end{equation}
	
	Let $s > 0$. It is plain to check that, for any positive $q$,
	\begin{equation}\label{AuxEstPer}
		 \left(\sum_{n=1}^\infty n^{p-2} (a_n^p + b_n^p) \right)^{1/p} \lesssim  \left(\sum_{n=1}^\infty n^{sq + q -q/p -1} (1 + \log n)^{b q} (a_{n}^q + b_n^q)\right)^{1/q}.
	\end{equation}
	Consequently, if the right-hand side of (\ref{BesovGMPer}) is finite then $f \in L_p(\T)$ (see (\ref{HLGMPer})).
\end{rem}

Next we establish the periodic counterpart of Theorem \ref{TheoremGMLip}.

\begin{thm}[\bf{Characterization of Lipschitz norms for Fourier series with $GM$ coefficients}]\label{TheoremGMLipPer}
		Let $1 < p < \infty, 0 < q \leq \infty, \alpha > 0$, and $b > 1/q \, (b \geq 0 \text{ if } q=\infty)$.
Let the Fourier series of
    $f \in L_1(\mathbb{T})$ be given by
	 \begin{equation*}
	 f(x) \sim \sum_{n=1}^\infty (a_n \cos n x + b_n \sin nx)
	 \end{equation*}
    where  $\{a_n\}_{n \in \mathbb{N}}, \{b_n\}_{n \in \mathbb{N}}$ are  non-negative general monotone sequences. Then
	\begin{equation}\label{LipGMPer}
		\|f\|_{\emph{Lip}^{(\alpha,-b)}_{p,q}(\T)}  \asymp \left(\sum_{n=1}^\infty  (1 + \log n)^{- b q} \left(\sum_{k=1}^n k^{\alpha p  + p - 2} (a_k^p + b_k^p) \right)^{q/p} \frac{1}{n}\right)^{1/q}.
	\end{equation}
\end{thm}

\begin{rem}
(i) If the right-hand side of (\ref{LipGMPer}) is finite then $f \in L_p(\T)$. Indeed, by monotonicity properties (see (\ref{3.16})),	
\begin{align}
		\left(\sum_{n=1}^\infty n^{\alpha q + q -q/p-1} (1 + \log n)^{-b q} (a_n^q + b_n^q) \right)^{1/q} & \nonumber \\
		&\hspace{-4cm}\lesssim \left(\sum_{n=1}^\infty  (1 + \log n)^{- b q} \left(\sum_{k=1}^n k^{\alpha p  + p - 1} (a_k^p + b_k^p) \frac{1}{k} \right)^{q/p} \frac{1}{n}\right)^{1/q} \label{AuxPer*}
	\end{align}
	and so, by (\ref{AuxEstPer}),
	\begin{equation}\label{TheoremGMLip4Per*}
	\left(\sum_{n=1}^\infty n^{p-2} (a_n^p + b_n^p) \right)^{1/p} \lesssim   \left(\sum_{n=1}^\infty  (1 + \log n)^{- b q} \left(\sum_{k=1}^n k^{\alpha p  + p - 1} (a_k^p + b_k^p) \frac{1}{k} \right)^{q/p} \frac{1}{n}\right)^{1/q}.
	\end{equation}
	Therefore, $f \in L_p(\T)$ (see (\ref{HLGMPer})).

	(ii) Setting $b=0$ and $q=\infty$ in (\ref{LipGMPer}) we recover the characterization of the periodic Sobolev space $H^\alpha_p(\T) = \L^{(\alpha,0)}_{p,\infty}(\T)$ given in \cite[Theorem 4.25]{DominguezTikhonov},
	\begin{equation*}
		\|f\|_{H^\alpha_p(\T)} \asymp  \left(\sum_{n=1}^\infty n^{\alpha p  + p - 2} (a_n^p + b_n^p) \right)^{1/p}, \quad \{a_n\}_{n \in \N} \in GM.
	\end{equation*}
\end{rem}

\begin{proof}[Proof of Theorem \ref{TheoremGMLipPer}]
	We shall rely on the following characterization of the moduli of smoothness for Fourier series with general monotone coefficients
	\begin{equation}\label{ModRealPer}
		\omega_\alpha\left(f, \frac{1}{n}\right)_p \asymp n^{-\alpha} \left(\sum_{k=1}^n k^{\alpha p + p - 2} (a_k^p + b_k^p) \right)^{1/p} + \left(\sum_{k= n+1}^\infty k^{p - 2} (a_k^p + b_k^p) \right)^{1/p}.
	\end{equation}
	See \cite[Theorem 6.2]{Tikhonov}.
	
	  In light of (\ref{ModRealPer}),
	    \begin{equation}\label{TheoremGMLip1Per}
    	\left(\int_0^{1} (t^{-\alpha} (1 - \log t)^{-b} \omega_\alpha(f,t)_p)^q \frac{\dint t}{t} \right)^{1/q}  \asymp   I + II,
\end{equation}
	 where
	 \begin{equation*}
	I =  \left(\sum_{n=1}^\infty (1 + \log n)^{-b q} \left(\sum_{k=1}^n k^{\alpha p + p - 2} (a_k^p + b_k^p) \right)^{q/p} \frac{1}{n}\right)^{1/q}
	\end{equation*}
	and
	\begin{equation*}
	II = \left(\sum_{n=1}^\infty n^{\alpha q} (1 + \log n)^{-b q} \left(\sum_{k= n}^\infty k^{p - 2} (a_k^p + b_k^p) \right)^{q/p}   \frac{1}{n} \right)^{1/q}.
    \end{equation*}

    Obviously, we have
    \begin{equation}\label{TheoremGMLip2Per}
    	I \lesssim \|f\|_{\L^{(\alpha,-b)}_{p,q}(\T)}.
    \end{equation}
    Conversely, if $q \geq p$ we can apply Hardy's inequality to get
    \begin{equation}\label{TheoremGMLip3Per}
    	II \lesssim \left(\sum_{n=1}^\infty n^{\alpha q + q -q/p -1} (1 + \log n)^{-b q} (a_n^q + b_n^q) \right)^{1/q} \lesssim I,
    \end{equation}
    where the last estimate follows from (\ref{AuxPer*}).

    Next we deal with $q < p$. Since
    \begin{equation*}
    	k^{1-1/p} (a_k + b_k) \lesssim \left(\sum_{j=k}^\infty j^{(1-1/p) q -1} (a_j^q + b_j^q) \right)^{1/q}
    \end{equation*}
    (see (\ref{3.16})), we have
    \begin{equation*}
    	k^{(1-1/p)(p-q)} (a_k + b_k)^{p-q} \lesssim \left(\sum_{j=k}^\infty j^{(1-1/p) q -1} (a_j + b_j)^q \right)^{(p-q)/q},
    \end{equation*}
   which yields that
    \begin{equation*}
    	\left(\sum_{k=n}^\infty k^{(1-1/p) p - 1} (a_k^p + b_k^p) \right)^{1/p} \lesssim \left(\sum_{k=n}^\infty k^{(1-1/p) q - 1} (a_k^q + b_k^q) \right)^{1/q}.
    \end{equation*}
   Changing the order of summation and applying (\ref{AuxPer*}), we derive
    \begin{align}
    	II &\lesssim \left(\sum_{n=1}^\infty n^{\alpha q} (1 + \log n)^{-b q} \sum_{k=n}^\infty k^{(1 - 1/p)q-1}(a_k^q + b_k^q) \frac{1}{n} \right)^{1/q}  \nonumber \\
	&\asymp \left(\sum_{k=1}^\infty k^{(\alpha + 1-1/p) q-1} (1 + \log k)^{-b q} (a_k^q + b_k^q)\right)^{1/q} \lesssim I. \label{TheoremGMLip4Per}
    \end{align}
    Hence, by (\ref{TheoremGMLip3Per}) and (\ref{TheoremGMLip4Per}),
    \begin{equation}\label{TheoremGMLip5Per}
    II \lesssim I.
    \end{equation}
    Combining (\ref{HLGMPer}), (\ref{TheoremGMLip4Per*}), (\ref{TheoremGMLip1Per}) and (\ref{TheoremGMLip5Per}) we arrive at $	 \|f\|_{\L^{(\alpha,-b)}_{p,q}(\T)} \lesssim I$ and so, by (\ref{TheoremGMLip2Per}),
    \begin{equation*}
    \|f\|_{\L^{(\alpha,-b)}_{p,q}(\T)} \asymp I,
    \end{equation*}
    that is, (\ref{LipGMPer}) holds.

\end{proof}

\section{Optimality}		
	
	Our first goal is to show the optimality of the embeddings given in Theorem \ref{ThMFrankeLip}. More precisely, we obtain the following
		
		\begin{thm}\label{ThMFrankeLipSharpness}
			Let $\alpha > 0, \frac{2 d}{d + 1} <  p_0 < p < p_1 < \infty, 0 < q, r \leq \infty, b > 1/q$ and $-\infty < \xi < \infty$. Then,
			\begin{equation}\label{ThMFrankeLipSharpness1}
				B^{\alpha + d(1/p_0 - 1/p), -b + \xi}_{p_0,q}(\R^d) \hookrightarrow \emph{Lip}^{(\alpha,-b)}_{p,q}(\R^d) \iff \xi \geq 1/\min\{p,q\},
			\end{equation}
			\begin{equation}\label{ThMFrankeLipSharpness2}
			  \emph{Lip}^{(\alpha,-b)}_{p,q}(\R^d) \hookrightarrow B^{\alpha + d(1/p_1 -1/p), -b +\xi}_{p_1,q}(\R^d) \iff \xi \leq 1/\max\{p,q\},\qquad
			\end{equation}
			\begin{equation}\label{ThMFrankeLipSharpness3}
				B^{\alpha + d(1/p_0 - 1/p), -b + 1/q}_{p_0,r}(\R^d) \hookrightarrow \emph{Lip}^{(\alpha,-b)}_{p,q}(\R^d) \iff r \leq \min\{p,q\},
			\end{equation}
			\begin{equation}\label{ThMFrankeLipSharpness4}
				 \emph{Lip}^{(\alpha,-b)}_{p,q}(\R^d) \hookrightarrow B^{\alpha + d(1/p_1 -1/p), -b + 1/q}_{p_1,r}(\R^d) \iff r \geq \max\{p,q\}.
			\end{equation}
		\end{thm}
		\begin{proof}
		We shall consider four cases.
		
	\textsc{Case 1:}
		Regarding (\ref{ThMFrankeLipSharpness1}), we will show that $\xi = 1/\min\{p, q\}$ is the best possible loss of logarithmic smoothness for which the following embedding holds
		\begin{equation*}
			B^{\alpha + d(1/p_0 - 1/p), -b + \xi}_{p_0,q}(\R^d) \hookrightarrow \L^{(\alpha,-b)}_{p, q}(\R^d).
		\end{equation*}
		We will proceed by contradiction, i.e., suppose that there exists $\varepsilon > 0$ such that
		\begin{equation}\label{ThmEmbLipschitz4new}
			B^{\alpha + d(1/p_0 - 1/p), -b + 1/\min\{p,q\} - \varepsilon}_{p_0,q}(\R^d) \hookrightarrow \L^{(\alpha,-b)}_{p,q}(\R^d).
		\end{equation}
		First, we assume $q \leq p$. Let $\theta \in (0,1)$. According to (\ref{FrankeMarschall}), (\ref{ThmEmbLipschitz4new}) and the interpolation property,
		\begin{equation}\label{ThmEmbLipschitz4new*}
			(B^{d(1/p_0 - 1/p)}_{p_0,p}(\R^d) , B^{\alpha + d(1/p_0 - 1/p), -b + 1/q - \varepsilon}_{p_0,q}(\R^d))_{\theta,q} \hookrightarrow (L_p(\R^d),  \L^{(\alpha,-b)}_{p, q}(\R^d))_{\theta,q}.
		\end{equation}
		Next we identify these interpolation spaces. Concerning the target space, we can apply (\ref{LipLimInter}), (\ref{LemmaReiteration2}) and (\ref{InterBesSobLp2}) to obtain
		\begin{align}
			(L_p(\R^d),  \L^{(\alpha,-b)}_{p,q}(\R^d))_{\theta,q} & = (L_p(\R^d), (L_p(\R^d), H^\alpha_p(\R^d))_{(1,-b),q})_{\theta , q} \nonumber \\
			& = (L_p(\R^d), H^\alpha_p(\R^d))_{\theta, q; \theta (-b + 1/q)} = B^{\theta \alpha, \theta(-b + 1/q)}_{p, q}(\R^d). \label{ThmEmbLipschitz4new**}
		\end{align}
		On the other hand, applying (\ref{InterBes}), one gets
		\begin{equation}\label{ThmEmbLipschitz4new***}
			(B^{d(1/p_0 - 1/p)}_{p_0,p}(\R^d) , B^{\alpha + d(1/p_0 - 1/p), -b + 1/q - \varepsilon}_{p_0,q}(\R^d))_{\theta,q}  = B^{\theta \alpha + d(1/p_0 -1/p), \theta (-b + 1/q) - \theta \varepsilon}_{p_0,q}(\R^d).
		\end{equation}
		Hence, by (\ref{ThmEmbLipschitz4new*}), (\ref{ThmEmbLipschitz4new**}) and (\ref{ThmEmbLipschitz4new***}), we get
		\begin{equation*}
			B^{\theta \alpha + d(1/p_0 -1/p), \theta (-b + 1/q) - \theta \varepsilon}_{p_0,q}(\R^d) \hookrightarrow B^{\theta \alpha, \theta(-b + 1/q)}_{p, q}(\R^d),
		\end{equation*}
		which is not true for $\varepsilon > 0$ (see \cite[Theorem 1]{Leopold} and \cite[Remark 6.4]{DominguezTikhonov}).
		
		 Assume now $q > p$. We let $\beta$ such that $-b + 1/p +1/q - \varepsilon < \beta <- b +1/p + 1/q$. Define
\begin{equation*}
		 F_0(t) = \left\{\begin{array}{lcl}
                            1 & ,  & 0 < t  < 1, \\
                            & & \\
                            t^{-\alpha + d/p -d} (1 + \log t)^{-\beta} & , & t \geq 1,
            \end{array}
            \right.
	\end{equation*}
	and $\widehat{f}(\xi) = F_0(|\xi|)$. Note that $F_0 \in GM$ and the radial component of $f$ is given by (\ref{3.4new+}). According to (\ref{BesovGM}),
	\begin{align*}
		\|f\|_{B^{\alpha + d(1/p_0-1/p),-b+1/p -\varepsilon}_{p_0,q}(\mathbb{R}^d)} &\\
		& \hspace{-4cm} \asymp\left(\int_0^1 t^{d p_0-d - 1} \dint t\right)^{1/p_0}  + \left(\int_1^\infty  (1 + \log t)^{(-b + 1/p -\varepsilon-\beta) q} \frac{\dint t}{t}\right)^{1/q} < \infty.
	\end{align*}
		On the other hand, since
		 \begin{align*}
		 \int_1^\infty  (1 + \log t)^{- b q} \left(\int_1^t u^{\alpha p  + d p - d} F_0^p(u) \frac{\dint u}{u} \right)^{q/p} \frac{\dint t}{t} & \asymp \int_1^\infty (1 + \log t)^{(- b-\beta + 1/p) q } \frac{\dint t}{t} = \infty,
		 \end{align*}
		 it follows from (\ref{LipGM}) that $f \not \in \L^{(\alpha,-b)}_{p,q}(\R^d)$. This contradicts (\ref{ThmEmbLipschitz4new}).
		
	\textsc{Case 2:}
Let us show (\ref{ThMFrankeLipSharpness2}), that is, $\xi = 1/\max\{p, q\}$ is the optimal exponent such that the following holds true
		\begin{equation*}
			\L^{(\alpha,-b)}_{p,q}(\R^d) \hookrightarrow B^{\alpha + d(1/p_1 -1/p), -b +\xi}_{p_1,q}(\R^d).
		\end{equation*}
		We argue by contradiction. Suppose that there exists $\varepsilon > 0$ such that
		\begin{equation}\label{ThmEmbLipschitz4new'}
		 \L^{(\alpha,-b)}_{p,q}(\R^d) \hookrightarrow B^{\alpha + d(1/p_1 -1/p), -b +1/\max\{p,q\} + \varepsilon}_{p_1,q}(\R^d).
		\end{equation}
		
		Let $q \geq p$ and $\theta \in (0,1)$. Since $L_p(\R^d) \hookrightarrow B^{d(1/p_1 -1/p)}_{p_1,p}(\R^d)$ (see the right-hand side embedding of (\ref{FrankeMarschall}) with $\alpha = 0$), we have
		\begin{equation*}
			(L_p(\R^d),  \L^{(\alpha,-b)}_{p, q}(\R^d))_{\theta,q} \hookrightarrow (B^{d(1/p_1 -1/p)}_{p_1,p}(\R^d), B^{\alpha + d(1/p_1 -1/p), -b +1/q + \varepsilon}_{p_1,q}(\R^d))_{\theta,q},
		\end{equation*}
		or equivalently (see (\ref{ThmEmbLipschitz4new**}) and (\ref{ThmEmbLipschitz4new***})),
		\begin{equation*}
			B^{\theta \alpha, \theta(-b + 1/q)}_{p, q}(\R^d) \hookrightarrow B^{\theta \alpha + d(1/p_1-1/p), \theta (-b + 1/q) + \theta \varepsilon}_{p_1,q}(\R^d),
		\end{equation*}
		which fails to be true because $\varepsilon > 0$ (see \cite[Theorem 1]{Leopold} and \cite[Remark 6.4]{DominguezTikhonov}). Therefore, (\ref{ThmEmbLipschitz4new'}) does not hold.
		
	Suppose now (\ref{ThmEmbLipschitz4new'}) with $q < p$. Let
\begin{equation*}
		 F_0(t) = \left\{\begin{array}{lcl}
                            1 & ,  & 0 < t  < 1, \\
                            & & \\
                            t^{-\alpha + d/p -d} (1 + \log t)^{-\beta} & , & t \geq 1,
            \end{array}
            \right.
	\end{equation*}
	where $-b + 1/p +1/q < \beta < \min\{1/p, -b + 1/p +1/q + \varepsilon\}$ and $\widehat{f}(\xi) = F_0(|\xi|)$. Then, according to (\ref{LipGM}),
	\begin{equation*}
		\|f\|_{\L^{(\alpha,-b)}_{p,q}(\mathbb{R}^d)}  \asymp \left(\int_0^1 t^{d p-d - 1} \dint t\right)^{1/p} + \left(\int_1^\infty  (1 + \log t)^{(- b -\beta +1/p) q} \frac{\dint t}{t}\right)^{1/q} < \infty,
	\end{equation*}
	but $f \not \in B^{\alpha + d(1/p_1 -1/p), -b +1/p + \varepsilon}_{p_1,q}(\R^d)$ because
	\begin{equation*}
		\int_1^\infty (1 + \log t)^{(-b + 1/p + \varepsilon -\beta) q} \frac{\dint t}{t} = \infty
	\end{equation*}
	(see (\ref{BesovGM})).
		
	\textsc{Case 3:}
To show optimality in (\ref{ThMFrankeLipSharpness3}) we will argue by contradiction, that is, assume that there exists $r > \min\{p,q\}$ such that
		\begin{equation}\label{ThmEmbLipschitz4new''}
				B^{\alpha + d(1/p_0 - 1/p), -b + 1/q}_{p_0,r}(\R^d) \hookrightarrow \L^{(\alpha,-b)}_{p,q}(\R^d).
			\end{equation}
			We shall distinguish two cases. Firstly, suppose that $p \leq q$ and so, $r > p$. Consider
			 \begin{equation*}
		 F_0(t) = \left\{\begin{array}{lcl}
                            1 & ,  & 0 < t  < 1, \\
                            & & \\
                            t^{-\alpha + d/p -d} (1 + \log t)^{-\beta} & , & t \geq 1,
            \end{array}
            \right.
            \end{equation*}
            where $-b + 1/q + 1/r < \beta < -b + 1/q + 1/p$. In virtue of (\ref{BesovGM}), we get
            \begin{align*}
            	\|f\|_{B^{\alpha + d(1/p_0 - 1/p), -b + 1/q}_{p_0,r}(\R^d)} & \\
	& \hspace{-3cm}\asymp \left(\int_0^1 t^{d p_0 - d -1} \dint t \right)^{1/p_0} + \left(\int_1^\infty (1 + \log t)^{(-b + 1/q -\beta) r} \frac{\dint t}{t} \right)^{1/r} < \infty.
            \end{align*}
            On the other hand, since
            \begin{equation*}
            	\int_1^\infty (1 + \log t)^{(-b - \beta + 1/p) q} \frac{\dint t}{t} = \infty,
            \end{equation*}
            one can invoke (\ref{LipGM}) to derive $f \not \in \L^{(\alpha,-b)}_{p,q}(\R^d)$. Hence, (\ref{ThmEmbLipschitz4new''}) does not hold.

            Secondly, assume $p > q$ and thus, $r > q$. Since the scale of Besov spaces is increasing with respect to the fine index, i.e.,
            \begin{equation*}
            	B^{\alpha + d(1/p_0 - 1/p), -b + 1/q}_{p_0,r_0}(\R^d) \hookrightarrow 	B^{\alpha + d(1/p_0 - 1/p), -b + 1/q}_{p_0,r_1}(\R^d), \quad r_0 < r_1,
            \end{equation*}
             we may assume, without loss of generality, that $q < r < p$. For each $\{a_n\}_{n \in \N}$ non-negative non-increasing sequence, we let $F_0$ be the step function with $F_0(n) = n^{-\alpha -d + d/p +1/p} a_n^{1/p}, \, n \in \N$ and $F_0(t) = a_1, \, t \in (0,1)$. Clearly, $F_0 \in GM$. It follows from (\ref{ThmEmbLipschitz4new''}), (\ref{BesovGM}) and (\ref{LipGM}) that
             \begin{align*}
              \left(\int_1^\infty  (1 + \log t)^{- b q} \left(\int_1^t u^{\alpha p  + d p - d} F_0^p(u) \frac{\dint u}{u} \right)^{q/p} \frac{\dint t}{t}\right)^{1/q} &\lesssim \left(\int_0^1 t^{d p_0-d - 1} F_0^{p_0}(t) \dint t\right)^{1/p_0}   \\
              &\hspace{-7cm}+ \left(\int_1^\infty t^{\alpha r +  dr -d r/p - 1} (1 + \log t)^{ (-b+1/q) r} F_0^r(t) \dint t\right)^{1/r},
             \end{align*}
             which implies, by monotonicity properties,
             \begin{equation}\label{ThmEmbLipschitz4new''*}
             	\left(\sum_{n=1}^\infty (1 + \log n)^{-b q} \left(\sum_{k=1}^n  a_k \right)^{q/p} \frac{1}{n} \right)^{1/q}  \lesssim \left(\sum_{n=1}^\infty n^{r/p-1} (1 + \log n)^{(-b + 1/q) r} a_n^{r/p}  \right)^{1/r}.
             \end{equation}
             Note that (\ref{ThmEmbLipschitz4new''*}) is a special case of Hardy-type inequalities for non-increasing sequences, which have been completely characterized  in \cite{BennettGrosseErdmann}. In particular, the following result can be found in \cite[Theorem 2]{BennettGrosseErdmann}.

             \begin{lem}\label{LemBeGE}
             	Let $\{\lambda_n\}_{n \in \N}, \{\gamma_n\}_{n \in \N}$ be non-negative sequences and let $0 < u < v \leq 1$. Set $\frac{1}{w} = \frac{1}{u} - \frac{1}{v}$ and $\Gamma_n = \sum_{k=1}^n \gamma_k, \, n \in \N$. Then, the inequality holds
	\begin{equation}\label{BeGE1}
		\left(\sum_{n=1}^\infty \lambda_n \left(\sum_{k=1}^n a_k \right)^{u} \right)^{1/u} \lesssim \left(\sum_{n=1}^\infty \gamma_n a_n^{v} \right)^{1/v}
	\end{equation}
	for all non-negative non-increasing sequences $\{a_n\}_{n \in \N}$, if and only if
	\begin{equation}\label{BeGE2}
		\sum_{n=1}^\infty n^{u} \lambda_n \left(\frac{1}{\Gamma_n} \sum_{k=1}^n k^{u} \lambda_k \right)^{w/v} < \infty
	\end{equation}
	and
	\begin{equation}\label{BeGE3}
		\sum_{n=1}^\infty \lambda_n \left(\sum_{k=n}^\infty \lambda_k \right)^{w/v} \max_{k \leq n} \left(\frac{k^{v}}{\Gamma_k} \right)^{w/v} < \infty.
	\end{equation}
             \end{lem}		
		Inequality (\ref{ThmEmbLipschitz4new''*}) is a particular case of (\ref{BeGE1}) with
		\begin{align*}
		\lambda_n = (1 + \log n)^{-b q} n^{-1}, &\quad  \gamma_n = (1 + \log n)^{(-b + 1/q) r} n^{r/p-1},\\
			u = q/p, & \quad  v = r/p.
		\end{align*}
		By assumptions, $0 < u < v < 1$. Further, the conditions (\ref{BeGE2}) and (\ref{BeGE3}) read
		\begin{equation*}
		\sum_{n=1}^\infty n^{u} \lambda_n \left(\frac{1}{\Gamma_n} \sum_{k=1}^n k^{u} \lambda_k \right)^{w/v}  = \sum_{n=1}^\infty (1 + \log n)^{- \frac{r}{r-q}} \frac{1}{n} < \infty
		\end{equation*}
		and
		\begin{equation*}
			\sum_{n=1}^\infty \lambda_n \left(\sum_{k=n}^\infty \lambda_k \right)^{w/v} \max_{k \leq n} \left(\frac{k^{v}}{\Gamma_k} \right)^{w/v} = \sum_{n=1}^\infty \frac{1}{(1 + \log n) n} = \infty.
		\end{equation*}
		Applying the characterization given in Lemma \ref{LemBeGE} we conclude that (\ref{ThmEmbLipschitz4new''*}) fails to be true for all non-negative non-increasing sequences $\{a_n\}_{n \in \N}$. Hence, (\ref{ThmEmbLipschitz4new''}) does not hold true.

	\textsc{Case 4:}
It remains to show (\ref{ThMFrankeLipSharpness4}). Assume that there exists $r < \max\{p,q\}$ such that
			\begin{equation}\label{ThMFrankeLipSharpness4.1}
				 \L^{(\alpha,-b)}_{p,q}(\R^d) \hookrightarrow B^{\alpha + d(1/p_1 -1/p), -b + 1/q}_{p_1,r}(\R^d).
			\end{equation}
					Firstly, let $p \geq q$. Define
			 \begin{equation*}
		 F_0(t) = \left\{\begin{array}{lcl}
                            1 & ,  & 0 < t  < 1, \\
                            & & \\
                            t^{-\alpha + d/p -d} (1 + \log t)^{-\beta} & , & t \geq 1,
            \end{array}
            \right.
            \end{equation*}
            where $-b + 1/q + 1/p< \beta < \min\{1/p, -b + 1/q + 1/r\}$. In light of (\ref{BesovGM}) and (\ref{LipGM}), we have
            \begin{equation*}
            	\|f\|_{\L^{(\alpha,-b)}_{p,q}(\R^d)} \asymp \left(\int_0^1 t^{d p - d -1} \dint t \right)^{1/p} + \left(\int_1^\infty  (1 + \log t)^{- b q - \beta q + q/p} \frac{\dint t}{t}\right)^{1/q} < \infty
            \end{equation*}
            and
            \begin{equation*}
            	\|f\|_{B^{\alpha + d(1/p_1 - 1/p), -b + 1/q}_{p_1,r}(\R^d)} \gtrsim \left(\int_1^\infty (1 + \log t)^{(-b + 1/q -\beta) r} \frac{\dint t}{t} \right)^{1/r} = \infty,
            \end{equation*}
          which contradicts (\ref{ThMFrankeLipSharpness4.1}).

          Suppose now that $p < q$. Note that it will be enough to show that (\ref{ThMFrankeLipSharpness4.1}) fails to be true with $p < r < q$. Let $g$ be a measurable function on $\R^d$ and define $F_0(t) = (g^\ast(t))^{1/p}$ and $f(x)= f_0(|x|)$ where $f_0$ is given by (\ref{3.4new+}). Then, $f \in \widehat{GM}^{d}$. According to (\ref{ThMFrankeLipSharpness4.1}), Theorems \ref{TheoremGMBesov} and \ref{TheoremGMLip}, we have
          \begin{align*}
          	\left(\int_1^\infty t^{\alpha r  +  dr -d r/p  - 1} (1 + \log t)^{(- b + 1/q) r} (g^\ast(t))^{r/p} \dint t\right)^{p/r} & \lesssim  \int_0^1 t^{d p-d - 1} g^\ast(t) \dint t \\
		&\hspace{-7cm}+ \left(\int_1^\infty  (1 + \log t)^{- b q} \left(\int_1^t s^{\alpha p  + d p - d} g^\ast(s) \frac{\dint s}{s} \right)^{q/p} \frac{\dint t}{t}\right)^{p/q}.
          \end{align*}
          Therefore, the inequality
          \begin{equation}\label{HardyTypeConverse}
          	\left(\int_0^\infty w(t) (g^\ast(t))^{r/p} \dint t \right)^{p/r} \lesssim  \left(\int_0^\infty v(t) \left(\frac{1}{U(t)}\int_0^t u(s) g^\ast(s) \dint s \right)^{q/p} \dint t \right)^{p/q}
          \end{equation}
          holds for every measurable function $g$ with
          \begin{equation*}
          	w(t) = \left\{\begin{array}{lcl}
                        0, & & 0 < t < 1,\\
                            & & \\
                         t^{\alpha r + d r -d r/p-1} (1 + \log t)^{(-b + 1/q) r},&   & t \geq 1,\\
            \end{array}
            \right.
          \end{equation*}
           \begin{equation*}
		v(t) = \left\{\begin{array}{lcl}
                         t^{d q - d q/p  -1}, & & 0 < t < 1,\\
                            & & \\
                         t^{\alpha q + d q -dq/p-1} (1 + \log t)^{-b q}, &   & t \geq 1,\\
            \end{array}
            \right.
            \end{equation*}
                         \begin{equation*}
		u(t) = \left\{\begin{array}{lcl}
                        t^{d p - d -1} , & & 0 < t < 1,\\
                            & & \\
                         t^{\alpha p + d p -d -1}, &   & t \geq 1,\\
            \end{array}
            \right.
            \end{equation*}
            and $U(t) = \int_0^t u(s) \dint s$.

            Next we will show that (\ref{HardyTypeConverse}) cannot be true, which yields a contradiction. Indeed, we apply the following characterization of the converse Hardy inequality.

            \begin{lem}[{\cite[Theorem 4.2]{GogatishviliPick}}]\label{LemmaGP}
            	Let $1 \leq R < Q < \infty$ and $P = Q R/(Q-R)$. Let $u, v, w$ be non-negative measurable functions on $[0, \infty)$ and let $U(t) = \int_0^t u(s) \dint s, V(t) = \int_0^t v(s) \dint s, W(t) = \int_0^t w(s) \dint s$.
	 Assume that
	\begin{equation}\label{LemmaGP1}
		\int_0^\infty u(t) \dint t = \infty,\quad \int_0^1 \frac{v(t)}{U^Q(t)} \dint t = \int_1^\infty v(t) \dint t = \infty,
	\end{equation}
	and
	\begin{equation}\label{LemmaGP2}
		\int_0^\infty \frac{v(s)}{U^Q(s) + U^Q(t)} \dint s < \infty, \quad t > 0.
	\end{equation}
	Then, the inequality
	    \begin{equation*}
          	\left(\int_0^\infty w(t) (g^\ast(t))^R \dint t \right)^{1/R} \lesssim  \left(\int_0^\infty v(t) \left(\frac{1}{U(t)}\int_0^t u(s) g^\ast(s) \dint s \right)^{Q} \dint t \right)^{1/Q}
          \end{equation*}
            holds for all measurable functions $g$ if and only if
            \begin{equation}\label{LemmaGP3}
            	\int_0^\infty \frac{U(t)^P \sup_{y \geq t} U(y)^{-P} W(y)^{P/R}}{\Big(V(t) + U(t)^Q \int_t^\infty U(s)^{-Q} v(s) \dint s\Big)^{\frac{P}{Q} + 2}} V(t) \int_t^\infty U(s)^{-Q} v(s) \, \dint s  \, \dint (U^Q(t)) <\infty.
            \end{equation}

            \end{lem}

          Set $R = r/p$ and $Q =q/p$. By assumption, $1 < R < Q$. Further, direct calculations show that (\ref{LemmaGP1}) and (\ref{LemmaGP2}) are fulfilled. Therefore, in virtue of Lemma \ref{LemmaGP},  inequality (\ref{HardyTypeConverse}) does not hold because
          \begin{align*}
          	\int_1^\infty \frac{U(t)^P \sup_{y \geq t} U(y)^{-P} W(y)^{P/R}}{\Big(V(t) + U(t)^Q \int_t^\infty U(s)^{-Q} v(s) \dint s\Big)^{\frac{P}{Q} + 2}} V(t) \int_t^\infty U(s)^{-Q} v(s) \, \dint s  \, \dint (U^Q(t)) \\
	& \hspace{-10cm} \asymp \int_1^\infty \frac{\dint t}{(1 + \log t) t} = \infty.
          \end{align*}
           Consequently, $r \geq q$.
	\end{proof}

%
%

Our next aim is to show the optimality of Theorem \ref{ThmBL}. Namely, we obtain the following result.

\begin{thm}\label{SharpThmBl}
			Let $\alpha > 0, \frac{2 d}{d + 1} < p < \infty, 0 < q, r \leq \infty, b > 1/q$ and $-\infty < \xi < \infty$. Then,
			\begin{equation}\label{SharpThmBL1}
				B^{\alpha, -b + \xi}_{p,q}(\R^d) \hookrightarrow \emph{Lip}^{(\alpha,-b)}_{p,q}(\R^d) \iff \xi \geq 1/\min\{2, p,q\},
			\end{equation}
			\begin{equation}\label{SharpThmBL2}
			 \emph{Lip}^{(\alpha,-b)}_{p,q}(\R^d) \hookrightarrow B^{\alpha, -b +\xi}_{p,q}(\R^d) \iff \xi \leq 1/\max\{2, p,q\},
			\end{equation}
			\begin{equation}\label{SharpThmBL3}
				B^{\alpha, -b + 1/q}_{p,r}(\R^d) \hookrightarrow \emph{Lip}^{(\alpha,-b)}_{p,q}(\R^d) \iff r \leq \min\{2, p,q\},
			\end{equation}
			\begin{equation}\label{SharpThmBL4}
				 \emph{Lip}^{(\alpha,-b)}_{p,q}(\R^d) \hookrightarrow B^{\alpha, -b + 1/q}_{p,r}(\R^d) \iff r \geq \max\{2, p,q\}.
			\end{equation}
		\end{thm}
		
\begin{proof}
	We start by showing (\ref{SharpThmBL1}). Assume first that
	\begin{equation*}
	B^{\alpha, -b + \xi}_{p,q}(\R^d) \hookrightarrow \L^{(\alpha,-b)}_{p,q}(\R^d)
	\end{equation*}
	holds with either $\min\{2,p,q\}=p$ or $\min\{2,p,q\}=q$. Let $\frac{2 d}{d + 1} < p_0 < p$. In light of the embeddings
	\begin{equation}\label{SobClas}
	B^{\alpha + d(1/p_0 -1/p), -b + \xi}_{p_0,q}(\R^d)	\hookrightarrow B^{\alpha, -b + \xi}_{p,q}(\R^d)
	\end{equation}
	(see \eqref{ClasSobEmbBesov}), we derive
	\begin{equation*}
		B^{\alpha + d(1/p_0 -1/p), -b + \xi}_{p_0,q}(\R^d)	\hookrightarrow \L^{(\alpha,-b)}_{p,q}(\R^d)
	\end{equation*}
	which, by (\ref{ThMFrankeLipSharpness1}),  implies that $\xi \geq \frac{1}{\min\{p,q\}} = \frac{1}{\min\{2,p,q\}}$.

	Suppose now that $\min\{2,p,q\}=2$. We will show that given any $\varepsilon > 0$, there exists $f \in B^{\alpha, -b + 1/2 - \varepsilon}_{p,q}(\R^d)$ such that $f \not \in \L^{(\alpha,-b)}_{p,q}(\R^d)$. Indeed, we define
	\begin{equation*}
		f(x) \sim \sum_{j=3}^\infty 2^{-j \alpha} (1 + j)^{-\beta} e^{i(2^j -2) x_1} \psi(x), \quad x \in \R^d,
	\end{equation*}
	where $-b + 1/q + 1/2 - \varepsilon < \beta < -b + 1/q + 1/2$ and $\psi \in \mathcal{S}(\R^d) \backslash \{0\}$ with (\ref{LacCondition}). According to Theorems \ref{ThmBesovLac} and \ref{ThmLipLac}, we have
	\begin{equation*}
		\|f\|_{B^{\alpha, -b + 1/2 - \varepsilon}_{p,q}(\R^d)}^q \asymp \sum_{j=3}^\infty (1 + j)^{(-b + 1/2 - \varepsilon -\beta) q} < \infty
	\end{equation*}
	and
	\begin{align*}
		\|f\|_{\L^{(\alpha,-b)}_{p,q}(\R^d)}^q & \asymp \sum_{k=3}^\infty (1 + k)^{-b q} \left(\sum_{j=3}^k (1 + j)^{-2 \beta} \right)^{q/2} \\
		& \asymp \sum_{k=3}^\infty (1 + k)^{- b q - \beta q + q/2} = \infty.
	\end{align*}
	
        We deal with \eqref{SharpThmBL2}.
	Suppose that either $\max\{2,p,q\}=p$ or $\max\{2,p,q\}=q$ and the following embedding
	\begin{equation*}
		 \L^{(\alpha,-b)}_{p,q}(\R^d) \hookrightarrow B^{\alpha, -b +\xi}_{p,q}(\R^d)
	\end{equation*}
	holds. Let $p < p_1 < \infty$. Then, by (\ref{SobClas}),
	\begin{equation*}
	 \L^{(\alpha,-b)}_{p,q}(\R^d) \hookrightarrow B^{\alpha + d(1/p_1-1/p), -b + \xi}_{p_1,q}(\R^d),
	\end{equation*}
	which implies $\xi \leq 1/\max\{p,q\} = 1/\max\{2,p,q\}$ (see (\ref{ThMFrankeLipSharpness2})).
	
	Let $\max\{2,p,q\}=2$. Assume that there exists $\varepsilon > 0$ for which
	\begin{equation*}
		\L^{(\alpha,-b)}_{p,q}(\R^d) \hookrightarrow B^{\alpha, -b +1/2 + \varepsilon}_{p,q}(\R^d).
	\end{equation*}
	This yields a contradiction. Indeed, set
	\begin{equation*}
		f(x) \sim \sum_{j=3}^\infty 2^{-j \alpha} (1 + j)^{-\beta} e^{i(2^j -2) x_1} \psi(x), \quad x \in \R^d,
	\end{equation*}
	where $-b+1/q + 1/2 < \beta < \min\{1/2, -b + 1/q + 1/2 + \varepsilon\}$. Applying Theorems \ref{ThmBesovLac} and \ref{ThmLipLac},
	\begin{equation*}
		\|f\|_{\L^{(\alpha,-b)}_{p,q}(\R^d)}^q \asymp \sum_{k=3}^\infty (1 + k)^{-b q-\beta q + q/2} < \infty,
	\end{equation*}
	but
	\begin{equation*}
		\|f\|_{B^{\alpha, -b +1/2 + \varepsilon}_{p,q}(\R^d)}^q \asymp \sum_{j=3}^\infty (1 + j)^{-b q + q/2 + \varepsilon q -\beta q} = \infty.
	\end{equation*}
	The proof of (\ref{SharpThmBL2}) is complete.
	
Next we treat (\ref{SharpThmBL3}). Firstly, assume that either $\min\{2,p,q\} = p$ or $\min\{2,p,q\}=q$. Let $\frac{2 d}{d + 1} < p_0 < p$. If
	\begin{equation*}
		B^{\alpha, -b + 1/q}_{p,r}(\R^d) \hookrightarrow \L^{(\alpha,-b)}_{p,q}(\R^d)
	\end{equation*}
	then (see (\ref{SobClas}))
	\begin{equation*}
		B^{\alpha + d(1/p_0-1/p), -b + 1/q}_{p_0,r}(\R^d) \hookrightarrow \L^{(\alpha,-b)}_{p,q}(\R^d).
	\end{equation*}
	According to (\ref{ThMFrankeLipSharpness3}), $r \leq \min\{p,q\} = \min\{2,p,q\}$.
	
	Secondly, suppose that $\min\{2,p,q\}=2$ and there exists $r > 2$ such that
	\begin{equation*}
		B^{\alpha, -b + 1/q}_{p,r}(\R^d) \hookrightarrow \L^{(\alpha,-b)}_{p,q}(\R^d).
	\end{equation*}
	Let
	\begin{equation*}
		f(x) \sim \sum_{j=3}^\infty 2^{-j \alpha} (1 + j)^{-\beta} e^{i(2^j -2) x_1} \psi(x), \quad x \in \R^d,
	\end{equation*}
	where $-b + 1/q + 1/r < \beta < -b + 1/q + 1/2$. Then, in virtue of Theorems \ref{ThmBesovLac} and \ref{ThmLipLac},
		\begin{equation*}
		\|f\|_{B^{\alpha, -b +1/q}_{p,r}(\R^d)}^r \asymp \sum_{j=3}^\infty (1 + j)^{-b r +r/q -\beta r} < \infty
	\end{equation*}
	and
		\begin{equation*}
		\|f\|_{\L^{(\alpha,-b)}_{p,q}(\R^d)}^q \asymp \sum_{k=3}^\infty (1 + k)^{-b q-\beta q + q/2} = \infty,
	\end{equation*}
	which is not possible.
	
	It remains to show (\ref{SharpThmBL4}). Assume that
	\begin{equation*}
		 \L^{(\alpha,-b)}_{p,q}(\R^d) \hookrightarrow B^{\alpha, -b + 1/q}_{p,r}(\R^d)
	\end{equation*}
	where either $\max\{2,p,q\}=p$ or $\max\{2,p,q\}=q$. In particular, by (\ref{SobClas}),
		\begin{equation*}
		 \L^{(\alpha,-b)}_{p,q}(\R^d) \hookrightarrow B^{\alpha + d(1/p_1-1/p), -b + 1/q}_{p_1,r}(\R^d), \quad  p_1 > p.
	\end{equation*}
	Then, (\ref{ThMFrankeLipSharpness4}) implies $r \geq \max\{p,q\} = \max\{2,p,q\}$.
	
	Let $\max\{2,p,q\}=2$. Suppose now that
	\begin{equation*}
		 \L^{(\alpha,-b)}_{p,q}(\R^d) \hookrightarrow B^{\alpha, -b + 1/q}_{p,r}(\R^d)
	\end{equation*}
	for some $r < 2$. However this is not possible, as the following counterexample shows. Let
	\begin{equation*}
		f(x) \sim \sum_{j=3}^\infty 2^{-j \alpha} (1 + j)^{-\beta} e^{i(2^j -2) x_1} \psi(x), \quad x \in \R^d,
	\end{equation*}
	where $-b + 1/q + 1/2 < \beta < -b + 1/q + 1/r$. Therefore, Theorems \ref{ThmBesovLac} and \ref{ThmLipLac} yield that
		\begin{equation*}
		\|f\|_{\L^{(\alpha,-b)}_{p,q}(\R^d)}^q \asymp \sum_{k=3}^\infty (1 + k)^{-b q-\beta q + q/2} < \infty
	\end{equation*}
	and
		\begin{equation*}
		\|f\|_{B^{\alpha, -b +1/q}_{p,r}(\R^d)}^r \asymp \sum_{j=3}^\infty (1 + j)^{-b r +r/q -\beta r} = \infty.
	\end{equation*}
\end{proof}

It was shown in Theorem \ref{ThmBL} that
\begin{equation*}
	\L^{(\alpha,-b)}_{2,2}(\R^d)  = B^{\alpha, -b + 1/2}_{2,2}(\R^d), \quad b > 1/2.
\end{equation*}
In fact, our next result asserts that this is the only possible case where Lipschitz spaces and Besov spaces coincide. More precisely, we have the following

\begin{thm}
	Let $\frac{2 d}{d+1} < p < \infty, 0 < q \leq \infty, \alpha > 0, b > 1/q$, and $-\infty < \xi < \infty$. Then,
	\begin{equation*}
		\emph{Lip}^{(\alpha,-b)}_{p,q}(\R^d)  = B^{\alpha, \xi}_{p,q}(\R^d) \iff p=q=2 \quad \text{and} \quad \xi = -b + 1/2.
	\end{equation*}
\end{thm}

\begin{proof}
	Assume that
	\begin{equation*}
		\L^{(\alpha,-b)}_{p,q}(\R^d)  = B^{\alpha, \xi}_{p,q}(\R^d).
	\end{equation*}
	Thus,
	\begin{equation}\label{ThmLipBesSharpHil}
		(L_p(\R^d), \L^{(\alpha,-b)}_{p,q}(\R^d))_{\theta,q} = (L_p(\R^d), B^{\alpha, \xi}_{p,q}(\R^d))_{\theta,q}, \quad 0 < \theta < 1.
	\end{equation}
	It is an immediate consequence of
	\begin{equation*}
	B^0_{p, \min\{2,p\}}(\R^d) \hookrightarrow L_p(\R^d) \hookrightarrow B^0_{p, \max\{2,p\}}(\R^d)
\end{equation*}
(see (\ref{ClassBH})) and (\ref{InterBes}) that
\begin{equation}\label{ThmLipBesSharpHil2}
 (L_p(\R^d), B^{\alpha, \xi}_{p,q}(\R^d))_{\theta,q} = B^{\theta \alpha, \theta \xi}_{p,q}(\R^d).
\end{equation}
On the other hand, it follows from (\ref{LipLimInter}), (\ref{LemmaReiteration2}) and (\ref{InterBesSobLp2}) that
\begin{align}
	(L_p(\R^d), \L^{(\alpha,-b)}_{p,q}(\R^d))_{\theta,q}  &= (L_p(\R^d), (L_p(\R^d), H^\alpha_p(\R^d))_{(1,-b),q} )_{\theta,q} \nonumber \\
	& \hspace{-3cm}= (L_p(\R^d), H^\alpha_p(\R^d))_{\theta,q; \theta (-b+1/q)} = B^{\theta \alpha, \theta(-b+1/q)}_{p,q}(\R^d). \label{ThmLipBesSharpHil3}
\end{align}
Combining (\ref{ThmLipBesSharpHil}), (\ref{ThmLipBesSharpHil2}) and (\ref{ThmLipBesSharpHil3}), we arrive at
\begin{equation*}
	 B^{\theta \alpha, \theta(-b+1/q)}_{p,q}(\R^d) =  B^{\theta \alpha, \theta \xi}_{p,q}(\R^d).
\end{equation*}
This implies $\xi = -b + 1/q$ and so
\begin{equation}\label{ThmLipBesSharpHil4}
		\L^{(\alpha,-b)}_{p,q}(\R^d)  = B^{\alpha, -b + 1/	q}_{p,q}(\R^d).
	\end{equation}
	Next we show by contradiction that $p=q$. Specifically, we will prove that if $p \neq q$ then (\ref{ThmLipBesSharpHil4}) fails to be true. Let
	\begin{equation*}
		 F_0(t) = \left\{\begin{array}{lcl}
                             1& ,  & 0 < t  < 1, \\
                            & & \\
                            t^{-\alpha + d/p -d} (1 + \log t)^{-\beta} & , & t \geq 1,
            \end{array}
            \right.
	\end{equation*}
	and $\widehat{f}(\xi) = F_0(|\xi|)$. Here, $\beta$ is a real number, which depends on $b, p$ and $q$, to be chosen. Assume for example that $q < p$. Then, we let $-b + 1/p + 1/q < \beta < \min\{1/p, -b + 2/q\}$. Invoking Theorems \ref{TheoremGMLip} and \ref{TheoremGMBesov}, one obtains
	\begin{equation*}
		\|f\|_{\L^{(\alpha,-b)}_{p,q}(\R^d)}  \asymp \left(\int_0^1 t^{d p - d -1} \dint t \right)^{1/p} + \left(\int_1^\infty (1 + \log t)^{-b q - \beta q + q/p} \frac{\dint t}{t}  \right)^{1/q} < \infty,
	\end{equation*}
	but
	\begin{equation*}
		\|f\|_{B^{\alpha, -b + 1/q}_{p,q}(\R^d)} \gtrsim \left(\int_1^\infty (1 + \log t)^{-b q + 1 -\beta q} \frac{\dint t}{t} \right)^{1/q} = \infty.
	\end{equation*}
	The case $q > p$ can be treated analogously by taking $-b + 2/q < \beta < -b + 1/p + 1/q$. Consequently, by (\ref{ThmLipBesSharpHil4}), we have
	\begin{equation}\label{ThmLipBesSharpHil5}
		\L^{(\alpha,-b)}_{q,q}(\R^d)  = B^{\alpha, -b + 1/	q}_{q,q}(\R^d).
	\end{equation}
	Finally, we will show that $q=2$. We proceed by contradiction. Assume first that $q < 2$. Let
	\begin{equation*}
		f(x) \sim \sum_{j=3}^\infty 2^{-j \alpha} (1 + j)^{-\beta} e^{i(2^j -2) x_1} \psi(x), \quad x \in \R^d,
	\end{equation*}
	where $-b + 1/q + 1/2 < \beta < \min\{-b + 2/q, 1/2\}$ and $\psi \in \mathcal{S}(\R^d) \backslash \{0\}$ with (\ref{LacCondition}). In virtue of Theorems \ref{ThmBesovLac} and \ref{ThmLipLac}, we have
	\begin{equation*}
		\|f\|_{\L^{(\alpha,-b)}_{q,q}(\R^d)} \asymp \left(\sum_{k=3}^\infty (1 + k)^{-b q -\beta q + q/2} \right)^{1/q} < \infty
	\end{equation*}
	and
	\begin{equation*}
		\|f\|_{B^{\alpha, -b + 1/	q}_{q,q}(\R^d)} \asymp \left( \sum_{k=3}^\infty (1 + k)^{-b q + 1 -\beta q}\right) = \infty.
	\end{equation*}
The case $q > 2$ can be done in a similar way. Hence, $q=2$ and the proof is complete.
	
\end{proof}

The following result establishes the optimality of Theorem \ref{ThmEmbLipIntegrability}.

\begin{thm}\label{ThmEmbLipIntegrabilitySharp}
	Let $1 < p < \infty, \alpha_i > 0, 0 < q_i \leq \infty$, and $b_i > 1/q_i, i = 0,1$. Then,
	\begin{equation*}
		\emph{Lip}^{(\alpha_0,-b_0)}_{p,q_0}(\R^d) \hookrightarrow \emph{Lip}^{(\alpha_1,-b_1)}_{p,q_1}(\R^d) \iff \left\{\begin{array}{lcl}
                            \alpha_0 > \alpha_1, &   & \\
                            & & \\
                            \alpha_0 = \alpha_1, &  b_1 -\frac{1}{q_1} > b_0 -\frac{1}{q_0}, & \\
                            & & \\
                             \alpha_0 = \alpha_1, &  b_1 -\frac{1}{q_1} = b_0 -\frac{1}{q_0}, & q_0 \leq q_1.
            \end{array}
            \right.
	\end{equation*}
	The corresponding result for periodic spaces also holds true.
\end{thm}

\begin{proof}
	Assume that
	\begin{equation}\label{ProofThmEmbLipIntegrabilitySharp1}
		\L^{(\alpha_0,-b_0)}_{p,q_0}(\R^d) \hookrightarrow \L^{(\alpha_1,-b_1)}_{p,q_1}(\R^d).
	\end{equation}
	Consider the Fourier series
	\begin{equation*}
		f (x) =  \sum_{j=3}^\infty a_j e^{i (2^j - 2) x_1}, \quad x \in \R^d,
	\end{equation*}
	where $\psi \in \mathcal{S}(\R^d) \backslash \{0\}$ with \eqref{LacCondition} and for some $(a_j)$ to be chosen.
	
	For $N \geq 3$, we let
	\begin{equation*}
		a_j =  \left\{\begin{array}{lc}
                           1, & 1 \leq j \leq N,  \\
                            &  \\
                          0, & j > N. \\
            \end{array}
            \right.
	\end{equation*}
	Elementary computations yield
	\begin{equation*}
		 \left(\sum_{k=3}^\infty (1 + k)^{- b_i q_i} \Big(\sum_{j=3}^k 2^{j \alpha_i 2} |a_j|^2 \Big)^{q_i/2} \right)^{1/q_i} \asymp 2^{N \alpha_i} (1 + N)^{-b_i+1/q_i}, \quad i=0, 1.
	\end{equation*}
	Therefore, by Theorem \ref{ThmLipLac} and \eqref{ProofThmEmbLipIntegrabilitySharp1}, we have
	\begin{align*}
		 2^{N \alpha_1} (1 + N)^{-b_1+1/q_1} \asymp \|f\|_{\L^{(\alpha_1,-b_1)}_{p,q_1}(\R^d)} \lesssim \|f\|_{\L^{(\alpha_0,-b_0)}_{p,q_0}(\R^d)} \asymp  2^{N \alpha_0} (1 + N)^{-b_0+1/q_0},
	\end{align*}
	which implies that either $\alpha_0 > \alpha_1$ or $\alpha_0 = \alpha_1, b_0 -1/q_0 \leq b_1 -1/q_1$. Next we show that if $\alpha = \alpha_0 = \alpha_1$ and $b = b_0 -1/q_0 = b_1 -1/q_1$ then $q_0 \leq q_1$. We shall proceed by contradiction. Suppose that $q _0 > q_1$ and let
	\begin{equation*}
		a_j = 2^{-j \alpha} (1 + j)^{b-1/2} (1 + \log (1 + j))^{-\varepsilon}, \quad 1/q_0 < \varepsilon < 1/q_1.
	\end{equation*}
	Applying again Theorem \ref{ThmLipLac}, we have
	\begin{align*}
		\|f\|_{\L^{(\alpha, -b_1)}_{p,q_1}(\R^d)}^{q_1} & \asymp \sum_{k=3}^\infty (1 + k)^{- b_1 q_1} \Big(\sum_{j=3}^k (1 + j)^{2 b - 1} (1 + \log (1 + j))^{-2 \varepsilon} \Big)^{q_1/2}  \\
		&\asymp \sum_{k=3}^\infty (1 + \log (1 + k))^{-\varepsilon q_1} \frac{1}{k+1} = \infty
	\end{align*}
	and
	\begin{equation*}
		\|f\|_{\L^{(\alpha, -b_0)}_{p,q_0}(\R^d)}^{q_0} \asymp \sum_{k=3}^\infty (1 + \log (1 + k))^{-\varepsilon q_0} \frac{1}{k+1} < \infty,
	\end{equation*}
which contradicts \eqref{ProofThmEmbLipIntegrabilitySharp1}.
	\end{proof}

The sharpness of Theorem \ref{ThmEmbLipschitz} is given below.

\begin{thm}\label{ThmEmbLipschitzSharp}
	Let $1 < p < \infty, 0 < \alpha_ 1 < \alpha_0 < \infty, \alpha_0 - 1 = \alpha_1-1/p$ and $-\infty < \xi < \infty$. Then, we have
	\begin{equation*}
		\emph{Lip}^{(\alpha_0,0)}_{1,\infty}(\T) \hookrightarrow \emph{Lip}^{(\alpha_1, -\xi)}_{p, \infty}(\T) \iff \xi \geq 1/p.
	\end{equation*}
	In particular,
	\begin{equation*}
		\emph{\text{BV}}(\T) \hookrightarrow \emph{Lip}^{(1/p, -\xi)}_{p, \infty}(\T) \iff \xi \geq 1/p.
	\end{equation*}
\end{thm}

\begin{proof}
	Define $f$ such that
$$
f^{(\alpha_0)}(x)=\sum_{n=1}^\infty a_n \, {\cos \,n x},
$$
where $\{a_n\}_{n \in \N}$ is a convex null-sequence of positive numbers to be chosen. Here, $f^{(\alpha_0)}$ denotes the fractional derivative in the sense of Weyl of order $\alpha_0 > 0$ of the integrable function $f$ (see \cite[XII, Section 8, page 134]{Zygmund}).
Hence, by \cite[V, (1.5), page 183]{Zygmund}, one gets $f^{(\alpha_0)}, f \in L_1(\T)$. Therefore we can apply \cite[Lemma 6(v), page 788]{ButzerDyckhoffGorlichStens}:
 $$
\omega_{\alpha_0}(f,t)_1\,\lesssim\,t^{\alpha_0} \big\|f^{(\alpha_0)}\big\|_{L_1(\T)},
 $$
which yields that $f \in \textnormal{Lip}_{1,\infty}^{(\alpha_0, 0)}(\T)$ and, thus,
$$
\omega_{\alpha_1}(f, t)_p \lesssim
\,t^{\alpha_1}
\Big( 1-\log t \Big)^\xi.
$$


Let $S_n(f)$ stand for the $n$-th partial sum of the Fourier series of $f$. Using the realization of the $K$-functional \cite{SimonovTikhonov},
\begin{equation}\label{realization}
	\omega_{\alpha_1}(f,t)_p \asymp \|f - S_{[1/t]} (f)\|_{L_p(\T)} + t^{\alpha_1} \Big\|\big(S_{[1/t]} (f)\big)^{(\alpha_1)} \Big\|_{L_p(\T)}, \quad 0 < t < 1,
\end{equation}
 we have
$$
\omega_{\alpha_1}(f,t)_p \gtrsim \,t^{\alpha_1}
\Big\|\big(
S_{[1/t]}(f)\big)^{(\alpha_1)}
\Big\|_{L_p(\T)}.
$$
Therefore, by Hardy-Littlewood's theorem \cite[XII, Section 6, (6.6), page 129]{Zygmund} (see also \eqref{HLGMPer}),
$$
\omega_{\alpha_1}(f,t)_p \gtrsim t^{\alpha_1}
\Big\|
\sum_{n=1}^{[1/t]} n^{\alpha_1} \frac{a_n}{n^{\alpha_0}}  \cos \,n x
\Big\|_{L_p(\T)}
\gtrsim
t^{\alpha_1}
\left(
\sum\limits_{n=1}^{[1/t]} a_n^p \frac{1}{n} \right)^{1/p}.
$$
By the monotonicity of $\{a_n\}_{n \in \N}$,
$$
\omega_{\alpha_1}(f,t)_p \gtrsim
\,t^{\alpha_1}
a_{[1/t]}
\Big( 1-\log t \Big)^{1/p}.
$$
Hence, we have arrived at the estimate
$$
\Big( 1-\log t \Big)^{\xi}
\,\,\ge\,\,
C\,\,a_{[1/t]}
\Big( 1-\log t \Big)^{1/p}.
$$

Let us show that if $\xi<1/p$, then
   there exists a convex null-sequence $\{a_n\}_{n \in \N}$ such that
\begin{equation}\label{inf}
a_{[1/t]}\Big( 1-\log t \Big)^{1/p-\xi}
 \to \infty\quad\mbox{as}\quad t\to 0,
\end{equation}
which contradicts the previous estimate.
 Indeed, we define
$$
c_n=\sqrt{\frac{1}
{
 \big(1 + \log n\big)^{1/p-\xi} }}\to 0 \quad\mbox{as}\quad n\to\infty.
$$
Then it is enough to take  a positive sequence $a_n$ (see, e.g., \cite[Chapter V, Section 1, page 184]{Zygmund}) such that
$$
c_n\le a_n\qquad\mbox{and}\qquad
a_n
\quad\mbox{is a  convex null-sequence}.$$
Thus, (\ref{inf}) holds and therefore $\xi \ge1/p$.
\end{proof}

The optimality of Theorem \ref{ThmEmbLipschitzinfty} is given in the next result.

\begin{thm}\label{ThmEmbLipschitzinftySharp}
		Let $1 < p < \infty, 0 < \alpha  < \infty, 0 < q \leq \infty, b  > 1/q \, (b \geq 0 \, \text{ if } \, q=\infty)$ and $-\infty < \xi < \infty$. Then, we have
	\begin{equation*}
		\emph{Lip}^{(\alpha + 1/p,-b)}_{p,q}(\T) \hookrightarrow \emph{Lip}^{(\alpha, -b-\xi)}_{\infty, q}(\T) \iff \xi \geq 1 - 1/p.
	\end{equation*}
\end{thm}

\begin{proof}
	Let $0 < q \leq \infty$ and $b > 1/q$. We shall proceed by contradiction, that is, suppose that the following holds
	\begin{equation}\label{ThmEmbLipschitzinftySharp1}
		\L^{(\alpha + 1/p,-b)}_{p,q}(\T) \hookrightarrow \L^{(\alpha, -b-\xi)}_{\infty, q}
	\end{equation}
	for some $\xi < 1 - 1/p$. Assume first that
$\alpha\neq 2l-1,$ $l\in \N$. Define then
$$
f(x) =\sum\limits_{n=1}^\infty
\frac{\cos nx}{n^{1+\alpha}\big(1 + \log n\big)^{-b + \varepsilon}},\qquad \frac{1}{p} + \frac{1}{q}<\varepsilon< \min \Big\{\frac{1}{p} + b, -\xi + 1 + \frac{1}{q}\Big\}.
$$
Clearly, $f \in C(\T)$. For $\nu \in \N$, applying the realization result given in \eqref{realization}, we obtain
\begin{align*}
 \omega_{\alpha + 1/p}(f,1/\nu)_p & \asymp
\nu^{-(\alpha + 1/p)}
\left\|
\sum_{n=1}^\nu
n^{\alpha+1/p} \frac{\cos nx}{n^{1 + \alpha}\big(1 + \log n\big)^{-b + \varepsilon} } \right\|_{L_p(\T)} \\
&
\hspace{1cm}+
\left\|\sum_{n=\nu+1}^\infty \frac{\cos nx}{n^{1 + \alpha}\big(1 + \log n\big)^{-b + \varepsilon}} \right\|_{L_p(\T)}.
\end{align*}
Using Hardy-Littlewood's theorem on monotone Fourier coefficients
 (to estimate the second term, we apply \cite[Theorem 3 (A)]{compt}), we derive
\begin{align*}
\omega_{\alpha + 1/p}(f,1/\nu)_p
 &\lesssim
\nu^{-(\alpha  + 1/p)}
\left(
\sum_{n=1}^\nu
\frac{n^{(1/p-1)p+p-2}}{\big(1 + \log n\big)^{(-b + \varepsilon) p}} \right)^{1/p}
\\& \hspace{1cm}+
\left(
\frac{\nu^{(-\alpha-1)p+p-1}}{\big(1 + \log \nu\big)^{(-b + \varepsilon) p}}
+\sum_{n=\nu+1}^\infty
\frac{n^{(-\alpha-1)p+p-2}}{\big(1 + \log n\big)^{(-b+\varepsilon) p}}  \right)^{1/p}\\
&\lesssim \nu^{-(\alpha  + 1/p)} (1 + \log \nu)^{b-\varepsilon + 1/p}
\end{align*}
because $\varepsilon < b +  1/p$. Thus, 
\begin{equation*}
	\omega_{\alpha + 1/p}(f,t)_p \lesssim t^{\alpha + 1/p} (1 - \log t)^{b-\varepsilon + 1/p}, \quad 0 < t < 1.
 \end{equation*}
This yields that $f \in \L_{p,q}^{(\alpha+1/p, -b)}(\T)$ because
\begin{equation*}
	\int_0^1 (t^{-\alpha - 1/p} (1 - \log t)^{-b} \omega_{\alpha + 1/p}(f,t)_p)^q \frac{\dint t}{t} \lesssim \int_0^1 (1 - \log t)^{(-\varepsilon + 1/p) q} \frac{\dint t}{t} < \infty.
\end{equation*}

Next we estimate $\omega_{\alpha}(f, t)_\infty$ from below. This will be done by using Corollary 4.4.1 from \cite{gm-jat}:
for $f(x)=\sum_n b_n\cos nx$ and $g(x)=\sum_n b_n\sin nx$ with   $\{b_n\}$ being a  general monotone (or just monotone) sequence, we have
\begin{align}\label{mod1}
& \omega_{{\beta}}\Big(f,\frac{1}{n}\Big)_\infty \asymp n^{-{\beta}} \sum\limits_{m=1}^n
m^{{\beta}} b_{m} + \sum\limits_{m=n}^ \infty b_{m} &&\mbox{for}\quad
\beta \ne 2l-1\, (l\in \N),
 \\ \label{mod2}
&\omega_{{\beta}}\Big(g,\frac{1}{n}\Big)_\infty \asymp n^{-{\beta}} \sum\limits_{m=1}^n
m^{{\beta}} b_{m} + \max\limits_{m\ge n}
(m b_{m})
&&\mbox{for}\quad \beta \ne 2l\, (l\in \N).
\end{align}
Taking into account (\ref{mod1}), we have
$$
\omega_{\alpha}(f, t)_\infty \gtrsim
t^{\alpha}
\sum_{k=1}^{[1/t]}
 \frac{k^{\alpha}}{k^{1+\alpha}
{\big(1 + \log k\big)^{-b + \varepsilon }}
}
\gtrsim
t^{\alpha} (1-\log t)^{1 + b-\varepsilon}
$$
because $\varepsilon < b + 1$. Hence
\begin{equation*}
	\int_0^1 (t^{-\alpha} (1-\log t)^{-b-\xi} \omega_\alpha(f,t)_\infty)^q \frac{\dint t}{t} \gtrsim \int_0^1 (1-\log t)^{(-\xi -\varepsilon + 1) q} \frac{\dint t}{t} = \infty.
\end{equation*}
This contradicts \eqref{ThmEmbLipschitzinftySharp1}.

In the case $\alpha= 2l-1,$ $l\in\N$, we proceed similarly and consider
$$
f(x)=\sum\limits_{n=1}^\infty
\frac{\sin nx}{n^{1+\alpha}\big(1 + \log n\big)^{-b + \varepsilon }}, \qquad \frac{1}{p} + \frac{1}{q}<\varepsilon< \min \Big\{\frac{1}{p} + b, -\xi + 1 + \frac{1}{q}\Big\},
$$
taking into account  (\ref{mod2}) in place of (\ref{mod1}).

The proof for $q=\infty$ and $b=0$ follows similar ideas as above but now taking
$$
f(x) =\sum\limits_{n=1}^\infty
\frac{\cos nx}{n^{1+\alpha}\big(1 + \log n\big)^\varepsilon},\qquad \frac{1}{p} <\varepsilon<1.
$$
We  leave further details to the reader.
\end{proof}

\section{Growth properties of Lipschitz functions}

The goal of this section is to study the growth properties of functions in Lipschitz spaces. Specifically,
 in Sections \ref{SectionEmbLInfty} and  \ref{lack}
 we characterize embeddings into the space of bounded (continuous) functions and $L_r(\R^d)$. Sections
 \ref{lz} and \ref{ri} concern  embeddings of Lipschitz spaces into Lorentz--Zygmund and r.i. spaces, respectively.

\subsection{Embeddings into $L_\infty$}\label{SectionEmbLInfty} For the convenience of the reader, we start by recalling the characterizations of the embeddings from Besov and Sobolev spaces into $L_\infty(\R^d)$. For further details, as well as extensions to function spaces of generalized smoothness, we refer the reader to \cite{Kalyabin}, \cite[Theorem 3.3.1]{SickelTriebel}, \cite[Theorem 11.4]{Triebel01} and \cite[Proposition 3.13]{CaetanoMoura04a}.

\begin{prop}
\begin{enumerate}[\upshape(i)]
\item Let $1 \leq p \leq \infty, 0 < q \leq \infty$, and $-\infty < s < \infty$. Then,
\begin{equation}\label{PropEmbSobBesLinfty1}
	B^s_{p,q}(\R^d) \hookrightarrow L_\infty(\R^d) \iff \left\{\begin{array}{lcl}
                            s > \frac{d}{p} &   &  \\
                            &\text{or} & \\
                          s=\frac{d}{p} & \text{and} &  q \leq 1.
            \end{array}
            \right.
\end{equation}
\item Let $1 < p < \infty$ and $\alpha > 0$. Then,
\begin{equation}\label{PropEmbSobBesLinfty2}
	H^\alpha_{p}(\R^d) \hookrightarrow L_\infty(\R^d) \iff \alpha > \frac{d}{p}.
\end{equation}
\end{enumerate}
In \eqref{PropEmbSobBesLinfty1} and \eqref{PropEmbSobBesLinfty2} one can replace the space $L_\infty(\R^d)$ by $C(\R^d)$.
\end{prop}

Accordingly we characterize embeddings into $L_\infty(\R^d)$ for Lipschitz spaces.

\begin{thm}\label{Thm10.2}
	Let $\alpha > 0, 1 < p < \infty, 0 < q \leq \infty$, and $b > 1/q$. Then,
	\begin{equation*}
		\emph{Lip}^{(\alpha,-b)}_{p,q}(\R^d) \hookrightarrow L_\infty(\R^d) \iff \alpha > d/p,
	\end{equation*}
	where $L_\infty(\R^d)$ can be replaced by $C(\R^d)$.
	
	The corresponding result for periodic spaces also holds true.
\end{thm}
\begin{proof}
	Let $\alpha > d/p$. Applying the embedding (see \eqref{ProofThmEmbLipIntegrability1})
	\begin{equation*}
		B^{\alpha, b}_{p,q}(\R^d) \hookrightarrow B^{d/p}_{p,1}(\R^d), \quad  -\infty < b < \infty,
	\end{equation*}
	and \eqref{ThmBL1}, we have
	\begin{equation*}
		\L^{(\alpha,-b)}_{p,q}(\R^d) \hookrightarrow  B^{\alpha, -b+1/\max\{2,p,q\}}_{p,q}(\R^d) \hookrightarrow B^{d/p}_{p,1}(\R^d).
	\end{equation*}
	Now, it follows from \eqref{PropEmbSobBesLinfty1} that $\L^{(\alpha,-b)}_{p,q}(\R^d) \hookrightarrow  L_\infty(\R^d)$.
	
	Next we prove the sharpness assertion. We first remark that
	\begin{equation*}
		H^\alpha_p(\R^d) \hookrightarrow \L^{(\alpha,-b)}_{p,q}(\R^d).
	\end{equation*}
	This is an immediate consequence of \eqref{LipSob}. Hence, if $\L^{(\alpha,-b)}_{p,q}(\R^d) \hookrightarrow L_\infty(\R^d)$ holds then, in particular,
	\begin{equation*}
		H^\alpha_p(\R^d) \hookrightarrow  L_\infty(\R^d),
	\end{equation*}
	which implies $\alpha > d/p$ (see \eqref{PropEmbSobBesLinfty2}).
\end{proof}

\begin{rem}\label{rem-growth-env}
Recall that the introduction and study of the spaces $\Lipx{1}{-b}{p}{q}(\R^d)$ in \cite{Haroske} was closely connected with the theory of growth and continuity envelopes, $\mathfrak{E}_{\mathsf G}(X)$ and $\mathfrak{E}_{\mathsf C}(X)$, respectively, with $X=\Lipx{1}{-b}{p}{q}(\R^d)$. This theory was started in \cite{Triebel01,Haroske} and aims at a finer (and, simultaneously,  easy to access) characterization of unboundedness and smoothness features of function spaces. In particular, the growth envelope function of a function space $X$ is defined by 
  \[
\mathcal{E}_{\mathsf G}^X(t) = \sup_{\|f\|_X\leq 1} f^\ast(t),\quad t>0,
\]
and the continuity envelope function of $X$ by
\[
\mathcal{E}_{\mathsf C}^X(t) = \sup_{\|f\|_X\leq 1} \frac{\omega_1(f,t)_\infty}{t},\quad t \downarrow 0,
\]
for further details and discussion we refer to the above-mentioned literature and the references given there. In view of the known envelope results (cf. \cite{CaetanoFarkas, CaetanoMoura, HaroskeMoura, HaroskeMoura08}) and the sharp embeddings in Theorem~\ref{ThmBL}, in particular \eqref{ThmBL2}, one obtains for $X=\Lipx{\alpha}{-b}{p}{q}(\mathbb{R}^d)$, $1<p<\infty$, $0<q\leq\infty$, $b>\frac1q$,
\[
\mathcal{E}_{\mathsf G}^X(t) \asymp t^{\frac{\alpha}{d}-\frac{1}{p}} \left|\log t\right|^{b-\frac1q},\quad t\downarrow 0,  \qquad 0<\alpha<\frac{d}{p},
\]
and
\[
\mathcal{E}_{\mathsf C}^X(t) \asymp t^{-1+\alpha-\frac{d}{p}} \left|\log t\right|^{b-\frac1q},\quad t\downarrow 0, \qquad \frac{d}{p}<\alpha<\frac{d}{p}+1.
\]
This extends the information about the unboundedness of functions in $\Lipx{\alpha}{-b}{p}{q}(\mathbb{R}^d)$ presented in Theorem~\ref{Thm10.2}, and the corresponding continuity result in Theorem~\ref{ThmEmbLipLipClas}.  It also shows that the interplay of all involved parameters influences both integrability and smoothness properties.
\end{rem}

\subsection{Lack of optimal embeddings into Lebesgue spaces}\label{lack}

The study of integrability properties of functions which satisfy additional smoothness assumptions are of central interest in the theory of function spaces. For instance, it is well known (see \cite{SickelTriebel}) that if $0 < s < d/p$ and $s - d/p = -d/r$ then
\begin{equation*}
	H^s_p(\R^d) \hookrightarrow L_{r_0}(\R^d), \quad p \leq r_0 \leq r,
\end{equation*}
and if, in addition, $0 < q \leq \infty$ then
\begin{equation}\label{LackSobEmbLebesgue}
	B^s_{p,q}(\R^d) \hookrightarrow L_r(\R^d) \iff q \leq r
\end{equation}
and
\begin{equation*}
	B^s_{p,q}(\R^d) \hookrightarrow L_{r_0}(\R^d), \quad p \leq r_0 < r.
\end{equation*}

In this section we focus on integrability properties of Lipschitz spaces. Surprisingly enough, even that  these spaces are very close to Sobolev and Besov spaces (cf. \eqref{LipSob}, Theorems \ref{ThmBL} and \ref{ThMFrankeLip}), our next result shows that they are never embedded into the optimal Lebesgue space. 

\begin{thm}\label{ThmSobEmbLebesgue}
	Let $\frac{2 d}{d+1} < p, r_0 < \infty, 0 < \alpha < d/p, \alpha -d/p = -d/r, 0 < q \leq \infty$ and $b > 1/q$. Then,
	\begin{equation*}
		\emph{Lip}^{(\alpha,-b)}_{p,q}(\R^d) \hookrightarrow L_{r_0}(\R^d) \iff p \leq r_0 < r.
	\end{equation*}
	The corresponding result for periodic spaces also holds true.
\end{thm}

\begin{proof}
	Obviously, we have $\L^{(\alpha,-b)}_{p,q}(\R^d) \hookrightarrow L_p(\R^d)$. Let $r_0 \in (p,r)$ and set $\alpha_0 = d/p-d/r_0$. According to \eqref{LackSobEmbLebesgue},
	\begin{equation*}
	B^{\alpha_0}_{p,r_0}(\R^d) \hookrightarrow L_{r_0}(\R^d).
	\end{equation*}
	Further, it follows from \eqref{ThmBL1} and \eqref{ProofThmEmbLipIntegrability1} that
	\begin{equation*}
		\L^{(\alpha,-b)}_{p,q}(\R^d) \hookrightarrow B^{\alpha, -b+1/q}_{p,\max\{2,p,q\}}(\R^d) \hookrightarrow B^{\alpha_0}_{p,r_0}(\R^d).
	\end{equation*}
	Therefore, $\L^{(\alpha,-b)}_{p,q}(\R^d) \hookrightarrow  L_{r_0}(\R^d)$.
	
	Next we prove the sharpness assertion $r_0 < r$. Arguing by contradiction, suppose that there exists $r_0 \geq r$ such that
	\begin{equation}\label{ThmSobEmbLebesgueProof1}
	\L^{(\alpha,-b)}_{p,q}(\R^d) \hookrightarrow  L_{r_0}(\R^d).
	\end{equation}
	For any measurable function $g$ on $\R^d$, we let $F_0(t) = (g^\ast(t))^{1/p}$ and $f(x)= f_0(|x|)$ where $f_0$ is given by (\ref{3.4new+}). Clearly, $f \in \widehat{GM}^{d}$. By \eqref{ThmSobEmbLebesgueProof1}, Theorem \ref{TheoremGMLip} and \eqref{HLGM}, we derive
	\begin{align*}
		 \left(\int_0^\infty t^{d r_0 - d -1} (g^*(t))^{r_0/p} \dint t \right)^{1/r_0} & \lesssim \left(\int_0^1 t^{d p-d - 1} g^*(t) \, \dint t\right)^{1/p} \\
		&\hspace{-3cm}+ \left(\int_1^\infty  (1 + \log t)^{- b q} \left(\int_1^t s^{\alpha p  + d p - d} g^*(s) \frac{\dint s}{s} \right)^{q/p} \frac{\dint t}{t}\right)^{1/q}.
	\end{align*}
	Hence, the following inequality holds
	 \begin{equation}\label{ThmSobEmbLebesgueProof2}
          	\left(\int_0^\infty w(t) (g^\ast(t))^{r_0/p} \dint t \right)^{p/r_0} \lesssim  \left(\int_0^\infty v(t) \left(\frac{1}{U(t)}\int_0^t u(s) g^\ast(s) \dint s \right)^{q/p} \dint t \right)^{p/q}
          \end{equation}
	with
	\begin{equation*}
          	w(t) = t^{d r_0 - d - 1},
          \end{equation*}
           \begin{equation*}
		v(t) = \left\{\begin{array}{lcl}
                         t^{d q - d q/p  -1}, & & 0 < t < 1,\\
                            & & \\
                         t^{\alpha q + d q -dq/p-1} (1 + \log t)^{-b q}, &   & t \geq 1,\\
            \end{array}
            \right.
            \end{equation*}
                         \begin{equation*}
		u(t) = \left\{\begin{array}{lcl}
                        t^{d p - d -1} , & & 0 < t < 1,\\
                            & & \\
                         t^{\alpha p + d p -d -1}, &   & t \geq 1,\\
            \end{array}
            \right.
            \end{equation*}
            and $U(t) = \int_0^t u(s) \dint s$.
	
To show that $r_0<r$, we distinguish two possible cases. Firstly, assume $q \leq r_0$ and so, $\frac{q}{p} \leq \frac{r_0}{p}$. Note that $\frac{r_0}{p} > 1$ because $p < r \leq r_0$. Hence, \eqref{ThmSobEmbLebesgueProof2} is a special case of converse Hardy-type inequality, which has been characterized in \cite[Theorem 4.2(i)]{GogatishviliPick}. Namely,
	
	  \begin{lem}\label{LemmaGP*}
            	Let $0< Q \leq R < \infty$ and $1 \leq R < \infty$. Let $u, v, w$ be non-negative measurable functions on $[0, \infty)$ and let $U(t) = \int_0^t u(s) \dint s, V(t) = \int_0^t v(s) \dint s, W(t) = \int_0^t w(s) \dint s$.
	 Assume that
	\begin{equation}\label{LemmaGP1*}
		\int_0^\infty u(t) \dint t = \infty,\quad \int_0^1 \frac{v(t)}{U^Q(t)} \dint t = \int_1^\infty v(t) \dint t = \infty,
	\end{equation}
	and
	\begin{equation}\label{LemmaGP2*}
		\int_0^\infty \frac{v(s)}{U^Q(s) + U^Q(t)} \dint s < \infty, \quad t > 0.
	\end{equation}
	Then, the inequality
	    \begin{equation*}
          	\left(\int_0^\infty w(t) (g^\ast(t))^R \dint t \right)^{1/R} \lesssim  \left(\int_0^\infty v(t) \left(\frac{1}{U(t)}\int_0^t u(s) g^\ast(s) \dint s \right)^{Q} \dint t \right)^{1/Q}
          \end{equation*}
            holds for all measurable functions $g$ if and only if
            \begin{equation}\label{LemmaGP3*}
            	\sup_{t > 0} \frac{W(t)^{1/R}}{\Big(V(t) + U(t)^Q \int_t^\infty U(s)^{-Q} v(s) \, \dint s \Big)^{1/Q}} < \infty.
            \end{equation}

            \end{lem}
	
	Let $Q = q/p$ and $R = r_0/p$. It is plain to check that \eqref{LemmaGP1*}  and \eqref{LemmaGP2*} hold. Hence, \eqref{ThmSobEmbLebesgueProof2} implies that \eqref{LemmaGP3*} holds. However, elementary estimates yield that
	\begin{equation*}
		\frac{W(t)^{1/R}}{\Big(V(t) + U(t)^Q \int_t^\infty U(s)^{-Q} v(s) \, \dint s \Big)^{1/Q}} \asymp t^{d p (1/r-1/r_0)} (1 + \log t)^{(b-1/q) p}, \quad t > 1,
	\end{equation*}
	which blows up as $t \to \infty$, because $r_0 \geq r$ and $b > 1/q$. Thus, we conclude that \eqref{ThmSobEmbLebesgueProof1} is not valid.
	
	Secondly, let $r_0 < q$ (and thus, $1 < \frac{r_0}{p} < \frac{q}{p}$). Note that the parameters involved in \eqref{ThmSobEmbLebesgueProof2} satisfy the assumptions of Lemma \ref{LemmaGP} with $R = r/p$ and $Q=q/p$. Accordingly, \eqref{LemmaGP3} holds. However, this is not true because
	\begin{align*}
		\int_1^\infty \frac{U(t)^P \sup_{y \geq t} U(y)^{-P} W(y)^{P/R}}{\Big(V(t) + U(t)^Q \int_t^\infty U(s)^{-Q} v(s) \dint s\Big)^{\frac{P}{Q} + 2}} V(t) \int_t^\infty U(s)^{-Q} v(s) \, \dint s  \, \dint(U^Q(t)) & \\
		& \hspace{-10cm} \asymp \int_1^\infty (1 + \log t)^{-1 + \frac{r_0 q}{q-r_0} \left(b-\frac{1}{q} \right)} \sup_{y \geq t} y^{d r_0 \left( \frac{1}{r} - \frac{1}{r_0}\right) \frac{q}{q-r_0}} \frac{\dint t}{t} \\
		& \hspace{-10cm} \gtrsim  \left\{\begin{array}{lcl}
                              \int_1^\infty (1 + \log t)^{-1 + \frac{r_0 q}{q-r_0} \left(b-\frac{1}{q} \right)} \frac{\dint t}{t} = \infty, & \text{if}  & r_0 = r, \\
                            & & \\
                           \sup_{t \geq 2}  t^{d r_0 \left( \frac{1}{r} - \frac{1}{r_0}\right) \frac{q}{q-r_0}} = \infty, & \text{if} & r_0 > r.
            \end{array}
            \right.
	\end{align*}
	This leads us to a contradiction.
	
	Finally, we show the sharpness assertion $p \leq r_0$. Suppose, on the contrary, that there exists $p > r_0$ such that
		\begin{equation}\label{ThmSobEmbLebesgueProof3}
	\L^{(\alpha,-b)}_{p,q}(\R^d) \hookrightarrow  L_{r_0}(\R^d).
	\end{equation}
	Let
\begin{equation*}
		 F_0(t) = \left\{\begin{array}{lcl}
                            t^{-\varepsilon} & ,  & 0 < t  < 1, \\
                            & & \\
                            t^{-\alpha + d/p -d} (1 + \log t)^{-\beta} & , & t \geq 1,
            \end{array}
            \right.
	\end{equation*}
	where $-b + 1/p +1/q < \beta < 1/p$ and $d (1-1/r_0) < \varepsilon  < d (1-1/p)$ and define $\widehat{f}(\xi) = F_0(|\xi|)$. By Theorem \ref{TheoremGMLip}, we have
	\begin{align*}
		\|f\|_{\L^{(\alpha,-b)}_{p,q}(\mathbb{R}^d)} & \asymp \left(\int_0^1 t^{d p-d -\varepsilon p} \frac{\dint t}{t}\right)^{1/p} \\
		&\hspace{1cm}+ \left(\int_1^\infty  (1 + \log t)^{- b q} \left(\int_1^t (1 + \log u)^{-\beta p} \frac{\dint u}{u} \right)^{q/p} \frac{\dint t}{t}\right)^{1/q}\\
		&  \asymp \left(\int_0^1 t^{d p-d -\varepsilon p} \frac{\dint t}{t}\right)^{1/p} + \left(\int_1^\infty (1 + \log t)^{-b q -\beta q + q/p} \frac{\dint t}{t} \right)^{1/q} < \infty,
	\end{align*}	
	but, by \eqref{HLGM},
	\begin{equation*}
		\|f\|_{L_{r_0}(\R^d)} \gtrsim \left(\int_0^1 t^{d r_0 -d-\varepsilon r_0} \frac{\dint t}{t} \right)^{1/r_0} = \infty.
	\end{equation*}
	  This contradicts \eqref{ThmSobEmbLebesgueProof3}.
\end{proof}

\subsection{Embeddings into Lorentz-Zygmund spaces}\label{lz}

In the next two sections we  work with periodic functions on $\T^d$.

In Theorem \ref{ThmSobEmbLebesgue} we have shown that Lebesgue norms do not provide full information on the integrability properties of Lipschitz functions, that is, if $0 < \alpha < d/p$ and $\alpha -d/p = -d/r$ then
\begin{equation*}
	\L^{(\alpha,-b)}_{p,q}(\R^d) \not \hookrightarrow L_{r}(\R^d),
\end{equation*}
but only
\begin{equation}\label{EmbLZIntro1}
	\L^{(\alpha,-b)}_{p,q}(\R^d) \hookrightarrow L_{r_0}(\R^d), \quad p \leq r_0 < r.
\end{equation}
We overcome this obstruction considering  the broader scale of Lorentz-Zygmund spaces, which allow us  to get sharp embedding results with $r = r_0$ in \eqref{EmbLZIntro1} but involving additional integrability properties with respect to the logarithmic scale. We start by recalling the definition of Lorentz-Zygmund spaces.

For $1 \leq p < \infty, 0 <q \leq \infty$ and $-\infty < b < \infty$, the \emph{Lorentz-Zygmund space} $L_{p,q}(\log L)_b(\T^d)$ is formed by all measurable functions $f$ on $\T^d$ such that
\begin{equation}\label{LZspace}
	\|f\|_{L_{p,q}(\log L)_b(\T^d)} = \Big(\int_0^1 (t^{1/p} (1 + |\log t|)^b f^*(t))^q \frac{\dint t}{t} \Big)^{1/q} < \infty
\end{equation}
(with the usual modification if $q=\infty$). As usual, $f^*$ denotes the non-increasing rearrangement of $f$. Note that if $p=q$ then we obtain the \emph{Zygmund space} $L_p(\log L)_b(\T^d)$ and if, in addition, $b=0$ then we recover the Lebesgue space $L_p(\T^d)$. Setting $b=0$ and $p \neq q$ in $L_{p,q}(\log L)_b(\T^d)$ we get the \emph{Lorentz space} $L_{p,q}(\T^d)$. For further information on Lorentz-Zygmund spaces, the reader may consult \cite{BennettRudnick} and \cite{EdmundsEvans}.

Embeddings of Besov spaces into Lorentz and Lorentz-Zygmund spaces have been extensively studied in the literature. We limit ourselves to quoting \cite{CaetanoMoura, CobosDominguez, DeVoreRiemenschneiderSharpley, GurkaOpic, HaroskeSchneider, Martin, Netrusov, Triebel01}. In particular, if $1 \leq p < \infty, 0 < s < d/p, s-d/p = -d/r, 0 < q \leq \infty$, and $-\infty < b < \infty$, then
\begin{equation}\label{EmbBesLZ}
	B^{s, b}_{p,q}(\T^d) \hookrightarrow L_{r,q}(\log L)_b(\T^d)
\end{equation}
and its counterpart for the entire Euclidean space $\R^d$ also holds true. Furthermore, this embedding is optimal within the scale of r.i. spaces. In particular,
\begin{equation}\label{EmbBesLZSharp}
	B^{s, b}_{p,q}(\T^d) \hookrightarrow L_{r,q}(\log L)_\xi(\T^d) \iff \xi \leq b.
\end{equation}

Now we are ready to establish the  analogue of \eqref{EmbBesLZ} for Lipschitz spaces.

\begin{thm}\label{ThmEmbLipLZ}
Let $1 < p < \infty, 0 < \alpha < d/p, \alpha -d/p = -d/r, 0 < q \leq \infty$, and $b > 1/q$. Then,
\begin{equation}\label{ThmEmbLipLZ1}
	\emph{Lip}^{(\alpha,-b)}_{p,q}(\T^d) \hookrightarrow L_{r,q}(\log L)_{-b + 1/\max\{p,q\}}(\T^d) \cap L_{r, \max\{p,q\}}(\log L)_{-b + 1/q}(\T^d).
\end{equation}
The corresponding result for function spaces on $\R^d$ also holds true.
\end{thm}

\begin{rem}
	Note that the target space in \eqref{ThmEmbLipLZ1} is $L_{r,q}(\log L)_{-b+1/q}(\T^d)$ if $q \geq p$. However, for $q < p$ the spaces $L_{r,q}(\log L)_{-b+1/p}(\T^d)$ and $L_{r,p}(\log L)_{-b+1/q}(\T^d)$ are not comparable. This easily follows by taking functions $f$ and $f_N$ such that $f^*(t) = t^{-1/r} (1 + |\log t|)^{b-1/p-1/q} (1 + \log (1 + |\log t|))^{-\varepsilon}, \, 1/p < \varepsilon < 1/q$, and $f_N^\ast(t) = \chi_{(0,1/N)}(t), \, N \in \N$.
\end{rem}

\begin{proof}[Proof of Theorem \ref{ThmEmbLipLZ}]
	Let $p < p_1 < \infty$. In virtue of Theorem \ref{ThMFrankeLip}, we have
	\begin{equation}\label{ProofThmEmbLipLZ1}
		\L^{(\alpha,-b)}_{p,q}(\T^d) \hookrightarrow B^{\alpha + d(1/p_1 -1/p), -b + 1/\max\{p,q\}}_{p_1,q}(\T^d) \cap B^{\alpha + d(1/p_1 -1/p), -b + 1/q}_{p_1,\max\{p,q\}}(\T^d).
	\end{equation}
	On the other hand, by \eqref{EmbBesLZ},
	\begin{equation}\label{ProofThmEmbLipLZ2}
		B^{\alpha + d(1/p_1 -1/p), -b + 1/\max\{p,q\}}_{p_1,q}(\T^d)  \hookrightarrow L_{r,q}(\log L)_{-b+1/\max\{p,q\}}(\T^d)
	\end{equation}
	and
	\begin{equation}\label{ProofThmEmbLipLZ3}
		 B^{\alpha + d(1/p_1 -1/p), -b + 1/q}_{p_1,\max\{p,q\}}(\T^d) \hookrightarrow L_{r,\max\{p,q\}}(\log L)_{-b+1/q}(\T^d).
	\end{equation}
	The desired embedding \eqref{ThmEmbLipLZ1} follows from \eqref{ProofThmEmbLipLZ1}--\eqref{ProofThmEmbLipLZ3}.
	
	For function spaces on $\R^d$, the proof follows line by line the argument given above.
\end{proof}

Next we prove the sharpness of Theorem \ref{ThmEmbLipLZ}.

\begin{thm}\label{ThmEmbLipLZSharp}
	Let $1 < p < \infty, 0 < \alpha < 1/p, \alpha -1/p = -1/r, 0 < q, u \leq \infty, b > 1/q$ and $-\infty < \xi < \infty$. Then,
	\begin{equation}\label{ThmEmbLipLZSharp1}
		\emph{Lip}^{(\alpha,-b)}_{p,q}(\T) \hookrightarrow L_{r,q}(\log L)_{-b + \xi}(\T)  \iff \xi \leq 1/\max\{p,q\},
		\end{equation}
		\begin{equation}\label{ThmEmbLipLZSharp2}
		\emph{Lip}^{(\alpha,-b)}_{p,q}(\T) \hookrightarrow L_{r, u}(\log L)_{-b + 1/q}(\T) \iff u \geq \max\{p,q\}.
	\end{equation}
\end{thm}
\begin{proof}
Assume that
	\begin{equation}\label{ProofThmEmbLipLZSharp1}
		\L^{(\alpha,-b)}_{p,q}(\T) \hookrightarrow L_{r,q}(\log L)_{-b + \xi}(\T).
		\end{equation}
	Consequently, given any $\theta \in (0,1)$, we have
	\begin{equation}\label{ProofThmEmbLipLZSharp2}
		(L_p(\T),\L^{(\alpha,-b)}_{p,q}(\T))_{\theta, q} \hookrightarrow (L_p(\T), L_{r,q}(\log L)_{-b + \xi}(\T))_{\theta,q}.
	\end{equation}
	Using the well-known interpolation properties of Lorentz-Zygmund spaces (see, e.g., \cite[Lemma 5.5]{GogatishviliOpicTrebels})
	\begin{equation*}
		(L_p(\T), L_{r,q}(\log L)_{-b + \xi}(\T))_{\theta,q} = L_{v,q}(\log L)_{\theta(-b+\xi)}(\T), \quad \frac{1}{v} = \frac{1-\theta}{p} + \frac{\theta}{r},
	\end{equation*}
	and the periodic counterpart of \eqref{ThmLipBesSharpHil3} (noting that the argument there can also be applied in the periodic setting), that is,
	\begin{equation*}
		(L_p(\T),\L^{(\alpha,-b)}_{p,q}(\T))_{\theta, q}  = B^{\theta \alpha, \theta (-b + 1/q)}_{p,q}(\T),
	\end{equation*}
	we can rewrite \eqref{ProofThmEmbLipLZSharp2} as
	\begin{equation*}
		B^{\theta \alpha, \theta (-b + 1/q)}_{p,q}(\T) \hookrightarrow L_{v,q}(\log L)_{\theta(-b+\xi)}(\T) \quad \text{with} \quad \theta \alpha -\frac{1}{p} = - \frac{1}{v}.
	\end{equation*}
	According to \eqref{EmbBesLZSharp}, we arrive at $\xi \leq 1/q$. This shows \eqref{ThmEmbLipLZSharp1} if $q \geq p$.
	
	Suppose now that $q < p$ and there is $\xi > 1/p$ such that \eqref{ProofThmEmbLipLZSharp1} holds. Define the function
	\begin{equation*}
		f(x) = \sum_{n=1}^\infty n^{-\alpha -1 + 1/p} (1 + \log n)^{-\varepsilon} \cos n x, \quad -b + \frac{1}{p} + \frac{1}{q} < \varepsilon< -b + \xi + \frac{1}{q} .
	\end{equation*}
By Theorem \ref{TheoremGMLipPer},
\begin{align}
	\|f\|_{\L^{(\alpha,-b)}_{p,q}(\T)}^q & \asymp \sum_{n=1}^\infty (1 + \log n)^{-b q} \left( \sum_{k=1}^n (1 + \log k)^{-\varepsilon p} \frac{1}{k}\right)^{q/p} \frac{1}{n} \nonumber \\
	&\asymp \sum_{n=1}^\infty (1 + \log n)^{-b q -\varepsilon q + q/p} \frac{1}{n} < \infty. \label{ProofThmEmbLipLZSharp3}
\end{align}
However, $f \not \in L_{r,q}(\log L)_{-b + \xi}(\T)$. To show this we will make use of the following result, which is a special case of the Hardy-Littlewood theorem for Fourier series with general monotone coefficients in Lorentz-type spaces (see \cite{Simonov, Tikhonov, Volo}).

\begin{lem}[\bf{Hardy-Littlewood theorem for Fourier series with monotone coefficients in Lorentz-Zygmund spaces}]\label{LemmaHLLorentz}
	Let $1 < r < \infty, 0 < q < \infty$, and $-\infty < b < \infty$. Let $\{a_n\}_{n \in \N}$ be a non-negative general monotone sequence. If $f(x) \sim \sum_{n=1}^\infty a_n \cos nx$ then
	\begin{equation*}
		\|f\|_{L_{r,q}(\log L)_b(\T)} \asymp \left(\sum_{n=1}^\infty n^{q-q/r-1} (1 + \log n)^{b q} a_n^q \right)^{1/q}.
	\end{equation*}
	The same result also holds true for $f(x) \sim \sum_{n=1}^\infty a_n \sin nx$.
\end{lem}

It follows from Lemma \ref{LemmaHLLorentz} that
\begin{equation*}
	\|f\|_{L_{r,q}(\log L)_{-b + \xi}(\T)}^q \asymp \sum_{n=1}^\infty (1 + \log n)^{(-b + \xi - \varepsilon) q}  \frac{1}{n} = \infty,
\end{equation*}
which is not possible. Therefore, $\xi \leq 1/p$. The proof of \eqref{ThmEmbLipLZSharp1} is complete.

Next we show \eqref{ThmEmbLipLZSharp2} by contradiction. Suppose first that $p \geq q$ and there exists $u < p$ such that
\begin{equation}\label{ProofThmEmbLipLZSharp4}
	\L^{(\alpha,-b)}_{p,q}(\T) \hookrightarrow L_{r, u}(\log L)_{-b+1/q}(\T).
\end{equation}
Let
	\begin{equation*}
		f(x) = \sum_{n=1}^\infty n^{-\alpha -1 + 1/p} (1 + \log n)^{-\varepsilon} \cos n x, \quad -b + \frac{1}{p} + \frac{1}{q} < \varepsilon< \min\Big\{\frac{1}{p},-b + \frac{1}{u} + \frac{1}{q}\Big\} .
	\end{equation*}
By \eqref{ProofThmEmbLipLZSharp3}, $f \in \L^{(\alpha,-b)}_{p,q}(\T)$ but, in virtue of Lemma \ref{LemmaHLLorentz}, $f \not \in L_{r, u}(\log L)_{-b+1/q}(\T)$ since
\begin{equation*}
	\|f\|_{ L_{r, u}(\log L)_{-b+1/q}(\T)}^u  \asymp \sum_{n=1}^\infty (1 + \log n)^{(-b + 1/q -\varepsilon) u } \frac{1}{n} = \infty.
\end{equation*}

Assume now that $p < q$ and \eqref{ProofThmEmbLipLZSharp4} holds true for some $u < q$. By monotonicity of Lorentz-Zygmund norms with respect to the second index (i.e., $L_{r, u_0}(\log L)_{-b+1/q}(\T) \hookrightarrow L_{r, u_1}(\log L)_{-b+1/q}(\T), \, u_0 < u_1$), we may assume that $p < u < q$. Define
    \begin{equation*}
          	w(t) = \left\{\begin{array}{lcl}
                        0, & & 0 < t < 1,\\
                            & & \\
                         t^{\alpha u +  u - u/p-1} (1 + \log t)^{(-b + 1/q) u},&   & t \geq 1,\\
            \end{array}
            \right.
          \end{equation*}
           \begin{equation*}
		v(t) = \left\{\begin{array}{lcl}
                         t^{ q -  q/p  -1}, & & 0 < t < 1,\\
                            & & \\
                         t^{\alpha q + q -q/p-1} (1 + \log t)^{-b q}, &   & t \geq 1,\\
            \end{array}
            \right.
            \end{equation*}
                         \begin{equation*}
		u(t) = \left\{\begin{array}{lcl}
                        t^{p - 2} , & & 0 < t < 1,\\
                            & & \\
                         t^{\alpha p + p -2}, &   & t \geq 1,\\
            \end{array}
            \right.
            \end{equation*}
            and $U(t) = \int_0^t u(s) \dint s$. Let $g$ be any measurable function on $\R^d$ such that
            \begin{equation*}
            \int_0^\infty v(t) \left(\frac{1}{U(t)}\int_0^t u(s) g^\ast(s) \dint s \right)^{q/p} \dint t  < \infty.
            \end{equation*}
            In particular, we have
            \begin{equation}\label{ProofThmEmbLipLZSharp5}
            	\sum_{n=1}^\infty (1 + \log n)^{-b q} \Big(\sum_{k=1}^n k^{\alpha p + p -2} g^*(k) \Big)^{q/p} \frac{1}{n} < \infty
            \end{equation}
	and $g^*(n) \to 0$. We let
\begin{equation*}
	f(x) = \sum_{n=1}^\infty (g^*(n))^{1/p} \cos nx
\end{equation*}
and it follows from Theorem \ref{TheoremGMLipPer} and \eqref{ProofThmEmbLipLZSharp5} that $f \in \L^{(\alpha,-b)}_{p,q}(\T)$. By assumption, $f \in L_{r,u}(\log L)_{-b+1/q}(T)$ and applying Lemma \ref{LemmaHLLorentz}, we get
\begin{align*}
	\left(\int_0^\infty w(t) (g^\ast(t))^{u/p} \dint t \right)^{p/u}  &\asymp \Big(\sum_{n=1}^\infty n^{ u - u/r-1} (1 + \log n)^{(-b + 1/q) u} (g^*(n))^{u/p} \Big)^{p/u} \\
	&  \lesssim 	\Big(\sum_{n=1}^\infty (1 + \log n)^{-b q} \Big(\sum_{k=1}^n k^{\alpha p + p -2} g^*(k) \Big)^{q/p} \frac{1}{n} \Big)^{p/q} \\
	& \lesssim  \Big(\int_0^\infty v(t) \left(\frac{1}{U(t)}\int_0^t u(s) g^\ast(s) \dint s \right)^{q/p} \dint t \Big)^{p/q}
\end{align*}
where we have used $\alpha -1/p = -1/r$. Therefore, invoking Lemma \ref{LemmaGP} (with $R = u/p$ and $Q=q/p$) we arrive at the desired contradiction since straightforward  computations show that \eqref{LemmaGP3} is not fulfilled.
\end{proof}

\subsection{Embeddings into rearrangement invariant spaces}\label{ri}

In the previous section we have investigated in detail embeddings of Lipschitz spaces into Lorentz-Zygmund spaces. Recall that we have shown in Theorem \ref{ThmEmbLipLZ} that if $1 < p < \infty, 0 < \alpha < d/p, \alpha -d/p = -d/r, 0 < q \leq \infty$ and $b > 1/q$ then
\begin{equation}\label{ThmEmbLipLZ1new}
	\L^{(\alpha,-b)}_{p,q}(\T^d) \hookrightarrow L_{r,q}(\log L)_{-b + 1/\max\{p,q\}}(\T^d) \cap L_{r, \max\{p,q\}}(\log L)_{-b + 1/q}(\T^d).
\end{equation}
Furthermore, these embeddings are optimal within the scale of Lorentz-Zygmund spaces (cf. Theorem \ref{ThmEmbLipLZSharp}). However, we will see that they can be improved working with the more general framework of rearrangement invariant spaces. Before proceeding further, let us introduce some notation.

For $1 \leq p < \infty, 0 < q, r \leq \infty$ and $-\infty < b < \infty$, the space $L^{(\mathcal{R})}_{p,r,b,q}(\T^d)$ is formed by all measurable functions $f$ on $\T^d$ for which
\begin{equation}\label{DefGrand}
	\|f\|_{L^{(\mathcal{R})}_{p,r,b,q}(\T^d)} =\left(\int_0^1 (1 - \log t)^{b r} \left(\int_t^1 (u^{1/p} f^*(u))^q \frac{\dint u}{u} \right)^{r/q} \frac{\dint t}{t} \right)^{1/r} < \infty
\end{equation}
(with appropriate modifications if $q=\infty$ and/or $r=\infty$). The spaces $L^{(\mathcal{R})}_{p,r,b,q}(\T^d)$ with $q=1$ were already investigated by Bennett and Rudnick \cite{BennettRudnick} and Pustylnik \cite{Pustylnik} to establish optimal interpolation theorems. Moreover, the spaces $L^{(\mathcal{R})}_{p,r,b,q}(\T^d)$ arise naturally in limiting interpolation/extrapolation theory \cite{FiorenzaKaradzhov, EvansOpic} and they are associate spaces of weighted Lorentz spaces \cite{FiorenzaRakotosonZitouni, GogatishviliPickSoudsky}. Note that if $1 < p < \infty$ then one can replace (up to equivalence) $f^\ast$ in \eqref{DefGrand} with its maximal function $f^{**}(t) = \frac{1}{t} \int_0^t f^*(u) \dint u$. To avoid trivial spaces, we  assume that $b <- 1/r \, (b \leq 0 \text{ if } r=\infty)$.

In this section we show that the spaces $L^{(\mathcal{R})}_{p,r,b,q}(\T^d)$ provide the natural framework to measure the growth properties of Lipschitz functions.

\begin{thm}\label{ThmEmbLipRI}
Let $1 < p < \infty, 0 < \alpha < d/p, \alpha -d/p = -d/r, 0 < q \leq \infty,$ and $b > 1/q$. Then,
\begin{equation}\label{ThmEmbLipRI1}
	\emph{Lip}^{(\alpha,-b)}_{p,q}(\T^d) \hookrightarrow L^{(\mathcal{R})}_{r,q,-b,p}(\T^d).
\end{equation}	
\end{thm}

\begin{rem}
	The embedding \eqref{ThmEmbLipRI1} sharpens \eqref{ThmEmbLipLZ1new} because
	\begin{equation}\label{RemEmbLipRI1}
		L^{(\mathcal{R})}_{r,q,-b,p}(\T^d) \hookrightarrow L_{r,q}(\log L)_{-b + 1/\max\{p,q\}}(\T^d) \cap L_{r, \max\{p,q\}}(\log L)_{-b + 1/q}(\T^d).
	\end{equation}
	Furthermore, \eqref{ThmEmbLipRI1} and \eqref{ThmEmbLipLZ1new} coincide if $q=p$, that is,
	\begin{equation}\label{RemEmbLipRI2}
		L^{(\mathcal{R})}_{r,p,-b,p}(\T^d) = L_{r,p}(\log L)_{-b+1/p}(\T^d),
	\end{equation}
	but this is not true in general since
	\begin{equation}\label{RemEmbLipRI3}
		L^{(\mathcal{R})}_{r,q,-b,p}(\T^d) \neq L_{r,q}(\log L)_{-b + 1/q}(\T^d), \quad q > p.
	\end{equation}
	
	We start by showing \eqref{RemEmbLipRI1}. Assume first that $q \geq p$. Then, by monotonicity properties and changing the order of summation, we have
	\begin{align*}
		\|f\|_{ L_{r,q}(\log L)_{-b + 1/q}(\T^d)} & \asymp \left(\sum_{k=0}^\infty 2^{-k q/r} (1 + k)^{(-b+1/q) q} (f^*(2^{-k}))^q \right)^{1/q} \\
		& \asymp \left(\sum_{k=0}^\infty 2^{-k q / r} (f^*(2^{-k}))^q \sum_{n=k}^\infty (1 + n)^{-b q} \right)^{1/q} \\
		& \asymp \left(\sum_{n=0}^\infty (1 + n)^{-b q} \sum_{k=0}^n 2^{-k q / r} (f^*(2^{-k}))^q  \right)^{1/q} \\
		& \lesssim \left(\sum_{n=0}^\infty (1 + n)^{-b q} \left(\sum_{k=0}^n 2^{-k p/r} (f^*(2^{-k}))^p \right)^{q/p} \right)^{1/q} \\
		& \asymp \|f\|_{L^{(\mathcal{R})}_{r,q,-b,p}(\T^d)}.
	\end{align*}
	
	Next we show \eqref{RemEmbLipRI1} with $q < p$. Let $\beta$ be such that $-b + 1/p < \beta < -1/q + 1/p$. Applying H\"older's inequality
	\begin{equation*}
		\int_t^1 (u^{1/r} (1-\log u)^\beta f^*(u))^q \frac{\dint u}{u} \lesssim (1 - \log t)^{\beta q + 1 -q/p} \left(\int_t^1 (u^{1/r} f^*(u))^p \frac{\dint u}{u} \right)^{q/p}
	\end{equation*}
	and thus
	\begin{align*}
		\|f\|_{L_{r,q} (\log L)_{-b + 1/p}(\T^d)}^q &\asymp \int_0^1 (u^{1/r} (1-\log u)^\beta f^*(u))^q \int_0^s (1-\log t)^{-b q - \beta q - 1 + q/p} \frac{\dint t}{t} \frac{\dint u}{u} \\
		& \hspace{-2cm} \lesssim \int_0^1 (1-\log t)^{-b q} \left(\int_t^1 (u^{1/r} f^*(u))^p \frac{\dint u}{u} \right)^{q/p} \frac{\dint t}{t} =  \|f\|_{L^{(\mathcal{R})}_{r,q,-b,p}(\T^d)}^q.
	\end{align*}
	On the other hand, we can apply Minkowski's inequality to get
	\begin{align*}
		\|f\|_{L_{r,p}(\log L)_{-b + 1/q}(\T^d)} & \asymp \left(\int_0^1 t^{p/r} (f^*(t))^p \left(\int_0^t (1-\log u)^{-b q} \frac{\dint u}{u} \right)^{p/q} \frac{\dint t}{t} \right)^{1/p} \\
		& \hspace{-2cm}  \leq \left(\int_0^1 (1-\log u)^{-b q} \left(\int_u^1 (t^{1/r} f^*(t))^p \frac{\dint t}{t} \right)^{q/p} \frac{\dint u}{u} \right)^{1/q} = \|f\|_{L^{(\mathcal{R})}_{r,q,-b,p}(\T^d)}.
	\end{align*}
	
	The formula \eqref{RemEmbLipRI2} is an immediate consequence of Fubini's theorem. To show \eqref{RemEmbLipRI3}, consider $f$ such that $f^\ast(t) = t^{-1/r} (1 - \log t)^{-\varepsilon}$ where $-b + 2/q < \varepsilon < -b + 1/q + 1/p$.
\end{rem}

The proof of Theorem \ref{ThmEmbLipRI} will rely on the following characterization of $L^{(\mathcal{R})}_{r,q,-b,p}(\T^d)$ as limiting interpolation spaces (see \cite[(8.27)]{EvansOpic}).

\begin{lem}\label{LemmaEmbLipRI}
	Let $1 \leq u < r < \infty, 0 < p, q \leq \infty$ and $b > 1/q \, (b \geq 0 \text{ if } q=\infty)$. Then,
	\begin{equation*}
		(L_u(\T^d), L_{r,p}(\T^d))_{(1,-b), q} = L^{(\mathcal{R})}_{r,q,-b,p}(\T^d).
	\end{equation*}
\end{lem}

We are now ready to give the

\begin{proof}[Proof of Theorem \ref{ThmEmbLipRI}]
	In light of the Sobolev embedding theorem (see, e.g., \cite[p. 120]{Haroske}),
	\begin{equation*}
		H^\alpha_p(\T^d) \hookrightarrow L_{r,p}(\T^d),
	\end{equation*}
	we derive
	\begin{equation}\label{ProofThmEmbLipRI1}
		(L_p(\T^d), H^\alpha_p(\T^d))_{(1,-b),q} \hookrightarrow (L_p(\T^d), L_{r,p}(\T^d))_{(1,-b),q}.
	\end{equation}
	By \eqref{LipLimInter}, the space in the left-hand side  is equal to 
$
\L^{(\alpha,-b)}_{p,q}(\T^d)$, while the right-hand side space coincides with $L^{(\mathcal{R})}_{r,q,-b,p}(\T^d)$ (cf. Lemma \ref{LemmaEmbLipRI}). Therefore, \eqref{ProofThmEmbLipRI1} can be rewritten as
	 \begin{equation*}
	 	\L^{(\alpha,-b)}_{p,q}(\T^d) \hookrightarrow  L^{(\mathcal{R})}_{r,q,-b,p}(\T^d).
	 \end{equation*}
\end{proof}

Our next result gives the sharpness of Theorem \ref{ThmEmbLipRI}.

\begin{thm}
	Let $1 < p < \infty, 0 < \alpha < 1/p, \alpha - 1/p = -1/r, 0 < q \leq \infty$ and $b > 1/q$. Then,
	\begin{equation*}
		\|f\|_{\emph{Lip}^{(\alpha,-b)}_{p,q}(\T)} \asymp \|f\|_{L^{(\mathcal{R})}_{r,q,-b,p}(\T)}
	\end{equation*}
	for all $f(x) \sim \sum_{n=1}^\infty( a_n \cos nx + b_n \sin nx)$ where $\{a_n\}_{n \in \mathbb{N}}, \{b_n\}_{n \in \mathbb{N}}$ are  non-negative general monotone sequences.
\end{thm}

\begin{proof}
	According to Theorem \ref{TheoremGMLipPer}, we have
		\begin{equation*}
		\|f\|_{\L^{(\alpha,-b)}_{p,q}(\T)}  \asymp
		 \left(\sum_{n=1}^\infty  (1 + \log n)^{- b q} \left(\sum_{k=1}^n k^{p - p/r -1} (a_k^p + b_k^p) \right)^{q/p} \frac{1}{n}\right)^{1/q}.
	\end{equation*}

	The proof is complete by applying the following

\begin{lem}[\bf{Hardy-Littlewood theorem for Fourier series with monotone coefficients in $L^{(\mathcal{R})}_{p,r,b,q}(\T)$}]\label{LemmaHLGrand}
	Let $1 < r < \infty, 0 < p, q \leq \infty$, and $ b > 1/q \, (b \geq 0 \text{ if } q=\infty)$. Let the Fourier series of
    $f \in L_1(\mathbb{T})$ be given by
	 \begin{equation*}
	 f(x) \sim \sum_{n=1}^\infty (a_n \cos n x + b_n \sin nx)
	 \end{equation*}
	where $\{a_n\}_{n \in \N}, \{b_n\}_{n \in \N}$ are non-negative general monotone sequences. Then
	\begin{equation*}
		\|f\|_{L^{(\mathcal{R})}_{r,q,-b,p}(\T)} \asymp \left(\sum_{n=1}^\infty (1 + \log n)^{-b q} \left(\sum_{k=1}^n k^{p-p/r-1} (a_k^p+ b_k^p) \right)^{q/p} \frac{1}{n} \right)^{1/q}.
	\end{equation*}
\end{lem}

The proof of Lemma \ref{LemmaHLGrand} follows along the lines of the proof of the corresponding result for Lebesgue norms (see \eqref{HLGMPer}) given in \cite{Tikhonov}, so that we omit further details.
\end{proof}

\end{document}